\documentclass[11pt,reqno]{amsart} 
\usepackage[a4paper, margin=3cm]{geometry}

\usepackage{graphicx}
\graphicspath{ {./images/} }
\usepackage{lipsum}
\usepackage{subcaption}

\usepackage{xcolor}

\usepackage{amsthm}

\DeclareMathOperator{\arccot}{arccot}

\usepackage{bbm}
\usepackage{mathrsfs}
\usepackage{dynkin-diagrams}

\usepackage[utf8]{inputenc}
\usepackage{xypic}

\usepackage{tikz-cd}

\usepackage{enumitem}

\usepackage{amssymb}
\usepackage{float}
\usepackage{amsbsy}
\usepackage[all]{xy}\usepackage{xypic}
\usepackage{braket}
\usepackage[colorlinks = true,citecolor=black]{hyperref}
\usepackage{amsmath}

\newcommand\myeq{\mathrel{\stackrel{\makebox[0pt]{\mbox{\normalfont\tiny loc}}}{=}}}

\hypersetup{
     colorlinks=true,
     linkcolor=blue,
     filecolor=blue,
     citecolor = blue,      
     urlcolor=cyan,
     }

\makeatletter
\newcommand\opteq[1]{\mathrel{\mathpalette\opt@eq{#1}}}
\newcommand{\opt@eq}[2]{%
  \begingroup
  \sbox\z@{$#1#2$}%
  \sbox\tw@{\resizebox{!}{.5\ht\z@}{$\m@th#1($}}%
  \nonscript\hskip-\wd\tw@
  \mkern1mu
  \raisebox{-.35\ht\z@}[0pt][0pt]{\resizebox{!}{.5\ht\z@}{$\m@th#1($}}%
  \mkern-1mu
  {#2}%
  \mkern-1mu
  \raisebox{-.35\ht\z@}[0pt][0pt]{\resizebox{!}{.5\ht\z@}{$\m@th#1)$}}%
  \mkern1mu
  \nonscript\hskip-\wd\tw@
  \endgroup
}
\makeatother

\newcommand{\geoq}{\opteq{\geq}}

\newtheorem{theorem}{Theorem}[section]
\newtheorem*{theorem*}{Theorem}
\newtheorem{theorem-non}{Theorem}
\newtheorem{lemma-non}{Lemma}

\theoremstyle{definition} 
\newtheorem{thm}{Theorem}

\theoremstyle{definition} 
\newtheorem{corollarynon}{Corollary}

\newtheorem{conjecture-non}{Conjecture}

\newtheorem{corollary-non}{Corollary}
\newtheorem{proposition}[theorem]{Proposition}
\newtheorem{lemma}[theorem]{Lemma}
\newtheorem*{lemma*}{Lemma}
\newtheorem{corollary}[theorem]{Corollary}
\newtheorem{conjecture}[theorem]{Conjecture}
\newtheorem*{conjecture*}{Conjecture}

\newtheorem{problem}[theorem]{Problem}
\theoremstyle{definition}
\newtheorem{definition}[theorem]{Definition}
\newtheorem{example}[theorem]{Example}

\theoremstyle{remark}
\newtheorem{remark}[theorem]{Remark}

\DeclareMathOperator{\rank}{rank}

\numberwithin{equation}{section}

\begin{document}
\title[DHYM equation on rational homogeneous varieties]{Deformed Hermitian Yang-Mills equation on rational homogeneous varieties}

\author{Eder M. Correa}


\address{{IMECC-Unicamp, Departamento de Matem\'{a}tica. Rua S\'{e}rgio Buarque de Holanda 651, Cidade Universit\'{a}ria Zeferino Vaz. 13083-859, Campinas-SP, Brazil}}
\address{E-mail: {\rm ederc@unicamp.br}}

\begin{abstract} 
In this paper, we show that the deformed Hermitian Yang-Mills (dHYM) equation on a rational homogeneous variety, equipped with any invariant K\"{a}hler metric, always admits a solution. In particular, we describe the Lagrangian phase, with respect to any invariant K\"{a}hler metric, of every closed invariant $(1,1)$-form in terms of Lie theory. Building on this, we characterize all supercritical and hypercritical homogeneous solutions of the dHYM equation using the Cartan matrix associated with the underlying complex simple Lie algebra. Further, we provide an explicit formula, in terms of Lie theory, for the slope of holomorphic vector bundles over rational homogeneous varieties. Using this formula, we derive a new criterion for slope semistability through restrictions of holomorphic vector bundles to the generators of the associated cone of curves. Moreover, we provide a new characterization, in terms of central charges defined by rational curves, for slope (semi)stability of holomorphic vector bundles over rational homogeneous varieties. As an application of our main results, we describe all supercritical and hypercritical homogeneous solutions of the dHYM equation on the Fano threefold defined by the Wallach flag manifold ${\mathbbm{P}}(T_{{\mathbbm{P}^{2}}})$. In addition, we introduce a constructive method to obtain non-trivial examples of Hermitian-Einstein metrics on certain holomorphic vector bundles over ${\mathbbm{P}}(T_{{\mathbbm{P}^{2}}})$ from solutions of linear diophantine equations. In this last case, we describe explicitly all associated Hermitian Yang-Mills instantons. We also present some new insights that explore the interplay between intersection theory and number theory.
\end{abstract}

\maketitle

\hypersetup{linkcolor=black}

\hypersetup{linkcolor=black}
\section{Introduction}

Let $(X,\omega)$ be a compact connected K\"{a}hler manifold, such that $\dim_{\mathbbm{C}}(X) = n$, and $[\psi] \in H^{1,1}(X,\mathbbm{R})$. Motivated by mirror symmetry, in \cite{JacobYau2017} the following problem was introduced:
\begin{problem}
\label{problem1}
Does there exist a smooth $(1,1)$-form $\chi \in [\psi]$, such that 
\begin{equation}
\label{dHYMeq}
{\rm{Im}}\big ( \omega + \sqrt{-1}\chi\big)^{n} = \tan(\hat{\Theta}) {\rm{Re}}\big ( \omega + \sqrt{-1}\chi\big)^{n},
\end{equation}
where $\hat{\Theta}$ is a $S^{1}$-valued topological constant (``phase angle") depending on $[\omega],[\psi]$?
\end{problem}
The equation described above (Eq. (\ref{dHYMeq})) is known as the {\textit{deformed Hermitian-Yang-Mills (dHYM) equation}}. According to \cite{lin2022deformed}, it was discovered around the same time in \cite{marino2000nonlinear} and \cite{leung2000special} using different points of view, for a detailed discussion about the physical origin of the dHYM equation and its relation with string theory and mirror symmetry, we suggest \cite{Collins2018deformed}. As it can be shown, Eq. (\ref{dHYMeq}) has an alternative (equivalent) formulation in terms of the notion of {\textit{Lagrangian phase}} \cite{Collins2020}, more precisely, Eq. (\ref{dHYMeq}) is equivalent to the fully nonlinear elliptic equation
\begin{equation}
\label{dHYMeq2v}
\Theta_{\omega}(\chi) := \sum_{j = 1}^{n}\arctan(\lambda_{j}) = \hat{\Theta}\mod 2\pi,
\end{equation}
where $\lambda_{1},\ldots,\lambda_{n}$ are the eigenvalues of $\omega^{-1} \circ \chi$. In this last equation $\Theta_{\omega}(\chi)$ is called the Lagrangian phase of $\chi$ with respect to $\omega$. The topological constant $\hat{\Theta}$ can be obtained integrating Eq. (\ref{dHYMeq}), i.e., it is given by the principal argument of the complex number
\begin{equation}
Z_{[\omega]} := \int_{X}\frac{(\omega + \sqrt{-1}\psi)^{n}}{n!},
\end{equation}
which depends only on the classes $[\omega]$, $[\psi]$. Thus a necessary condition for existence of a solution is that $Z_{[\omega]} \neq 0$. In \cite{JacobYau2017}, it was shown that solutions of the dHYM equation are unique, up to addition of a constant. In view of the formulation provided by Eq. (\ref{dHYMeq2v}), since $\Theta_{\omega}(\chi)$ is real valued and $\hat{\Theta}$ is $S^{1}$-valued, we need to lift $\hat{\Theta}$ to $\mathbbm{R}$ study Eq. (\ref{dHYMeq2v}). Further, under the hypothesis of the supercritical condition, which means that the lifted angle $\hat{\Theta}$ satisfies $\hat{\Theta} > (n-2)\frac{\pi}{2}$, Collins–Jacob–Yau \cite{Chen2021j} showed that if there exists a supercritical $C$-subsolution, then the dHYM equation is solvable. Also, in \cite{Pingali2019}, some existence results for the dHYM equation were obtained for some ranges of the phase angle $\hat{\Theta}$ assuming the existence of a subsolution. 

As demonstrated in \cite{Collins2020, Collins2021moment, Collins2020stability}, it is possible to use algebraic geometry to study the problem related to the existence of solutions to the dHYM equation. More precisely, in the setting of Problem \ref{problem1}, given an analytic subvariety $Y \subseteq X$ of dimension $p$, one can define the {\textit{central charge}}
\begin{equation}
\label{centralchargedef}
Z_{Y}([\psi]):= -\int_{Y}{\rm{e}}^{-\sqrt{-1}(\omega + \sqrt{-1}\psi)} = -\frac{(-\sqrt{-1})^{p}}{p!}\int_{Y}(\omega + \sqrt{-1}\psi)^{p}.
\end{equation}
If $\dim_{\mathbbm{C}}(X) = 2$, it has been shown in \cite[Proposition 8.5]{Collins2020} that a solution to the dHYM equation exists in $[\psi] \in H^{1,1}(X,\mathbbm{R})$ if and only if, for every curve $C \subset X$ we have
\begin{equation}
{\rm{Im}} \bigg( \frac{Z_{C}([\psi])}{Z_{X}([\psi])}\bigg) > 0.
\end{equation}
More generally, we have the following conjecture.
\begin{conjecture}[Collins-Jacob-Yau \cite{Collins2020}]
\label{conjecture1}
There exists a solution to the deformed Hermitian Yang-Mills equation in the class $[\psi]$ with phase angle $\hat{\Theta}\in \big ( (n-2)\frac{\pi}{2},n\frac{\pi}{2}\big)$ if and only if 
\begin{equation}
{\rm{Im}} \bigg( \frac{Z_{Y}([\psi])}{Z_{X}([\psi])}\bigg) > 0.
\end{equation}
for all irreducible analytic subvariety $Y \varsubsetneq X_{P}$. 
\end{conjecture}

As mentioned in \cite{Collins2020}, the complex number $Z_{Y}(-)$ resembles various notions of central charge appearing in stability conditions in several physical and mathematical theories, e.g. \cite{douglas2001a, douglas2001b,douglas2002c} and \cite{thomas2002stability}. In the particular case that $[\psi] = c_{1}(L)$, for some $L \in {\rm{Pic}}(X)$, it is conjectured \cite{Collins2020stability} that the existence of a solution to the dHYM equation in the class $[\psi]$ should be equivalent to the Bridgeland stability \cite{Bridgeland2007stability} of the line bundle $L$, for an introduction to Bridgeland stability we suggest \cite{Macri2017lectures}. In this paper, we study the dHYM equation on projective manifolds defined by rational homogeneous varieties. In this setting, we show that the dHYM equation on a rational homogeneous variety, equipped with any invariant K\"{a}hler metric, always admits a solution. In particular, we describe the Lagrangian phase, with respect to any invariant K\"{a}hler metric, of every closed invariant $(1,1)$-form in terms of Lie theory. From this description, we characterize all supercritical and hypercritical homogeneous solutions of the dHYM equation using the Cartan matrix associated with the underlying complex simple Lie algebra. Further, we provide an explicit formula, in terms of Lie theory, for the slope of holomorphic vector bundles over rational homogeneous varieties. Using this formula, we derive a new criterion for slope semistability through restrictions of holomorphic vector bundles to the generators of the associated cone of curves. Moreover, we provide a new characterization, in terms of central charges defined by rational curves, for slope (semi)stability of holomorphic vector bundles over rational homogeneous varieties. As an application of our main results, we describe all supercritical and hypercritical homogeneous solutions of the dHYM equation on the Fano threefold defined by the Wallach flag manifold ${\mathbbm{P}}(T_{{\mathbbm{P}^{2}}})$. In addition, we introduce a constructive method to obtain non-trivial Hermitian-Einstein metrics on holomorphic vector bundles over ${\mathbbm{P}}(T_{{\mathbbm{P}^{2}}})$ from solutions of linear diophantine equations. In this last case, we describe explicitly all associated Hermitian Yang-Mills instantons. We also present some new insights that explore the interplay between intersection theory and number theory.

\subsection{Main results} A rational homogeneous variety can be described as a quotient $X_{P} = G^{\mathbbm{C}}/P$, where $G^{\mathbbm{C}}$ is a semisimple complex algebraic group with Lie algebra $\mathfrak{g}^{\mathbbm{C}} = {\rm{Lie}}(G^{\mathbbm{C}})$, and $P$ is a parabolic Lie subgroup (Borel-Remmert \cite{BorelRemmert}). Regarding $G^{\mathbbm{C}}$ as a complex analytic space, without loss of generality, we may assume that $G^{\mathbbm{C}}$ is a connected simply connected complex simple Lie group. Fixed a compact real form $G \subset G^{\mathbbm{C}}$, and considering $X_{P} = G/G \cap P$ as a $G$-space, we are interested in the $G$-invariant solutions of Eq. (\ref{dHYMeq}). Fixed a Cartan subalgebra $\mathfrak{h} \subset \mathfrak{g}^{\mathbbm{C}}$ and a simple root system $\Delta \subset \mathfrak{h}^{\ast}$, up to conjugation, we have $P \subset G^{\mathbbm{C}}$ completely determined by some $I \subset \Delta$, e.g. \cite[\S 3.1]{Akhiezer}. Moreover, considering the associated fundamental weights $\varpi_{\alpha} \in \mathfrak{h}^{\ast}$, $\alpha \in \Delta$, it follows that 
\begin{equation}
{\rm{Pic}}(X_{P}) \cong H^{1,1}(X_{P},\mathbbm{Z}) \cong \Lambda_{P}:= \bigoplus_{\alpha \in \Delta \backslash I}\mathbbm{Z}\varpi_{\alpha},
\end{equation}
see for instance \cite[Theorem 6.4]{Snowhomovec} (or Remark \ref{remark1_1}). From the aforementioned isomorphisms, we have a map
\begin{equation}
\label{map11character}
[\omega] \in H^{1,1}(X_{P},\mathbbm{Z}) \mapsto \lambda([\omega]) \in \Lambda_{P}.
\end{equation}
Denoting by $\Phi = \Phi^{+} \cup \Phi^{-}$ the root system associated with $\Delta \subset \mathfrak{h}^{\ast}$, and considering  $\Phi_{I}^{\pm}:= \Phi^{\pm} \backslash \langle I \rangle^{\pm}$, we have the following result.
\begin{thm}
\label{theoremA}
Given a  K\"{a}hler class $[\omega] \in \mathcal{K}(X_{P})$, then for every $[\psi] \in H^{1,1}(X_{P},\mathbbm{R})$ we have 
\begin{equation}
\label{phaseangleCartan}
\hat{\Theta}: = {\rm{Arg}} \int_{X_{P}}\frac{(\omega + \sqrt{-1}\psi)^{n}}{n!} = \sum_{\beta \in \Phi_{I}^{+}} \arctan \bigg( \frac{\langle \lambda([\psi]),\beta^{\vee} \rangle}{\langle \lambda([\omega]),\beta^{\vee} \rangle}\bigg) \mod 2 \pi,
\end{equation}
such that $\lambda([\psi]), \lambda([\omega]) \in \Lambda_{P} \otimes \mathbbm{R}$. In particular, fixed the unique $G$-invariant representative $\omega_{0} \in [\omega]$, there exists $\phi \in C^{\infty}(X_{P})$, such that $\chi_{\phi}:= \psi + \sqrt{-1}\partial \overline \partial \phi$ satisfies the deformed Hermitian Yang-Mills equation
\begin{equation}
{\rm{Im}}\big ( \omega_{0} + \sqrt{-1}\chi_{\phi}\big)^{n} = \tan(\hat{\Theta}) {\rm{Re}}\big ( \omega_{0} + \sqrt{-1}\chi_{\phi}\big)^{n}.
\end{equation}
\end{thm}
We notice that the sum on the right-hand side of Eq. (\ref{phaseangleCartan}) can be explicitly described in concrete cases through the coefficients of $\lambda([\psi]), \lambda([\omega]) \in \Lambda_{P} \otimes \mathbbm{R}$ with respect to the basis $\{\varpi_{\alpha} \ | \ \alpha \in \Delta \backslash I\}$ and the Cartan matrix of $\mathfrak{g}^{\mathbbm{C}}$ (e.g. \cite{Humphreys}). As a consequence of the above theorem, we have the following corollary.
\begin{corollarynon}
In the setting of the previous theorem, given $[\psi] \in H^{1,1}(X_{P},\mathbbm{R})$, considering $\eta_{\phi} = \psi + \sqrt{-1}\partial \overline{\partial} \phi$, for each $\phi \in C^{\infty}(X_{P})$, we have that the space of the almost calibrated $(1,1)$-forms 
\begin{equation}
\label{spacealmostcalibrated}
\mathcal{H} := \big \{ \phi \in C^{\infty}(X_{P}) \ \big | \ {\rm{Re}} \big ( {\rm{e}}^{\sqrt{-1} \hat{\Theta}}(\omega_{0} + \sqrt{-1} \eta_{\phi})^{n} \big) > 0 \big \},
\end{equation}
is non-empty ($\mathcal{H} \neq \emptyset$). In particular, we have the unique lift $\hat{\Theta}([\psi]) \in (-n\frac{\pi}{2},n\frac{\pi}{2})$ of $\hat{\Theta}$ given by
\begin{equation}
\label{uniquelifted}
\hat{\Theta}([\psi]):= \sum_{\beta \in \Phi_{I}^{+}} \arctan \bigg( \frac{\langle \lambda([\psi]),\beta^{\vee} \rangle}{\langle \lambda([\omega]),\beta^{\vee} \rangle}\bigg),
\end{equation}
such that $\lambda([\psi]), \lambda([\omega]) \in \Lambda_{P} \otimes \mathbbm{R}$.
\end{corollarynon}

In the above setting, given $[\psi] \in H^{1,1}(X_{P},\mathbbm{R})$, for the sake of simplicity, we shall identify the associated phase angle $\hat{\Theta}$ with its unique lift $\hat{\Theta}([\psi])$ described in Eq. (\ref{uniquelifted}). The space $\mathcal{H}$ described in Eq. (\ref{spacealmostcalibrated}) plays an important role in the study of mirror symmetry \cite{Collins2021moment}, it is related by mirror symmetry to the space of positive (or almost calibrated) Lagrangians studied by Solomon \cite{Solomon2013calabi, Solomon2014curvature}. In \cite{Chu2021space}, it was shown that $\mathcal{H}$ is an infinite dimensional Riemannian manifold with non-positive sectional curvature and, in the hypercritical phase, which means that $\hat{\Theta}$ satisfies $\hat{\Theta} > (n-1)\frac{\pi}{2}$, it has a well-defined metric structure, and its completion is a ${\rm{CAT}}(0)$ geodesic metric space. We observe that, since $\lambda([\psi]), \lambda([\omega]) \in \Lambda_{P} \otimes \mathbbm{R}$ depend only on the cohomology classes involved, the unique lifted angle $\hat{\Theta}([\psi]) \in (-n\frac{\pi}{2},n\frac{\pi}{2})$ coincides with the Lagrangian phase of the unique $G$-invariant representative in the class $[\psi]$. Inspired by Conjecture \ref{conjecture1}, we prove the following result.

\begin{thm}
\label{theoremB}
For every irreducible analytic subvariety $Y \subset X_{P}$, define
\begin{equation}
\label{phasesubvariety}
\Theta_{Y} := {\rm{Arg}}\int_{Y}(\omega + \sqrt{-1}\psi)^{\dim(Y)},
\end{equation}
such that $[\omega] \in \mathcal{K}(X_{P})$ and $[\psi] \in H^{1,1}(X_{P},\mathbbm{R})$. If
\begin{equation}
\label{hypercritical}
\sum_{\beta \in \Phi_{I}^{+}} \arccot \bigg( \frac{\langle \lambda([\psi]),\beta^{\vee} \rangle}{\langle \lambda([\omega]),\beta^{\vee} \rangle}\bigg) < \pi,
\end{equation}
then, for every proper irreducible analytic subvariety $Y \subset X_{P}$, the following holds
\begin{equation}
\label{phasecharge}
\Theta_{Y} > \Theta_{X_{P}} - \big (n -\dim(Y) \big ) \frac{\pi}{2}.
\end{equation}
\end{thm}
As it can be seen, Eq. (\ref{hypercritical}) is equivalent to the supercritical condition, thus the inequality provided by Eq. (\ref{phasecharge}) is a necessary condition for a solution of Eq. (\ref{dHYMeq}) on $X_{P}$ to be supercritical. The proof which we present for Theorem \ref{theoremB} is adapted to the homogeneous setting and based on the ideas introduced in \cite[Lemma 8.2]{Collins2020}. From above, we have the following characterization for the unique lifted angle provided by (Eq. (\ref{uniquelifted})) in terms of central charges.
\begin{corollarynon}
In the setting of the previous theorem, the unique lifted angle $\hat{\Theta} \in (-n\frac{\pi}{2},n\frac{\pi}{2})$ of $[\psi] \in H^{1,1}(X_{P},\mathbbm{R})$ is given by
\begin{equation}
\hat{\Theta} =  \sum_{\beta \in \Phi_{I}^{+}} \Theta_{{\mathbbm{P}}^{1}_{\beta}},
\end{equation}
such that $\mathbbm{P}_{\beta}^{1} = \overline{\exp(\mathfrak{g}_{-\beta})P} \subset X_{P}$, $\forall \beta \in \Phi_{I}^{+}$.
\end{corollarynon}

The corollary above shows that the unique lifted angle associated with some $[\psi] \in H^{1,1}(X_{P},\mathbbm{R})$ can be completely described in terms of the central charges $Z_{{\mathbbm{P}}^{1}_{\beta}}([\psi])$, $\beta \in \Phi_{I}^{+}$, notice that
\begin{equation}
\Theta_{{\mathbbm{P}}^{1}_{\beta}}={\rm{Arg}}\big ({\rm{e}}^{-\frac{ \pi}{2} \sqrt{-1}}
Z_{{\mathbbm{P}}^{1}_{\beta}}([\psi])\big) \mod 2\pi, \ \ \forall \beta \in \Phi_{I}^{+},
\end{equation}
see Eq. (\ref{centralchargedef}) and Eq. (\ref{phasesubvariety}). In the setting of the Problem \ref{problem1}, restricting our attention to the case that $[\psi] = c_{1}({\bf{E}})$, for some holomorphic vector bundle ${\bf{E}} \to X_{P}$, we investigate the relationship between slope (semi)stability and central charges defined by certain rational curves. Let us recall that, fixed a K\"{a}hler class $\xi \in \mathcal{K}(X_{P})$, the slope $\mu_{\xi}({\bf{E}})$ of a holomorphic vector bundle ${\bf{E}} \to X_{P}$, with respect to $\xi$, is defined by
\begin{equation}
\mu_{\xi}({\bf{E}}):= \frac{\int_{X_{P}}c_{1}({\bf{E}}) \wedge \xi^{n-1}}{\rank({\bf{E}})}.
\end{equation}
From above, a holomorphic vector bundle ${\bf{E}} \to X_{P}$ is said to be $\xi$-(semi)stable, in the sense of Mumford-Takemoto, if 
\begin{equation}
\mu_{\xi}({\bf{E}}) \geoq \mu_{\xi}(\mathcal{F}),
\end{equation}
for every subbundle $0 \neq \mathcal{F} \varsubsetneq {\bf{E}}$. In order to provide a criterion for (semi)stability in terms of central charges, we prove firstly the following result.
\begin{thm}
\label{theoremC}
Let ${\bf{E}} \to X_{P}$ be a holomorphic vector bundle with $\rank({\bf{E}}) = r$. Then, the slope of ${\bf{E}}$ with respect to a K\"{a}hler class $\xi \in \mathcal{K}(X_{P})$ is given by
\begin{equation}
\label{slopeform}
\mu_{\xi}({\bf{E}}) = \frac{(n-1)!}{r} \Bigg [ \sum_{\beta \in \Phi_{I}^{+} } \frac{\langle \lambda({\bf{E}}), \beta^{\vee} \rangle}{\langle \lambda(\xi), \beta^{\vee}\rangle} \Bigg ] \Bigg [ \prod_{\beta \in \Phi_{I}^{+}} \frac{\langle \lambda(\xi),\beta^{\vee} \rangle}{\langle \varrho^{+},\beta^{\vee} \rangle}\Bigg],
\end{equation}
such that $\lambda({\bf{E}}) \in \Lambda_{P}$, and $\lambda(\xi) \in \Lambda_{P} \otimes \mathbbm{R}$. Moreover, if the restriction of ${\bf{E}}$ to each generator $[{\mathbbm{P}}^{1}_{\alpha}]$, $\alpha \in \Delta \backslash I$, of the cone of curves ${\rm{NE}}(X_{P})$ is semistable, then ${\bf{E}}$ is $\xi$-semistable with respect to any K\"{a}hler class $\xi \in \mathcal{K}(X_{P})$.
\end{thm}
Notice that $\lambda({\bf{E}}):= \lambda(c_{1}({\bf{E}})) \in \Lambda_{P}$, see Eq. (\ref{map11character}), and $\varrho^{+} = \frac{1}{2}\sum_{\alpha \in \Phi^{+}} \alpha$. The so-called restriction theorems for slope (semi)stability play an important role in the moduli theory of sheaves, see for instance \cite{mehta1982semistable, mehta1984restriction}, \cite{flenner1984restrictions}, and \cite{langer2004semistable}. In fact, such theorems are useful tools used to reduce the study of sheaves to lower dimensions, e.g. \cite[Theorem 7.3.1]{huybrechts2010geometry}. Besides the explicit formula (in terms of Lie theory) for the slope of holomorphic vector bundles over rational homogeneous varieties, the result above provides a new criterion to slope semistability through restrictions of holomorphic vector bundles to the generators of the underlying cone of curves. It is worth pointing out that the expression provided by Eq. (\ref{slopeform}) depends only on certain intersection numbers, namely,
\begin{equation}
\langle \lambda({\bf{E}}), \beta^{\vee} \rangle = \langle c_{1}({\bf{E}}),[\mathbbm{P}^{1}_{\beta}] \rangle \ \ {\text{and}} \ \ \langle \lambda(\xi), \beta^{\vee} \rangle = \langle \xi,[\mathbbm{P}^{1}_{\beta}] \rangle,
\end{equation}
such that $\mathbbm{P}_{\beta}^{1} = \overline{\exp(\mathfrak{g}_{-\beta})P} \subset X_{P}$, $\forall \beta \in \Phi_{I}$. This fact leads us to investigate some aspects of intersection theory underlying the definition of certain central charges. In this setting, we obtain the following result.

\begin{thm}
\label{theoremD}
Fixed a $G$-invariant K\"{a}hler metric $\omega_{0}$ on $X_{P}$, for every holomorphic vector bundle ${\bf{E}} \to X_{P}$ and every irreducible analytic subvariety $Y \subset X_{P}$, define
\begin{equation}
Z_{[\omega_{0}]}({\bf{E}},Y):= - \int_{Y}{\rm{e}}^{-\sqrt{-1}[\omega_{0}]}{\rm{ch}}({\bf{E}}).
\end{equation}
Then, the following hold:
\begin{enumerate}
\item In the particular case that ${\bf{E}} \in {\rm{Pic}}(X_{P})$ and $Y \in {\rm{Div}}(X_{P})$, we have
\begin{equation}
Z_{[\omega_{0}]}({\bf{E}},Y) = -\sum_{\alpha \in \Phi_{I}^{+}}  \Bigg [ \prod_{\beta \in \Phi_{I}^{+}\backslash \{\alpha\}}\Bigg (\frac{ \langle \lambda({\bf{E}}), \beta^{\vee} \rangle}{\langle \lambda([\omega_{0}]), \beta^{\vee} \rangle}\frac{ \langle \lambda_{Y}, \alpha^{\vee} \rangle}{\langle \lambda([\omega_{0}]), \alpha^{\vee} \rangle} - \sqrt{-1}\frac{ \langle \lambda_{Y}, \alpha^{\vee} \rangle}{\langle \lambda([\omega_{0}]), \alpha^{\vee} \rangle} \Bigg ) \Bigg ]V_{0},
\end{equation}
such that $V_{0} = {\rm{Vol}}(X_{P},\omega_{0})$, $\lambda({\bf{E}}),\lambda_{Y} \in \Lambda_{P}$, and $\lambda([\omega_{0}]) \in \Lambda_{P} \otimes \mathbbm{R}$;
\item For every ${\bf{E}} \in {\rm{Pic}}(X_{P})$, such that $c_{1}({\bf{E}}) \neq 0$, we have
\begin{equation}
\label{conditionB}
\frac{|Z_{[\omega_{0}]}({\bf{E}},X_{P})|}{||{\rm{ch}}(\bf{E})||} > 0,
\end{equation}
where $||\cdot||$ is any norm on the finite dimensional vector space $H^{{\text{even}}}(X_{P},\mathbbm{R})$;
\item For every holomorphic vector bundle ${\bf{E}} \to X_{P}$, define
\begin{equation}
\label{slopeangle}
\hat{\mu}_{[\omega_{0}]}({\bf{E}}) := \sum_{\beta \in \Phi_{I}^{+} } \tan \big ( \Theta_{\omega_{0}}({\bf{E}},\mathbbm{P}_{\beta}^{1})\big),
\end{equation}
such that 
\begin{equation}
\Theta_{\omega_{0}}({\bf{E}},\mathbbm{P}_{\beta}^{1}) := {\rm{Arg}} \big (Z_{[\omega_{0}]}({\bf{E}},\mathbbm{P}_{\beta}^{1}) \big) - \frac{\pi}{2}, 
\end{equation}
where $\mathbbm{P}_{\beta}^{1} = \overline{\exp(\mathfrak{g}_{-\beta})P} \subset X_{P}$, $\forall \beta \in \Phi_{I}^{+}$. Then, we have that ${\bf{E}}$ is $[\omega_{0}]$-(semi)stable if, and only if,
\begin{equation}
\hat{\mu}_{[\omega_{0}]}({\bf{E}}) \geoq \hat{\mu}_{[\omega_{0}]}(\mathcal{F}),
\end{equation}
for every subbundle $0 \neq \mathcal{F} \varsubsetneq {\bf{E}}$;
\item Given ${\bf{E}} \in {\rm{Pic}}(X_{P})$, if
\begin{equation}
\label{supercritivalbundle}
\frac{\pi (n-2)}{2} < \Theta_{\omega_{0}}({\bf{E}}):= \sum_{\beta \in \Phi_{I}^{+} }  \Theta_{\omega_{0}}({\bf{E}},\mathbbm{P}_{\beta}^{1}) < \frac{n\pi}{2},
\end{equation}
such that $\lambda({\bf{E}}) \in \Lambda_{P}$ and $\lambda([\omega_{0}]) \in \Lambda_{P} \otimes \mathbbm{R}$, then
\begin{equation}
\label{inequalitysupercritical}
{\rm{Im}} \bigg( \frac{Z_{[\omega_{0}]}({\bf{E}},Y)}{Z_{[\omega_{0}]}({\bf{E}},X_{P})}\bigg) > 0,
\end{equation}
for every irreducible analytic subvariety $Y \varsubsetneq X_{P}$;
\item Given a holomorphic vector bundle ${\bf{E}} \to X_{P}$, if
\begin{equation}
\label{stabilityinequalities}
{\rm{Arg}} \big (Z_{[\omega_{0}]}({\bf{E}},\mathbbm{P}_{\beta}^{1}) \big)  \geoq {\rm{Arg}} \big (Z_{[\omega_{0}]}(\mathcal{F},\mathbbm{P}_{\beta}^{1}) \big), 
\end{equation}
for every subbundle $0 \neq \mathcal{F} \varsubsetneq {\bf{E}}$ and every $\beta \in \Phi_{I}^{+}$, then ${\bf{E}}$ is $[\omega_{0}]$-(semi)stable.
\end{enumerate}
\end{thm}
From item (1) of the above theorem, we have a general formula for the central charges defined by pairs $(Y,{\bf{E}})$, such that $Y \in {\rm{Div}}(X_{P})$ and ${\bf{E}} \in {\rm{Pic}}(X_{P})$, in terms of Lie theory. In view of the results provided in \cite{duan2003degree}, it seems to be possible to obtain an explicit formula for all central charges defined by Schubert varieties and line bundles. The result provided by item (2) follows directly from Theorem \ref{theoremA}. In particular, it shows that the homogeneous representative of $c_{1}({\bf{E}})$, for every ${\bf{E}} \in {\rm{Pic}}(X_{P})$, minimizes the functional
\begin{equation}
\chi \in c_{1}({\bf{E}}) \mapsto \int_{X_{P}}r_{\omega_{0}}(\chi)\frac{\omega_{0}^{n}}{n!},
\end{equation}
where $r_{\omega_{0}}(\chi)= |\det \big (\mathbbm{1} + \sqrt{-1}(\omega_{0}^{-1} \circ \chi)\big )|$, $\forall \chi \in c_{1}({\bf{E}})$, see for instance \cite{JacobYau2017}. It is worth mentioning that Eq. (\ref{conditionB}) is precisely the support property required in the definition of a Bridgeland stability condition, cf. \cite[\S 2.2]{bayer2011bridgeland}, \cite[\S 5.2]{Macri2017lectures}. From item (3) of the above theorem, we obtain a new characterization for (semi)stability of holomorphic vector bundles over rational homogeneous varieties through the argument of certain central charges defined by rational curves. In order to prove item (3), we show that
\begin{equation}
\mu_{[\omega_{0}]}({\bf{E}}) = -(n-1)!\Bigg [ \sum_{\beta \in \Phi_{I}^{+} } \tan \big ( \Theta_{\omega_{0}}({\bf{E}},\mathbbm{P}_{\beta}^{1}) \big )\Bigg ]{\rm{Vol}}(X_{P},\omega_{0}).
\end{equation}
It follows from the above expression that the $[\omega_{0}]$-slope of a holomorphic vector bundle over a rational homogeneous variety can be obtained through the standard notion of slope (inclination) of lines in the complex plane. The expression on the right-hand side of Eq. (\ref{slopeangle}) also can be written in the following way
\begin{equation}
\hat{\mu}_{[\omega_{0}]}({\bf{E}}) = - \sum_{\beta \in \Phi_{I}^{+}} \frac{{\rm{Re}}\big (Z_{[\omega_{0}]}({\bf{E}},\mathbbm{P}_{\beta}^{1})\big )}{{\rm{Im}}\big (Z_{[\omega_{0}]}({\bf{E}},\mathbbm{P}_{\beta}^{1})\big )}.
\end{equation}
We observe that the above description of $\hat{\mu}_{[\omega_{0}]}({\bf{E}})$ resembles the notion of $Z$-slope (or generalized slope) used to define stability conditions on abelian categories, cf. \cite[\S 4]{Macri2017lectures}. In fact, considering the Grothendieck group of coherent sheaves $K_{0}(X_{P})$, we have that
\begin{equation}
Z_{[\omega_{0}]}(-,\mathbbm{P}_{\beta}^{1}) \colon K_{0}(X_{P}) \to \mathbbm{C},
\end{equation}
defines an additive homomorphism, for every $\beta \in \Phi_{I}^{+}$, see Remark \ref{additivecentralcharge}. The result provided by item (4) is a consequence of the inequality obtained in Eq. (\ref{phasecharge}), cf. Conjecture \ref{conjecture1}. It is worth pointing out that Eq. (\ref{supercritivalbundle}) is equivalent to the supercritical condition on the unique solution of the dHYM equation in $c_{1}({\bf{E}})$. Thus, the inequality described in Eq. (\ref{inequalitysupercritical}) is a necessary condition for a solution of Eq. (\ref{dHYMeq}) in $c_{1}({\bf{E}})$ to be supercritical. The proof of item (5) follows directly from item (3). The criterion for $[\omega_{0}]$-(semi)stability obtained through the inequalities given in Eq. (\ref{stabilityinequalities}) resembles the notion of stability for a given stability function in the setting of Bridgeland stability \cite{Bridgeland2007stability}, \cite[\S 4]{Macri2017lectures}. From Kobayashi-Hitchin correspondence \cite{donaldson1985anti,donaldson1987infinite},\cite{uhlenbeck1986existence}, we have the following differential-geometric counterpart of item (5) of the above theorem.
\begin{corollarynon}
Given a $G$-invariant K\"{a}hler metric $\omega_{0}$ on $X_{P}$ and a simple holomorphic vector bundle ${\bf{E}} \to X_{P}$, if
\begin{equation}
\label{argumentinequalities}
{\rm{Arg}} \big (Z_{[\omega_{0}]}({\bf{E}},\mathbbm{P}_{\beta}^{1}) \big) > {\rm{Arg}} \big (Z_{[\omega_{0}]}(\mathcal{F},\mathbbm{P}_{\beta}^{1}) \big), 
\end{equation}
for every subbundle $0 \neq \mathcal{F} \varsubsetneq {\bf{E}}$ and every $\beta \in \Phi_{I}^{+}$, then ${\bf{E}}$ admits a Hermitian metric $H$ solving the Hermitian-Yang-Mills equation
\begin{equation}
\sqrt{-1}\Lambda_{\omega_{0}}F(H) = c \mathbbm{1}_{{\bf{E}}}.
\end{equation}
\end{corollarynon}
Further, combining \cite[Proposition 7.11]{kobayashi_dif_cplx_vec} with item (5) of the previous theorem, we obtain the following.

\begin{corollarynon}
Let ${\bf{E}}$ and ${\bf{F}}$ be $[\omega_{0}]$-semistable holomorphic vector bundles over $X_{P}$, for some $G$-invariant K\"{a}hler metric $\omega_{0}$ on $X_{P}$. If
\begin{equation}
{\rm{Arg}} \big (Z_{[\omega_{0}]}({\bf{E}},\mathbbm{P}_{\beta}^{1}) \big) > {\rm{Arg}} \big (Z_{[\omega_{0}]}({\bf{F}},\mathbbm{P}_{\beta}^{1}) \big), 
\end{equation}
for every $\beta \in \Phi_{I}^{+}$, then ${\rm{Hom}}({\bf{E}},{\bf{F}}) = 0$.
\end{corollarynon}
The above result resembles Schur’s property required in the definition of a stability condition on a triangulated category, see \cite[Definition 1.1, item (c)]{Bridgeland2007stability}. As an application of our main results, among other explicit computations, we classify all supercritical and hypercritical solutions of the dHYM equation on $({\mathbbm{P}}(T_{{\mathbbm{P}^{2}}}),\omega_{0})$, such that $\omega_{0} \in c_{1}({\mathbbm{P}}(T_{{\mathbbm{P}^{2}}}))$. In particular, we give a new example which illustrates that the “easier” direction of Collins–Jacob–Yau’s conjecture (Conjectures \ref{conjecture1}) holds only in the supercritical case (cf. \cite[Remark 1.10]{Chen2021j}). Further, we provide a constructive method to obtain non-trivial examples of Hermitian-Einstein metrics on certain holomorphic vector bundles over ${\mathbbm{P}}(T_{{\mathbbm{P}^{2}}})$ from solutions of linear diophantine equations. The general idea behind this last result is proved in Appendix \ref{appendix}. Additionally, in Appendix \ref{appendix}, we also present some new insights that explore the interplay between intersection theory and number theory. Among other results, we provide a complete description for the Hodge-Riemann bilinear form 
\begin{equation}
\mathcal{Q}_{\omega_{0}} \colon H^{2}(X_{P},\mathbbm{Z}) \times H^{2}(X_{P},\mathbbm{Z}) \to \mathbbm{Z}, \ \ \mathcal{Q}_{\omega_{0}}(a,b) := \int_{X_{P}} a \wedge b \wedge [\omega_{0}]^{n-2},
\end{equation}
defined by an integral K\"{a}hler class $[\omega_{0}] \in \mathcal{K}(X_{P})$ (e.g. \cite{VoisinBook1}, \cite{peters2008mixed}), in terms of the Cartan matrix associated with the complex simple Lie algebra underlying $X_{P}$. From this description, fixed some integral K\"{a}hler class $[\omega_{0}] \in \mathcal{K}(X_{P})$, and considering
\begin{equation}
\mathcal{A}_{[\omega_{0}]}(n) := \Bigg \{ m \in  \mathbbm{Z}_{>0}  \ \ \Bigg | \ \ \begin{array}{l}
 \ \ \ \ 1 \leq m \leq n, \\
\exists \ {\bf{E}} \in {\rm{Pic}}(X_{P}), \ {\text{s.t.}} \ \mu_{[\omega_{0}]}({\bf{E}}) = m.
\end{array}\Bigg \},
\end{equation}
for every $n \in \mathbbm{Z}_{>0}$, we obtain the following result.
\begin{thm}
\label{densityslope}
For every integral K\"{a}hler class $[\omega_{0}] \in \mathcal{K}(X_{P})$, we have
\begin{equation}
\lim_{n \to +\infty}\frac{\big | \mathcal{A}_{[\omega_{0}]}(n)\big |}{n} = \frac{1}{\tau([\omega_{0}])},
\end{equation}
such that $\tau([\omega_{0}]):= {\rm{gcd}}\big \{ \deg_{\omega_{0}}(\mathscr{O}_{\alpha}(1)) \ \big | \ \alpha \in \Delta \backslash I\big\}$, where $\mathscr{O}_{\alpha}(1)$, $\alpha \in \Delta \backslash I$, are the generators of the Picard group ${\rm{Pic}}(X_{P})$.
\end{thm}
The above result shows that the subset 
\begin{equation}
\mathcal{A}_{[\omega_{0}]} = \big \{m \in \mathbbm{Z}_{>0} \ \big | \ \mu_{[\omega_{0}]}({\bf{E}}) = m, \ {\text{for some}} \ {\bf{E}} \in {\rm{Pic}}(X_{P}) \big\},
\end{equation}
has asymptotic density (a.k.a. natural density) 
\begin{equation}
d(\mathcal{A}_{[\omega_{0}]}) := {\displaystyle{\lim_{n\to +\infty}}}\frac{\big |\mathcal{A}_{[\omega_{0}]} \cap [1,n] \big |}{n} = \frac{1}{\tau([\omega_{0}])}. 
\end{equation}
In a suitable sense, $d(\mathcal{A}_{[\omega_{0}]})$ measures the proportion of natural numbers that belong to $\mathcal{A}_{[\omega_{0}]}$, i.e., the proportion of natural numbers which can be realized as the $[\omega_{0}]$-slope of some holomorphic line bundle over $X_{P}$, for some integral K\"{a}hler class $[\omega_{0}] \in \mathcal{K}(X_{P})$. For more details about natural density, we suggest \cite{nathanson2008elementary} and references therein. Also as an application of the description of $\mathcal{Q}_{\omega_{0}}$ provided in Appendix \ref{appendix}, we prove tha following theorem.
\begin{thm}
\label{Kgroup}
Given an integral K\"{a}hler class $[\omega_{0}] \in \mathcal{K}(X_{P})$, then we have 
\begin{equation}
K_{0}(X_{P}) \cong SK_{0}(X_{P}) \oplus {\rm{Pic}}^{0}_{\omega_{0}}(X_{P}) \oplus \tau([\omega_{0}])\mathbbm{Z},
\end{equation}
such that 
\begin{enumerate}
\item $SK_{0}(X_{P}) := \ker \big ( \det \colon K_{0}(X_{P}) \to {\rm{Pic}}(X_{P})\big)$,
\item ${\rm{Pic}}^{0}_{\omega_{0}}(X_{P}) := \big \{ {\bf{E}} \in {\rm{Pic}}(X_{P}) \ \big | \ \deg_{\omega_{0}}({\bf{E}}) = 0 \big \}$,
\item $\tau([\omega_{0}]):= {\rm{gcd}}\big \{ \deg_{\omega_{0}}(\mathscr{O}_{\alpha}(1)) \ \big | \ \alpha \in \Delta \backslash I\big\}$,
\end{enumerate}
where $\mathscr{O}_{\alpha}(1)$, $\alpha \in \Delta \backslash I$, are the generators of ${\rm{Pic}}(X_{P})$. Moreover, the generators of ${\rm{Pic}}^{0}_{\omega_{0}}(X_{P})$ are completely determined by the Hodge-Riemann bilinear form ${\mathcal{Q}}_{\omega_{0}}$.
\end{thm}
The result above shows that every integral K\"{a}hler class $[\omega_{0}] \in \mathcal{K}(X_{P})$ induces a decomposition on the Grothendieck group $K_{0}(X_{P})$. Considering $X_{P}= {\mathbbm{P}}(T_{{\mathbbm{P}^{2}}})$, we illustrate the results provided in Theorem \ref{densityslope} and Theorem \ref{Kgroup} through explicit computations.

{\bf{Acknowledgements.}} E. M. Correa is partially supported by FAEPEX/Unicamp grant 2528/22 and by S\~{a}o Paulo Research Foundation FAPESP grant 2022/10429-3.

\section{Generalities on flag varieties}\label{generalities}
In this section, we review some basic facts about flag varieties. For more details on the subject presented in this section, we suggest \cite{Akhiezer}, \cite{Flagvarieties}, \cite{HumphreysLAG}, \cite{BorelRemmert}.
\subsection{The Picard group of flag varieties}
\label{subsec3.1}
Let $G^{\mathbbm{C}}$ be a connected, simply connected, and complex Lie group with simple Lie algebra $\mathfrak{g}^{\mathbbm{C}}$. By fixing a Cartan subalgebra $\mathfrak{h}$ and a simple root system $\Delta \subset \mathfrak{h}^{\ast}$, we have a decomposition of $\mathfrak{g}^{\mathbbm{C}}$ given by
\begin{center}
$\mathfrak{g}^{\mathbbm{C}} = \mathfrak{n}^{-} \oplus \mathfrak{h} \oplus \mathfrak{n}^{+}$, 
\end{center}
where $\mathfrak{n}^{-} = \sum_{\alpha \in \Phi^{-}}\mathfrak{g}_{\alpha}$ and $\mathfrak{n}^{+} = \sum_{\alpha \in \Phi^{+}}\mathfrak{g}_{\alpha}$, here we denote by $\Phi = \Phi^{+} \cup \Phi^{-}$ the root system associated with the simple root system $\Delta \subset \mathfrak{h}^{\ast}$. Let us denote by $\kappa$ the Cartan-Killing form of $\mathfrak{g}^{\mathbbm{C}}$. From this, for every  $\alpha \in \Phi^{+}$, we have $h_{\alpha} \in \mathfrak{h}$, such  that $\alpha = \kappa(\cdot,h_{\alpha})$, and we can choose $x_{\alpha} \in \mathfrak{g}_{\alpha}$ and $y_{-\alpha} \in \mathfrak{g}_{-\alpha}$, such that $[x_{\alpha},y_{-\alpha}] = h_{\alpha}$. From these data, we can define a Borel subalgebra\footnote{A maximal solvable subalgebra of $\mathfrak{g}^{\mathbbm{C}}$.} by setting $\mathfrak{b} = \mathfrak{h} \oplus \mathfrak{n}^{+}$. 

\begin{remark}
In the above setting, $\forall \phi \in \mathfrak{h}^{\ast}$, we also denote $\langle \phi, \alpha \rangle = \phi(h_{\alpha})$, $\forall \alpha \in \Phi^{+}$.
\end{remark}

Now we consider the following result (see for instance \cite{Flagvarieties}, \cite{HumphreysLAG}):
\begin{theorem}
Any two Borel subgroups are conjugate.
\end{theorem}
From the result above, given a Borel subgroup $B \subset G^{\mathbbm{C}}$, up to conjugation, we can always suppose that $B = \exp(\mathfrak{b})$. In this setting, given a parabolic Lie subgroup\footnote{i.e., a Lie subgroup which contains some Borel subgroup.} $P \subset G^{\mathbbm{C}}$, without loss of generality, we can suppose that
\begin{center}
$P  = P_{I}$, \ for some \ $I \subset \Delta$,
\end{center}
where $P_{I} \subset G^{\mathbbm{C}}$ is the parabolic subgroup which integrates the Lie subalgebra 
\begin{center}

$\mathfrak{p}_{I} = \mathfrak{n}^{+} \oplus \mathfrak{h} \oplus \mathfrak{n}(I)^{-}$, \ with \ $\mathfrak{n}(I)^{-} = \displaystyle \sum_{\alpha \in \langle I \rangle^{-}} \mathfrak{g}_{\alpha}$. 

\end{center}
By definition, it is straightforward to show that $P_{I} = N_{G^{\mathbbm{C}}}(\mathfrak{p}_{I})$, where $N_{G^{\mathbbm{C}}}(\mathfrak{p}_{I})$ is the normalizer in  $G^{\mathbbm{C}}$ of $\mathfrak{p}_{I} \subset \mathfrak{g}^{\mathbbm{C}}$, see for instance \cite[\S 3.1]{Akhiezer}. In what follows, it will be useful for us to consider the following basic chain of Lie subgroups

\begin{center}

$T^{\mathbbm{C}} \subset B \subset P \subset G^{\mathbbm{C}}$.

\end{center}
For each element in the aforementioned chain of Lie subgroups we have the following characterization: 

\begin{itemize}

\item $T^{\mathbbm{C}} = \exp(\mathfrak{h})$;  \ \ (complex torus)

\item $B = N^{+}T^{\mathbbm{C}}$, where $N^{+} = \exp(\mathfrak{n}^{+})$; \ \ (Borel subgroup)

\item $P = P_{I} = N_{G^{\mathbbm{C}}}(\mathfrak{p}_{I})$, for some $I \subset \Delta \subset \mathfrak{h}^{\ast}$. \ \ (parabolic subgroup)

\end{itemize}
Now let us recall some basic facts about the representation theory of $\mathfrak{g}^{\mathbbm{C}}$, a detailed exposition on the subject can be found in \cite{Humphreys}. For every $\alpha \in \Phi$, we set 
$$\alpha^{\vee} := \frac{2}{\langle \alpha, \alpha \rangle}\alpha.$$ 
The fundamental weights $\{\varpi_{\alpha} \ | \ \alpha \in \Delta\} \subset \mathfrak{h}^{\ast}$ of $(\mathfrak{g}^{\mathbbm{C}},\mathfrak{h})$ are defined by requiring that $\langle \varpi_{\alpha}, \beta^{\vee} \rangle= \delta_{\alpha \beta}$, $\forall \alpha, \beta \in \Delta$. We denote by 
$$\Lambda^{+} = \bigoplus_{\alpha \in \Delta}\mathbbm{Z}_{\geq 0}\varpi_{\alpha},$$ 
the set of integral dominant weights of $\mathfrak{g}^{\mathbbm{C}}$. Let $V$ be an arbitrary finite dimensional $\mathfrak{g}^{\mathbbm{C}}$-module. By considering its weight space decomposition
\begin{center}
$\displaystyle{V = \bigoplus_{\mu \in \Phi(V)}V_{\mu}},$ \ \ \ \ 
\end{center}
such that $V_{\mu} = \{v \in V \ | \ h \cdot v = \mu(h)v, \ \forall h \in \mathfrak{h}\} \neq \{0\}$, $\forall \mu \in \Phi(V) \subset \mathfrak{h}^{\ast}$, we have the following definition.

\begin{definition}
A highest weight vector (of weight $\lambda$) in a $\mathfrak{g}^{\mathbbm{C}}$-module $V$ is a non-zero vector $v_{\lambda}^{+} \in V_{\lambda}$, such that 
\begin{center}
$x \cdot v_{\lambda}^{+} = 0$, \ \ \ \ \ ($\forall x \in \mathfrak{n}^{+}$).
\end{center}
A weight $\lambda \in \Phi(V)$ associated with a highest weight vector is called highest weight of $V$.
\end{definition}

From above, we consider the following standard results.

\begin{theorem}
Every finite dimensional irreducible $\mathfrak{g}^{\mathbbm{C}}$-module $V$ admits a highest weight vector $v_{\lambda}^{+}$. Moreover, $v_{\lambda}^{+}$ is the unique highest weight vector of $V$, up to non-zero scalar multiples. 
\end{theorem}

\begin{theorem}
Let $V$ and $W$ be finite dimensional irreducible $\mathfrak{g}^{\mathbbm{C}}$-modules with highest weight $\lambda \in \mathfrak{h}^{\ast}$. Then, $V$ and $W$ are isomorphic.
\end{theorem}
\begin{remark}
We will denote by $V(\lambda)$ a finite dimensional irreducible $\mathfrak{g}^{\mathbbm{C}}$-module with highest weight $\lambda \in \mathfrak{h}^{\ast}$.
\end{remark}
\begin{theorem} In the above setting, the following hold:
\begin{itemize}
\item[(1)] If $V$ is a finite dimensional irreducible $\mathfrak{g}^{\mathbbm{C}}$-module with highest weight $\lambda \in \mathfrak{h}^{\ast}$, then $\lambda \in \Lambda^{+}$.

\item[(2)] If $\lambda \in \Lambda^{+}$, then there exists a finite dimensional irreducible $\mathfrak{g}^{\mathbbm{C}}$-module $V$, such that $V = V(\lambda)$. 
\end{itemize}
\end{theorem}
From the above theorem, it follows that the map $\lambda \mapsto V(\lambda)$ induces an one-to-one correspondence between $\Lambda^{+}$ and the isomorphism classes of finite dimensional irreducible $\mathfrak{g}^{\mathbbm{C}}$-modules.

\begin{remark} In what follows, it will be useful also to consider the following facts:
\begin{enumerate}
\item[(i)] For all $\lambda \in \Lambda^{+}$, we have $V(\lambda) = \mathfrak{U}(\mathfrak{g}^{\mathbbm{C}}) \cdot v_{\lambda}^{+}$, where $\mathfrak{U}(\mathfrak{g}^{\mathbbm{C}})$ is the universal enveloping algebra of $\mathfrak{g}^{\mathbbm{C}}$;
\item[(ii)] The fundamental representations are defined by $V(\varpi_{\alpha})$, $\alpha \in \Delta$; 
\item[(iii)] For all $\lambda \in \Lambda^{+}$, we have the following equivalence of induced irreducible representations
\begin{center}
$\varrho \colon G^{\mathbbm{C}} \to {\rm{GL}}(V(\lambda))$ \ $\Longleftrightarrow$ \ $\varrho_{\ast} \colon \mathfrak{g}^{\mathbbm{C}} \to \mathfrak{gl}(V(\lambda))$,
\end{center}
such that $\varrho(\exp(x)) = \exp(\varrho_{\ast}x)$, $\forall x \in \mathfrak{g}^{\mathbbm{C}}$, notice that $G^{\mathbbm{C}} = \langle \exp(\mathfrak{g}^{\mathbbm{C}}) \rangle$.
\end{enumerate}
\end{remark}
Given a representation $\varrho \colon G^{\mathbbm{C}} \to {\rm{GL}}(V(\lambda))$, for the sake of simplicity, we shall denote $\varrho(g)v = gv$, for all $g \in G^{\mathbbm{C}}$, and all $v \in V(\lambda)$. Let $G \subset G^{\mathbbm{C}}$ be a compact real form for $G^{\mathbbm{C}}$. Given a complex flag variety $X_{P} = G^{\mathbbm{C}}/P$, regarding $X_{P}$ as a homogeneous $G$-space, that is, $X_{P} = G/G\cap P$, the following theorem allows us to describe all $G$-invariant K\"{a}hler structures on $X_{P}$ through elements of representation theory.
\begin{theorem}[Azad-Biswas, \cite{AZAD}]
\label{AZADBISWAS}
Let $\omega \in \Omega^{1,1}(X_{P})^{G}$ be a closed invariant real $(1,1)$-form, then we have

\begin{center}

$\pi^{\ast}\omega = \sqrt{-1}\partial \overline{\partial}\varphi$,

\end{center}
where $\pi \colon G^{\mathbbm{C}} \to X_{P}$ is the natural projection, and $\varphi \colon G^{\mathbbm{C}} \to \mathbbm{R}$ is given by 
\begin{center}
$\varphi(g) = \displaystyle \sum_{\alpha \in \Delta \backslash I}c_{\alpha}\log \big (||gv_{\varpi_{\alpha}}^{+}|| \big )$, \ \ \ \ $(\forall g \in G^\mathbbm{C})$
\end{center}
with $c_{\alpha} \in \mathbbm{R}$, $\forall \alpha \in \Delta \backslash I$. Conversely, every function $\varphi$ as above defines a closed invariant real $(1,1)$-form $\omega_{\varphi} \in \Omega^{1,1}(X_{P})^{G}$. Moreover, $\omega_{\varphi}$ defines a $G$-invariant K\"{a}hler form on $X_{P}$ if, and only if, $c_{\alpha} > 0$,  $\forall \alpha \in \Delta \backslash I$.
\end{theorem}

\begin{remark}
\label{innerproduct}
It is worth pointing out that the norm $|| \cdot ||$ considered in the above theorem is a norm induced from some fixed $G$-invariant inner product $\langle \cdot, \cdot \rangle_{\alpha}$ on $V(\varpi_{\alpha})$, $\forall \alpha \in \Delta \backslash I$. 
\end{remark}

\begin{remark}
An important consequence of Theorem \ref{AZADBISWAS} is that it allows us to describe the local K\"{a}hler potential for any homogeneous K\"{a}hler metric in a quite concrete way, for some examples of explicit computations, we suggest \cite{CorreaGrama}, \cite{Correa}.
\end{remark}

By means of the above theorem we can describe the unique $G$-invariant representative of each integral class in $H^{2}(X_{P},\mathbbm{Z})$. In fact, consider the associated $P$-principal bundle $P \hookrightarrow G^{\mathbbm{C}} \to X_{P}$. By choosing a trivializing open covering $X_{P} = \bigcup_{i \in J}U_{i}$, in terms of $\check{C}$ech cocycles we can write 
\begin{center}
$G^{\mathbbm{C}} = \Big \{(U_{i})_{i \in J}, \psi_{ij} \colon U_{i} \cap U_{j} \to P \Big \}$.
\end{center}
Given $\varpi_{\alpha} \in \Lambda^{+}$, we consider the induced character $\vartheta_{\varpi_{\alpha}} \in {\text{Hom}}(T^{\mathbbm{C}},\mathbbm{C}^{\times})$, such that $({\rm{d}}\vartheta_{\varpi_{\alpha}})_{e} = \varpi_{\alpha}$. From the homomorphism $\vartheta_{\varpi_{\alpha}} \colon P \to \mathbbm{C}^{\times}$ one can equip $\mathbbm{C}$ with a structure of $P$-space, such that $pz = \vartheta_{\varpi_{\alpha}}(p)^{-1}z$, $\forall p \in P$, and $\forall z \in \mathbbm{C}$. Denoting by $\mathbbm{C}_{-\varpi_{\alpha}}$ this $P$-space, we can form an associated holomorphic line bundle $\mathscr{O}_{\alpha}(1) = G^{\mathbbm{C}} \times_{P}\mathbbm{C}_{-\varpi_{\alpha}}$, which can be described in terms of $\check{C}$ech cocycles by
\begin{equation}
\label{linecocycle}
\mathscr{O}_{\alpha}(1) = \Big \{(U_{i})_{i \in J},\vartheta_{\varpi_{\alpha}}^{-1} \circ \psi_{i j} \colon U_{i} \cap U_{j} \to \mathbbm{C}^{\times} \Big \},
\end{equation}
that is, $\mathscr{O}_{\alpha}(1) = \{g_{ij}\} \in \check{H}^{1}(X_{P},\mathcal{O}_{X_{P}}^{\ast})$, such that $g_{ij} = \vartheta_{\varpi_{\alpha}}^{-1} \circ \psi_{i j}$, $\forall i,j \in J$. 
\begin{remark}
\label{parabolicdec}
We observe that, if we have a parabolic Lie subgroup $P \subset G^{\mathbbm{C}}$, such that $P = P_{I}$, for some $I \subset \Delta$, the decomposition 
\begin{equation}
P_{I} = \big[P_{I},P_{I} \big]T(\Delta \backslash I)^{\mathbbm{C}}, \ \  {\text{such that }} \ \ T(\Delta \backslash I)^{\mathbbm{C}} = \exp \Big \{ \displaystyle \sum_{\alpha \in  \Delta \backslash I}a_{\alpha}h_{\alpha} \ \Big | \ a_{\alpha} \in \mathbbm{C} \Big \},
\end{equation}
e.g. \cite[Proposition 8]{Akhiezer}, shows us that ${\text{Hom}}(P,\mathbbm{C}^{\times}) = {\text{Hom}}(T(\Delta \backslash I)^{\mathbbm{C}},\mathbbm{C}^{\times})$. Therefore, if we take $\varpi_{\alpha} \in \Lambda^{+}$, such that $\alpha \in I$, it follows that $\mathscr{O}_{\alpha}(1) = X_{P} \times \mathbbm{C}$, i.e., the associated holomorphic line bundle $\mathscr{O}_{\alpha}(1)$ is trivial.
\end{remark}

\begin{remark}
Throughout this paper we shall use the following notation
\begin{equation}
\mathscr{O}_{\alpha}(k) := \mathscr{O}_{\alpha}(1)^{\otimes k},
\end{equation}
for every $k \in \mathbbm{Z}$ and every $\alpha \in \Delta \backslash I$. 
\end{remark}

Given $\mathscr{O}_{\alpha}(1) \in {\text{Pic}}(X_{P})$, such that $\alpha \in \Delta \backslash I$, as described above, if we consider an open covering $X_{P} = \bigcup_{i \in J} U_{i}$ which trivializes both $P \hookrightarrow G^{\mathbbm{C}} \to X_{P}$ and $ \mathscr{O}_{\alpha}(1) \to X_{P}$, by taking a collection of local sections $(s_{i})_{i \in J}$, such that $s_{i} \colon U_{i} \to G^{\mathbbm{C}}$, we can define $q_{i} \colon U_{i} \to \mathbbm{R}^{+}$, such that 
\begin{equation}
\label{functionshermitian}
q_{i} :=  {\mathrm{e}}^{-2\pi (\varphi_{\varpi_{\alpha}} \circ s_{i})} = \frac{1}{||s_{i}v_{\varpi_{\alpha}}^{+}||^{2}},
\end{equation}
for every $i \in J$. Since $s_{j} = s_{i}\psi_{ij}$ on $U_{i} \cap U_{j} \neq \emptyset$, and $pv_{\varpi_{\alpha}}^{+} = \vartheta_{\varpi_{\alpha}}(p)v_{\varpi_{\alpha}}^{+}$, for every $p \in P$, and every $\alpha \in \Delta \backslash I$, the collection of functions $(q_{i})_{i \in J}$ satisfy $q_{j} = |\vartheta_{\varpi_{\alpha}}^{-1} \circ \psi_{ij}|^{2}q_{i}$ on $U_{i} \cap U_{j} \neq \emptyset$. Hence, we obtain a collection of functions $(q_{i})_{i \in J}$ which satisfies on the overlaps $U_{i} \cap U_{j} \neq \emptyset$ the following relation
\begin{equation}
\label{collectionofequ}
q_{j} = |g_{ij}|^{2}q_{i},
\end{equation}
such that $g_{ij} = \vartheta_{\varpi_{\alpha}}^{-1} \circ \psi_{i j}$, $\forall i,j \in J$. From this, we can define a Hermitian structure ${\bf{h}}$ on $\mathscr{O}_{\alpha}(1)$ by taking on each trivialization $f_{i} \colon \mathscr{O}_{\alpha}(1) \to U_{i} \times \mathbbm{C}$ the metric defined by
\begin{equation}
\label{hermitian}
{\bf{h}}(f_{i}^{-1}(x,v),f_{i}^{-1}(x,w)) = q_{i}(x) v\overline{w},
\end{equation}
for every $(x,v),(x,w) \in U_{i} \times \mathbbm{C}$. The Hermitian metric above induces a Chern connection $\nabla \myeq {\rm{d}} + \partial \log {\bf{h}}$ with curvature $F_{\nabla}$ satisfying (locally)
\begin{equation}
\displaystyle \frac{\sqrt{-1}}{2\pi}F_{\nabla} \myeq \frac{\sqrt{-1}}{2\pi} \partial \overline{\partial}\log \Big ( \big | \big | s_{i}v_{\varpi_{\alpha}}^{+}\big | \big |^{2} \Big).
\end{equation}
Therefore, by considering the closed $G$-invariant $(1,1)$-form ${\bf{\Omega}}_{\alpha} \in \Omega^{1,1}(X_{P})^{G}$, which satisfies $\pi^{\ast}{\bf{\Omega}}_{\alpha} = \sqrt{-1}\partial \overline{\partial} \varphi_{\varpi_{\alpha}}$, where $\pi \colon G^{\mathbbm{C}} \to G^{\mathbbm{C}} / P = X_{P}$, and $\varphi_{\varpi_{\alpha}}(g) = \frac{1}{2\pi}\log||gv_{\varpi_{\alpha}}^{+}||^{2}$, $\forall g \in G^{\mathbbm{C}}$, we have 
\begin{equation}
{\bf{\Omega}}_{\alpha} |_{U_{i}} = (\pi \circ s_{i})^{\ast}{\bf{\Omega}}_{\alpha} = \frac{\sqrt{-1}}{2\pi}F_{\nabla} \Big |_{U_{i}},
\end{equation}
i.e., $c_{1}(\mathscr{O}_{\alpha}(1)) = [ {\bf{\Omega}}_{\alpha}]$, $\forall \alpha \in \Delta \backslash I$.

\begin{remark}
In what follows, given $I \subset \Delta$, we shall denote $\Phi_{I}^{\pm}:= \Phi^{\pm} \backslash \langle I \rangle^{\pm}$.
\end{remark}

\begin{remark}
\label{bigcellcosntruction}
In order to perform some local computations we shall consider the open set $U^{-}(P) \subset X_{P}$ defined by the ``opposite" big cell in $X_{P}$. This open set is a distinguished coordinate neighbourhood $U^{-}(P) \subset X_{P}$ of $x_{0} = eP \in X_{P}$ defined as follows
\begin{equation}
\label{bigcell}
 U^{-}(P) =  B^{-}x_{0} = R_{u}(P_{I})^{-}x_{0} \subset X_{P},  
\end{equation}
 where $B^{-} = \exp(\mathfrak{h} \oplus \mathfrak{n}^{-})$, and
 
 \begin{center}
 
 $R_{u}(P_{I})^{-} = \displaystyle \prod_{\alpha \in \Phi_{I}^{+}}N_{\alpha}^{-}$, \ \ (opposite unipotent radical)
 
 \end{center}
with $N_{\alpha}^{-} = \exp(\mathfrak{g}_{-\alpha})$, $\forall \alpha \in \Phi_{I}^{+}$, e.g. \cite[\S 3]{Lakshmibai2},\cite[\S 3.1]{Akhiezer}. It is worth mentioning that the opposite big cell defines a contractible open dense subset in $X_{P}$, thus the restriction of any vector bundle (principal bundle) over this open set is trivial.
\end{remark}

The next lemma is fundamental for the ideas that will be developed later on this work.

\begin{lemma}
\label{funddynkinline}
Consider $\mathbbm{P}_{\beta}^{1} = \overline{\exp(\mathfrak{g}_{-\beta})x_{0}} \subset X_{P}$, such that $\beta \in \Phi_{I}^{+}$. Then, 
\begin{equation}
\int_{\mathbbm{P}_{\beta}^{1}} {\bf{\Omega}}_{\alpha} = \langle \varpi_{\alpha}, \beta^{\vee}  \rangle, \ \forall \alpha \in \Delta \backslash I.
\end{equation}

\end{lemma}
\begin{proof}
By definition, we have $\mathbbm{P}_{\beta}^{1} = \overline{N_{\beta}^{-}x_{0}}$, $\forall \beta \in \Phi_{I}^{+}$, thus
\begin{center}
$\displaystyle \int_{\mathbbm{P}_{\beta}^{1}} {\bf{\Omega}}_{\alpha} = \int_{N_{\beta}^{-}x_{0}} {\bf{\Omega}}_{\alpha}$.
\end{center}
Consider the local section $s_{U} \colon U^{-}(P) \to G^{\mathbbm{C}}$, such that $s_{U}(gP) = g$, $\forall gP \in U^{-}(P)$, and the parameterization $u_{\beta} \colon \mathbbm{C} \to N_{\beta}^{-}x_{0}$, such that
\begin{equation}
u_{-\beta}(z) = \phi_{\beta} \begin{pmatrix} 1 & 0 \\
z & 1 \end{pmatrix}, \ \ \ \forall z \in \mathbbm{C},
\end{equation}
where $ \phi_{\beta}$ is the natural Lie group isomorphism between ${\rm{SL}}_{2}(\mathbbm{C})$ and the Lie subgroup corresponding to the Lie subalgebra $\mathfrak{g}_{\beta} \oplus [\mathfrak{g}_{\beta},\mathfrak{g}_{-\beta}] \oplus \mathfrak{g}_{-\beta}$, i.e.,
\begin{equation}
\phi_{\beta \ast} \colon \begin{pmatrix}0 & 1 \\
0 & 0 \end{pmatrix} \mapsto e_{\beta}, \ \  \phi_{\beta \ast} \colon \begin{pmatrix}0 & 0 \\
1 & 0 \end{pmatrix} \mapsto f_{-\beta}, \ \ \phi_{\beta \ast} \colon \begin{pmatrix}1 & \ \ 0 \\
0 & -1 \end{pmatrix} \mapsto \frac{2}{\langle \beta,\beta \rangle}h_{\beta},
\end{equation}
such that $e_{\beta} \in \mathfrak{g}_{\beta}$, $f_{-\beta} \in \mathfrak{g}_{-\beta}$, and $[e_{\beta},f_{-\beta}] = \frac{2}{\langle \beta,\beta \rangle}h_{\beta}$. From above, since $N_{\beta}^{-}x_{0} \subset U^{-}(P)$, we have
\begin{equation}
\int_{N_{\beta}^{-}x_{0}}{\bf{\Omega}}_{\alpha} = \int_{\mathbbm{C}} u_{-\beta}^{\ast}({\bf{\Omega}}_{\alpha}|_{U^{-}(P)}) =  \int_{\mathbbm{C}} u_{-\beta}^{\ast} \circ (\pi \circ s_{U})^{\ast}{\bf{\Omega}}_{\alpha} = \frac{\sqrt{-1}}{2\pi}\int_{\mathbbm{C}}\partial \overline{\partial} \log ||u_{-\beta}(z)v_{\varpi_{\alpha}}^{+}||^{2}.
\end{equation}
From the Iwasawa decomposition of ${\rm{SL}}_{2}(\mathbbm{C})$, we have 
\begin{equation}
u_{-\beta}(z) = \phi_{\beta}(k)\phi_{\beta} \begin{pmatrix}(1+|z|^{2})^{\frac{1}{2}} & 0 \\
0 & (1+|z|^{2})^{-\frac{1}{2}}
\end{pmatrix}\phi_{\beta} \begin{pmatrix} 1 & z \\
0 & 1 \end{pmatrix},
\end{equation}
such that $k \in {\rm{SU}}(2)$, so $\phi_{\beta}(k) \in G$. Since $||\cdot||$ is obtained from some $G$-invariant inner product on $V(\varpi_{\alpha})$, and $\mathfrak{n}^{+} \cdot v_{\varpi_{\alpha}}^{+} = \{ 0\}$, we have
\begin{equation}
\int_{\mathbbm{P}_{\beta}^{1}}{\bf{\Omega}}_{\alpha} = \frac{\sqrt{-1}}{2\pi}\int_{\mathbbm{C}}\partial \overline{\partial} \log \Bigg | \Bigg| (\vartheta_{\varpi_{\alpha}} \circ \phi_{\beta}) \begin{pmatrix}(1+|z|^{2})^{\frac{1}{2}} & 0 \\
0 & (1+|z|^{2})^{-\frac{1}{2}}
\end{pmatrix}\Bigg | \Bigg|^{2}.
\end{equation}
Observing that $(\vartheta_{\varpi_{\alpha}} \circ \phi_{\beta})|_{T_{\mathbbm{C}}^{1}}$ is a character of the complex torus $ T_{\mathbbm{C}}^{1} \subset {\rm{SL}}_{2}(\mathbbm{C})$, such that 
\begin{equation}
T_{\mathbbm{C}}^{1} = \Bigg \{ \begin{pmatrix}a & 0 \\
0 & \ \ \ a^{-1}
\end{pmatrix} \in {\rm{SL}}_{2}(\mathbbm{C}) \ \ \Bigg | \ \ a \in \mathbbm{C}^{\times}\Bigg\},
\end{equation}
it follows that
\begin{equation}
(\vartheta_{\varpi_{\alpha}} \circ \phi_{\beta})\begin{pmatrix}a & 0 \\
0 & \ \ \ a^{-1}
\end{pmatrix} = a^{({\rm{d}}\vartheta_{\varpi_{\alpha}})_{e}(\phi_{\beta \ast}\begin{pmatrix}1 & \ \ 0 \\
0 & -1 \end{pmatrix})} = a^{\langle \varpi_{\alpha},\beta^{\vee}\rangle}, \ \ \forall a \in \mathbbm{C}^{\times}.
\end{equation}
Thus, we have
\begin{equation}
\int_{\mathbbm{P}_{\beta}^{1}}{\bf{\Omega}}_{\alpha} =   \frac{\sqrt{-1} \langle \varpi_{\alpha},\beta^{\vee} \rangle}{2\pi}\int_{\mathbbm{C}}\partial \overline{\partial} \log (1+|z|^{2}) = \frac{\sqrt{-1} \langle \varpi_{\alpha},\beta^{\vee} \rangle}{2\pi}\int_{\mathbbm{C}}\frac{{\rm{d}z \wedge {\rm{d}}\overline{z}}}{(1+|z|^{2})^{2}}.
\end{equation}
Considering polar coordinates $(r,\theta)$ on $\mathbbm{C}$, we conclude that 
\begin{equation}
\int_{\mathbbm{P}_{\beta}^{1}}{\bf{\Omega}}_{\alpha} = \frac{\langle \varpi_{\alpha},\beta^{\vee} \rangle}{\pi}\int_{0}^{\infty}\int_{0}^{2\pi} \frac{r{\rm{d}}r \wedge {\rm{d}}\theta}{(1+r^{2})^{2}} = \langle \varpi_{\alpha},\beta^{\vee} \rangle \int_{1}^{\infty} \frac{{\rm{d}}u}{u^{2}} = \langle \varpi_{\alpha},\beta^{\vee} \rangle.
\end{equation}
\end{proof}

 From the ideas described above we have the following result.

\begin{proposition}
\label{C8S8.2Sub8.2.3P8.2.6}
Let $X_{P}$ be a complex flag variety associated with some parabolic Lie subgroup $P = P_{I}$. Then, we have
\begin{equation}
\label{picardeq}
{\text{Pic}}(X_{P}) = H^{1,1}(X_{P},\mathbbm{Z}) = H^{2}(X_{P},\mathbbm{Z}) = \displaystyle \bigoplus_{\alpha \in \Delta \backslash I}\mathbbm{Z}[{\bf{\Omega}}_{\alpha} ].
\end{equation}
\end{proposition}
\begin{proof}

In order to introduce some ideas and some notations, let us sketch the proof. The last equality on the right-hand side of Eq. (\ref{picardeq}) follows from the following facts:

\begin{itemize}

\item[(i)] $\pi_{2}(X_{P}) \cong \pi_{1}(T(\Delta \backslash I)^{\mathbbm{C}}) = \mathbbm{Z}^{|\Delta \backslash I|}$, where $T(\Delta \backslash I)^{\mathbbm{C}}$ is given as in Remark \ref{parabolicdec};

\item[(ii)] Since $X_{P}$ is simply connected, it follows that $H_{2}(X_{P},\mathbbm{Z}) \cong \pi_{2}(X_{P})$ (Hurewicz's theorem);

\item[(iii)] By taking the (Schubert) curves $\mathbbm{P}_{\beta}^{1} \hookrightarrow X_{P}$, such that 
\begin{equation}
\label{Scurve}
\mathbbm{P}_{\alpha}^{1} = \overline{\exp(\mathfrak{g}_{-\beta})x_{0}} \subset X_{P},
\end{equation}
for all $\beta \in \Delta \backslash I$, where $x_{0} = eP \in X_{P}$, it follows from Lemma \ref{funddynkinline} that 
\begin{equation}
\big \langle c_{1}(\mathscr{O}_{\alpha}(1)), [ \mathbbm{P}_{\beta}^{1}] \big \rangle = \displaystyle \int_{\mathbbm{P}_{\beta}^{1}} {\bf{\Omega}}_{\alpha} = \langle \varpi_{\alpha},\beta^{\vee} \rangle = \delta_{\alpha \beta},
\end{equation}
for every $\alpha,\beta \in \Delta \backslash I$. Hence, we obtain
\begin{center}

$\pi_{2}(X_{P}) = \displaystyle \bigoplus_{\alpha \in \Delta \backslash I} \mathbbm{Z} [ \mathbbm{P}_{\alpha}^{1}],$ \ \ and \ \ $H^{2}(X_{P},\mathbbm{Z}) = \displaystyle \bigoplus_{\alpha \in \Delta \backslash I}  \mathbbm{Z} c_{1}(\mathscr{O}_{\alpha}(1))$.
\end{center}
\end{itemize}
Since $c_{1}(\mathscr{O}_{\alpha}(1)) \in H^{1,1}(X_{P},\mathbbm{Z}), \forall \alpha \in \Delta \backslash I$, it follows that $H^{1,1}(X_{P},\mathbbm{Z}) = H^{2}(X_{P},\mathbbm{Z})$. In order to conclude the proof, from the Lefschetz theorem on (1,1)-classes \cite{MR2093043}, and from the fact that ${\text{rk}}({\text{Pic}}^{0}(X_{P})) = 0$, we obtain ${\text{Pic}}(X_{P}) = H^{1,1}(X_{P},\mathbbm{Z})$.
\end{proof}

\begin{remark}
\label{dykinlineroot}
Combining the above result with Lemma \ref{funddynkinline}, we obtain the following
\begin{equation}
[\mathbbm{P}_{\beta}^{1}] = \sum_{\alpha \in \Delta \backslash I}\langle \varpi_{\alpha},\beta^{\vee} \rangle [\mathbbm{P}_{\alpha}^{1}] \ \ \ (\forall \beta \in \Phi_{I}^{+}).
\end{equation}
\end{remark}

\begin{remark}
\label{remark1_1}
In the above setting, we consider the weights of $P = P_{I}$ as being  
\begin{center}
$\displaystyle \Lambda_{P} := \bigoplus_{\alpha \in \Delta \backslash I}\mathbbm{Z}\varpi_{\alpha}$. 
\end{center}
From this, the previous result provides $\Lambda_{P} \cong {\rm{Hom}}(P,\mathbbm{C}^{\times}) \cong {\rm{Pic}}(X_{P})$, such that 
\begin{enumerate}
\item$ \displaystyle \lambda = \sum_{\alpha \in \Delta \backslash I}k_{\alpha}\varpi_{\alpha} \mapsto \prod_{\alpha \in \Delta \backslash I} \vartheta_{\varpi_{\alpha}}^{k_{\alpha}} \mapsto \bigotimes_{\alpha \in \Delta \backslash I} \mathscr{O}_{\alpha}(k_{\alpha})$;
\item $ \displaystyle {\bf{E}} \mapsto \vartheta_{{\bf{E}}}: = \prod_{\alpha \in \Delta \backslash I} \vartheta_{\varpi_{\alpha}}^{\langle c_{1}({\bf{L}}),[\mathbbm{P}^{1}_{\alpha}] \rangle} \mapsto \lambda({\bf{E}}) := \sum_{\alpha \in \Delta \backslash I}\langle c_{1}({\bf{E}}),[\mathbbm{P}^{1}_{\alpha}] \rangle\varpi_{\alpha}$.
\end{enumerate}
Thus, $\forall {\bf{E}} \in {\rm{Pic}}(X_{P})$, we have $\lambda({\bf{E}}) \in \Lambda_{P}$. More generally, $\forall \xi \in H^{1,1}(X_{P},\mathbbm{R})$, we have $\lambda (\xi) \in \Lambda_{P}\otimes \mathbbm{R}$, such that
\begin{equation}
\label{weightcohomology}
\lambda(\xi) := \sum_{\alpha \in \Delta \backslash I}\langle \xi,[\mathbbm{P}^{1}_{\alpha}] \rangle\varpi_{\alpha}.
\end{equation}
From above, for every holomorphic vector bundle ${\bf{E}} \to X_{P}$, we define $\lambda({\bf{E}}) \in \Lambda_{P}$, such that 
\begin{equation}
\label{weightholomorphicvec}
\lambda({\bf{E}}) := \sum_{\alpha \in \Delta \backslash I} \langle c_{1}({\bf{E}}),[\mathbbm{P}_{\alpha}^{1}] \rangle \varpi_{\alpha},
\end{equation}
where $c_{1}({\bf{E}}) = c_{1}(\bigwedge^{r}{\bf{E}})$, such that $r = \rank({\bf{E}})$.
\end{remark}

\begin{remark}[Harmonic 2-forms on $X_{P}$]Given any $G$-invariant Riemannian metric $g$ on $X_{P}$, let us denote by $\mathscr{H}^{2}(X_{P},g)$ the space of real harmonic 2-forms on $X_{P}$ with respect to $g$, and by $\mathscr{I}_{G}^{1,1}(X_{P})$ the space of closed invariant real $(1,1)$-forms. Combining the result of Proposition \ref{C8S8.2Sub8.2.3P8.2.6} with \cite[Lemma 3.1]{MR528871}, we obtain 
\begin{equation}
\mathscr{I}_{G}^{1,1}(X_{P}) = \mathscr{H}^{2}(X_{P},g). 
\end{equation}
Therefore, the closed $G$-invariant real $(1,1)$-forms described in Theorem \ref{AZADBISWAS} are harmonic with respect to any $G$-invariant Riemannian metric on $X_{P}$.
\end{remark}

\begin{remark}[K\"{a}hler cone of $X_{P}$]It follows from Eq. (\ref{picardeq}) and Theorem \ref{AZADBISWAS} that the K\"{a}hler cone of a complex flag variety $X_{P}$ is given explicitly by
\begin{equation}
\mathcal{K}(X_{P}) = \displaystyle \bigoplus_{\alpha \in \Delta \backslash I} \mathbbm{R}^{+}[ {\bf{\Omega}}_{\alpha}].
\end{equation}
\end{remark}

\begin{remark}[Cone of curves of $X_{P}$] It is worth observing that the cone of curves ${\rm{NE}}(X_{P})$ of a flag variety $X_{P}$ is generated by the rational curves $[\mathbbm{P}_{\alpha}^{1}] \in \pi_{2}(X_{P})$, $\alpha \in \Delta \backslash I$, see for instance \cite[\S 18.3]{Timashev} and references therein.
\end{remark}

\begin{proposition}
\label{eigenvalueatorigin}
Let $X_{P}$ be a flag variety and let $\omega_{0}$ be a $G$-invariant K\"{a}hler metric on $X_{P}$. Then, for every closed $G$-invariant real $(1,1)$-form $\psi$, the eigenvalues of the endomorphism $\omega_{0}^{-1} \circ \psi$ are given by 
\begin{equation}
\label{eigenvalues}
{\bf{q}}_{\beta}(\omega_{0}^{-1} \circ \psi) = \frac{ \langle \lambda([\psi]), \beta^{\vee} \rangle}{\langle \lambda([\omega_{0}]), \beta^{\vee} \rangle}, \ \ \beta \in \Phi_{I}^{+}.
\end{equation}

\end{proposition}
\begin{proof}
Given a closed $G$-invariant real $(1,1)$-form $\omega \in \Omega^{1,1}(X_{P})^{G}$, it follows from Theorem \ref{AZADBISWAS} that $\omega = \omega_{\varphi}$, such that 
\begin{center}
$\varphi(g) = \displaystyle \sum_{\alpha \in \Delta \backslash I}c_{\alpha}\log \big (||gv_{\varpi_{\alpha}}^{+}|| \big )$, \ \ \ \ $(\forall g \in G^\mathbbm{C})$
\end{center}
where $c_{\alpha} \in \mathbbm{R}$, $\forall \alpha \in \Delta \backslash I$. Let ${\mathcal{H}}_{\varphi}$ be the Hermitian form induced by $\omega_{\varphi}$ on the holomorphic tangent bundle $T^{1,0}X_{P}$, that is, 
\begin{center}
${\mathcal{H}}_{\varphi}(Y,Y) := -\sqrt{-1}\omega_{\varphi}(Y,\overline{Y})$,
\end{center}
for all $Y\in T^{1,0}X_{P}$. Considering the underlying canonical complex structure $J_{0}$ on $X_{P}$, a straightforward computation shows that 
\begin{center}
$\omega_{\varphi}(v,J_{0}v) = {\mathcal{H}}_{\varphi}\big (\frac{1}{2}(v - \sqrt{-1}J_{0}v),\frac{1}{2}(v - \sqrt{-1}J_{0}v) \big ),$ 
\end{center}
for all $\forall v \in TX_{P}$. By considering the coordinate neighborhood $U^{-}(P) \subset X_{P}$ of $x_{0} \in X_{P}$ defined by the opposite big cell (see Eq. (\ref{bigcell})), we obtain a suitable basis for $T_{x_{0}}^{1,0}X_{P}$, given by $Y^{\ast}_{\beta} = \frac{\partial}{\partial z}|_{z = 0}\exp(zy_{-\beta})x_{0}$, $\beta \in \Phi_{I}^{+}$. The vectors $Y_{\beta}^{\ast}$, $\beta \in \Phi_{I}^{+}$, are orthogonal relative to any $(T^{\mathbbm{C}} \cap G)$-invariant Hermitian form. Moreover, we have 
\begin{equation}
{\mathcal{H}}_{\varphi}(Y_{\beta}^{\ast},Y_{\beta}^{\ast}) = \sum_{\alpha \in \Delta \backslash I } \frac{c_{\alpha}}{2} \langle \varpi_{\alpha},\beta^{\vee} \rangle,
\end{equation}
for every $\beta \in \Phi_{I}^{+}$, see for instance \cite{AZAD}. From this, considering the associated dual basis $\theta_{\beta}$, $\beta \in \Phi_{I}^{+}$, at $x_{0} \in X_{P}$, we have 
\begin{equation}
\omega_{\varphi} = \sum_{\beta \in \Phi_{I}^{+}} \frac{\sqrt{-1}}{2} \Bigg (  \sum_{\alpha \in \Delta \backslash I } \frac{c_{\alpha}}{2} \langle \varpi_{\alpha},\beta^{\vee} \rangle\Bigg) \theta_{\beta} \wedge \overline{\theta_{\beta}}.
\end{equation}
On the other hand, from Proposition \ref{picardeq}, a straightforward computation shows that 
\begin{equation}
\omega_{\varphi} = \sum_{\alpha \in \Delta \backslash I} \pi c_{\alpha} {\bf{\Omega}}_{\alpha},
\end{equation}
such that $ c_{\alpha} =  \frac{\langle [\omega_{\varphi}],[\mathbbm{P}_{\alpha}^{1}] \rangle}{\pi}$, $ \forall \alpha \in  \Delta \backslash I$. Joining the above descriptions, we conclude that the eigenvalues of the endomorphism associated with $\omega_{\varphi}$ at $x_{0} \in X_{P}$ are given by 
\begin{equation}
\label{eigenvalue}
{\bf{q}}_{\beta}(\omega_{\varphi}) = - \sum_{\alpha \in \Delta \backslash I } \frac{\big \langle [\omega_{\varphi}],[\mathbbm{P}_{\alpha}^{1}] \big \rangle}{4\pi \sqrt{-1}} \langle \varpi_{\alpha},\beta^{\vee} \rangle = -\frac{\langle\lambda([\omega_{\varphi}]),\beta^{\vee}\rangle}{4\pi \sqrt{-1}} ,
\end{equation}
for every $\beta \in \Phi_{I}^{+}$, such that $\lambda([\omega_{\varphi}])\in \Lambda_{P}\otimes \mathbbm{R}$ (see Eq (\ref{weightcohomology})). Since $\omega_{\varphi}$ is $G$-invariant, it follows that the eigenvalues of the endomorphism $\omega_{\varphi} \colon T^{1,0}X_{P} \to T^{1,0}X_{P}$ are constant. Therefore, given some $G$-invariant K\"{a}hler metric $\omega_{0}$, for every closed $G$-invariant real $(1,1)$-form $\psi$, it follows that the eigenvalues of the endomorphism $\omega_{0}^{-1} \circ \psi \colon T^{1,0}X_{P} \to T^{1,0}X_{P}$ are given by
\begin{equation}
\label{eigenvalue2}
{\bf{q}}_{\beta}(\omega_{0}^{-1} \circ \psi) = \frac{{\bf{q}}_{\beta}(\psi)}{{\bf{q}}_{\beta}(\omega_{0})} = \frac{\langle \lambda([\psi]), \beta^{\vee} \rangle}{\langle \lambda([\omega_{0}]), \beta^{\vee} \rangle} ,
\end{equation}
for every $\beta \in \Phi_{I}^{+}$, which concludes the proof.
\end{proof}

\begin{remark}
\label{primitivecalc}
In the setting of the last proposition, since $n \psi\wedge \omega_{0}^{n-1} = \Lambda_{\omega_{0}}(\psi)\omega_{0}^{n}$, such that $n=\dim_{\mathbbm{C}}(X_{P})$, and $\Lambda_{\omega_{0}}(\psi)={\rm{tr}}(\omega_{0}^{-1} \circ \psi)$, it follows that
\begin{equation}
\label{contractiongenerators}
\Lambda_{\omega_{0}}({\bf{\Omega}}_{\alpha})=\sum_{\beta \in \Phi_{I}^{+}} \frac{\langle \varpi_{\alpha}, \beta^{\vee} \rangle}{\langle \lambda([\omega_{0}]), \beta^{\vee}\rangle},
\end{equation}
for every $\alpha \in \Delta \backslash I$. In particular, for every ${\bf{E}} \in {\rm{Pic}}(X_{P})$, we have a Hermitian structure ${\bf{h}}$, such that the curvature $F_{\nabla}$ of the Chern connection $\nabla \myeq {\rm{d}} + \partial \log ({\bf{h}})$, satisfies 
\begin{equation}
\frac{\sqrt{-1}}{2\pi} \Lambda_{\omega_{0}}(F_{\nabla}) = \sum_{\beta \in \Phi_{I}^{+} } \frac{\langle \lambda({\bf{E}}), \beta^{\vee} \rangle}{\langle \lambda([\omega_{0}]), \beta^{\vee}\rangle}.
\end{equation}
From this, we have that $\nabla$ is a Hermitian-Yang-Mills (HYM) connection. Notice that 
\begin{equation}
c_{1}({\bf{E}}) = \sum_{\alpha \in \Delta \backslash I}\langle \lambda({\bf{E}}), \alpha^{\vee} \rangle [{\bf{\Omega}}_{\alpha}],
\end{equation}
for every ${\bf{E}} \in {\rm{Pic}}(X_{P})$
\end{remark}

\begin{remark}
\label{origineigenvalues}
In the setting of the proof of Proposition \ref{eigenvalueatorigin}, if we consider 
\begin{equation}
\zeta_{\beta} := \sqrt{\frac{\langle \lambda([\omega_{0}]), \beta^{\vee}\rangle}{2\pi}} \theta_{\beta}, \ \ \beta \in \Phi_{I}^{+},
\end{equation}
we obtain obtain from the previous result the following description
\begin{equation}
 \omega_{0} = \sum_{\beta \in \Phi_{I}^{+}} \frac{\sqrt{-1}}{2} \zeta_{\beta} \wedge \overline{\zeta_{\beta}} \ \ \ \ {\text{and}} \ \ \ \ \displaystyle \psi = \sum_{\beta \in \Phi_{I}^{+}} \frac{\sqrt{-1}}{2} {\bf{q}}_{\beta}(\omega_{0}^{-1} \circ \psi)\zeta_{\beta} \wedge \overline{\zeta_{\beta}},
\end{equation}
 for every closed $G$-invariant real $(1,1)$-form $\psi \in \Omega^{1,1}(X_{P})$.
\end{remark}

\subsection{The first Chern class of flag varieties} In this subsection, we shall review some basic facts related with the Ricci form of $G$-invariant K\"{a}hler metrics on flag varieties. Let $X_{P}$ be a complex flag variety associated with some parabolic Lie subgroup $P = P_{I} \subset G^{\mathbbm{C}}$. By considering the identification $T_{x_{0}}^{1,0}X_{P} \cong \mathfrak{m} \subset \mathfrak{g}^{\mathbbm{C}}$, such that 

\begin{center}
$\mathfrak{m} = \displaystyle \sum_{\alpha \in \Phi_{I}^{-}} \mathfrak{g}_{\alpha}$,
\end{center}
 we can realize $T^{1,0}X_{P}$ as being a holomoprphic vector bundle, associated with the $P$-principal bundle $P \hookrightarrow G^{\mathbbm{C}} \to X_{P}$, given by

\begin{center}

$T^{1,0}X_{P} = \Big \{(U_{i})_{i \in J}, \underline{{\rm{Ad}}}\circ \psi_{i j} \colon U_{i} \cap U_{j} \to {\rm{GL}}(\mathfrak{m}) \Big \}$,

\end{center}
where $\underline{{\rm{Ad}}} \colon P \to {\rm{GL}}(\mathfrak{m})$ is the isotropy representation. From this, we obtain 
\begin{equation}
\label{canonicalbundleflag}
{\bf{K}}_{X_{P}}^{-1} = \det \big(T^{1,0}X_{P} \big) = \Big \{(U_{i})_{i \in J}, \det (\underline{{\rm{Ad}}}\circ \psi_{i j}) \colon U_{i} \cap U_{j} \to \mathbbm{C}^{\times} \Big \}.
\end{equation}
Since $P= [P,P] T(\Delta \backslash I)^{\mathbbm{C}}$, regarding $\det \circ \underline{{\rm{Ad}}} \in {\text{Hom}}(T(\Delta \backslash I)^{\mathbbm{C}},\mathbbm{C}^{\times})$, we have 
\begin{equation}
\det \underline{{\rm{Ad}}}(\exp({\bf{t}})) = {\rm{e}}^{{\rm{tr}}({\rm{ad}}({\bf{t}})|_{\mathfrak{m}})} = {\rm{e}}^{- \langle \delta_{P},{\bf{t}}\rangle },
\end{equation}
$\forall {\bf{t}} \in {\rm{Lie}}(T(\Delta \backslash I)^{\mathbbm{C}})$, such that $\delta_{P} = \sum_{\alpha \in \Phi_{I}^{+} } \alpha$. Denoting $\vartheta_{\delta_{P}}^{-1} = \det \circ \underline{{\rm{Ad}}}$, it follows that 
\begin{equation}
\label{charactercanonical}
\vartheta_{\delta_{P}} = \displaystyle \prod_{\alpha \in \Delta \backslash I} \vartheta_{\varpi_{\alpha}}^{\langle \delta_{P},\alpha^{\vee} \rangle} \Longrightarrow {\bf{K}}_{X_{P}}^{-1} = \bigotimes_{\alpha \in \Delta \backslash I}\mathscr{O}_{\alpha}(\ell_{\alpha}),
\end{equation}
such that $\ell_{\alpha} = \langle \delta_{P}, \alpha^{\vee} \rangle, \forall \alpha \in \Delta \backslash I$. Notice that $\lambda({\bf{K}}_{X_{P}}^{-1}) = \delta_{P}$, see Eq. (\ref{weightholomorphicvec}). If we consider the invariant K\"{a}hler metric $\rho_{0} \in \Omega^{1,1}(X_{P})^{G}$, describe by
\begin{equation}
\label{riccinorm}
\rho_{0} = \sum_{\alpha \in \Delta \backslash I}2 \pi \langle \delta_{P}, \alpha^{\vee} \rangle {\bf{\Omega}}_{\alpha},
\end{equation}
it follows that
\begin{equation}
\label{ChernFlag}
c_{1}(X_{P}) = \Big [ \frac{\rho_{0}}{2\pi}\Big].
\end{equation}
By the uniqueness of $G$-invariant representative of $c_{1}(X_{P})$, we have 
\begin{center}
${\rm{Ric}}(\rho_{0}) = \rho_{0}$, 
\end{center}
i.e., $\rho_{0} \in \Omega^{1,1}(X_{P})^{G}$ defines a $G$-ivariant K\"{a}hler-Einstein metric on $X_{P}$ (cf. \cite{MATSUSHIMA}). 
\begin{remark}
Given any $G$-invariant K\"{a}hler metric $\omega$ on $X_{P}$, we have ${\rm{Ric}}(\omega) = \rho_{0}$. Thus, it follows that the smooth function $\frac{\det(\omega)}{\det(\rho_{0})}$ is constant. From this, we obtain
\begin{equation}
{\rm{Vol}}(X_{P},\omega) = \frac{1}{n!}\int_{X_{P}}\omega^{n}  =  \frac{\det(\rho_{0}^{-1} \circ \omega)}{n!} \int_{X_{P}}\rho_{0}^{n}.
\end{equation}
Since $\det(\rho_{0}^{-1} \circ \omega) = \frac{1}{(2\pi)^{n}} \prod_{\beta \in \Phi_{I}^{+}}\frac{\langle \lambda([\omega]),\beta^{\vee} \rangle}{\langle \delta_{P},\beta^{\vee} \rangle}$ and $\frac{1}{n!}\int_{X_{P}}c_{1}(X_{P})^{n} = \prod_{\beta \in \Phi_{I}^{+}} \frac{\langle \delta_{P},\beta^{\vee} \rangle}{\langle \varrho^{+},\beta^{\vee}\rangle}$, we conclude that\footnote{cf. \cite{AZAD}.} 
\begin{equation}
\label{VolKahler}
{\rm{Vol}}(X_{P},\omega) = \prod_{\beta \in \Phi_{I}^{+}} \frac{\langle \lambda([\omega]),\beta^{\vee} \rangle}{\langle \varrho^{+},\beta^{\vee} \rangle},
\end{equation}
where $\varrho^{+} = \frac{1}{2} \sum_{\alpha \in \Phi^{+}}\alpha$. Combining the above formula with the ideas introduced in Remark \ref{primitivecalc} we obtain the following expression for the degree of a holomorphic vector bundle ${\bf{E}} \to X_{P}$ with respect to some $G$-invariant K\"{a}hler metric $\omega$ on $X_{P}$:
\begin{equation}
\label{degreeVB}
\deg_{\omega}({\bf{E}}) = \int_{X_{P}}c_{1}({\bf{E}}) \wedge [\omega]^{n-1} = (n-1)!\Bigg [\sum_{\beta \in \Phi_{I}^{+} } \frac{\langle \lambda({\bf{E}}), \beta^{\vee} \rangle}{\langle \lambda([\omega]), \beta^{\vee}\rangle}\Bigg ] \Bigg [\prod_{\beta \in \Phi_{I}^{+}} \frac{\langle \lambda([\omega]),\beta^{\vee} \rangle}{\langle \varrho^{+},\beta^{\vee} \rangle} \Bigg ],
\end{equation}
such that $\lambda({\bf{E}}) \in \Lambda_{P}$, and $\lambda([\omega]) \in \Lambda_{P} \otimes \mathbbm{R}$. It is worth to point out that, as far as the author is aware about, the explicit formula for the degree of a holomorphic vector bundle over complex flag variety provided above is new in the literature. For more details and further results, see the proof of Theorem \ref{TheoA1}.

\end{remark}

\subsection{Schubert cycles and divisors} 
\label{divisorsandcycles}
The aim of this subsection is to recall some general well-known facts on Schubert cycles and their relationship with divisors and line bundles. The details about the facts covered in this subsection can be found in \cite{BernsteinGelfand}, \cite{FultonWoodward}, \cite{Brion}, \cite{Popov} see also \cite[\S 17 and \S 18]{Timashev}. 

Following the notation of the previous sections, for every $\alpha \in \Phi^{+}$, consider the root reflection $r_{\alpha} \colon \mathfrak{h}^{\ast} \to \mathfrak{h}^{\ast}$, defined by
\begin{equation}
r_{\alpha}(\phi) = \phi - \langle \phi,h_{\alpha}^{\vee} \rangle \alpha, \ \ \ \ \forall \phi \in \mathfrak{h}^{\ast}.
\end{equation}
From above, the Weyl group associated with the root system $\Phi$ is defined by 
\begin{center}
$\mathscr{W} = \big \langle r_{\alpha} \ | \ \alpha \in \Delta \big \rangle$.
\end{center}
Under the identification $\mathscr{W} \cong N_{G^{\mathbbm{C}}}(T^{\mathbbm{C}})/T^{\mathbbm{C}}$, by abuse of notation, for any $w \in \mathscr{W}$, we still denote by $w \in G^{\mathbbm{C}}$ one of its representative in $G^{\mathbbm{C}}$. Given a parabolic subgroup $P = P_{I} \subset G^{\mathbbm{C}}$, we denote by $\mathscr{W}_{P}$ the subgroup of $\mathscr{W}$ generated by the reflections $r_{\alpha}$, $\alpha \in I$, and by $\mathscr{W}^{P}$ the quotient $\mathscr{W}/\mathscr{W}_{P}$. Also, we identify $\mathscr{W}^{P}$ with the set of minimal length representatives in $\mathscr{W}$. By considering the $B$-orbit $Bwx_{0} \subset X_{P}$ (Bruhat cell), for every $w \in \mathscr{W}^{P}$, we have a cellular decomposition for $X_{P}$ given by
\begin{equation}
X_{P} = \coprod_{w \in \mathscr{W}^{P}}Bwx_{0}, \ \ \ ({\text{Bruhat decomposition}})
\end{equation}
In the above decomposition we have $Bwx_{0} \cong \mathbbm{C}^{\ell(w)}$, for every $w \in \mathscr{W}^{P}$, where $\ell(w)$ is the length\footnote{$\ell(w)$ denotes the length of a reduced (i.e. minimal) decomposition of $w$ as a product of simple reflections, e.g. \cite{Humphreys}.} of $w \in \mathscr{W}^{P}$. The Schubert varieties are defined by the closure of the above cells; we denote them by $X_{P}(w) = \overline{Bwx_{0}}$, $\forall w \in \mathscr{W}^{P}$. Notice that $\mathbbm{P}_{\alpha}^{1} = X_{P}(r_{\alpha})$, $\forall \alpha \in \Delta \backslash I$ (cf. Eq. (\ref{Scurve})). Similarly, we let $Y_{P}(w) = \overline{B^{-}wx_{0}}$ be the opposite Schubert variety associated with $w \in \mathscr{W}^{P}$; it is a variety of codimension $\ell(w)$, and denoting by $w_{0} \in \mathscr{W}$ the element of maximal length, it follows that $Y_{P}(w) = w_{0}X_{P}(w_{0}w)$, for all $w \in \mathscr{W}^{P}$. For the sake of simplicity, we shall denote $w^{\vee} := w_{0}w$, for all $w \in \mathscr{W}^{P}$. The irreducible $B$-stable divisors of $X_{P}$ are the Schubert varieties of codimension 1 (Schubert divisors). We shall denote them by
\begin{equation}
D_{\alpha} := X_{P}(r_{\alpha}^{\vee}) = w_{0}Y_{P}(r_{\alpha}), \ \ \ \ \forall \alpha \in \Delta \backslash I.
\end{equation}
By considering the homology cycles induced, respectively, by the Schubert varieties and opposite Schubert varieties, it follows that 
\begin{center}
$[X_{P}(w)] \in H_{2\ell(w)}(X_{P}, \mathbbm{Z})$ \ \ \ \ \ and \ \ \ \ \ $[Y_{P}(w)] = [X_{P}(w^{\vee})] \in H_{2n-2\ell(w)}(X_{P}, \mathbbm{Z})$,
\end{center}
$\forall w \in \mathscr{W}^{P}$. From above, we have the following well-known facts:
\begin{enumerate}
\item[(i)] $H_{\bullet}(X_{P},\mathbbm{Z}) = \bigoplus_{w \in \mathscr{W}^{P}}\mathbbm{Z}[X_{P}(w)]$;

\item[(ii)] If $w,u \in \mathscr{W}^{P}$, such that $\ell(w) = \ell(u)$, then $[Y_{P}(w)] \cdot [X_{P}(u)] = \delta_{wu}$;

\item[(iii)] Considering $[Y_{P}(w)] \in {\rm{Hom}}(H_{2\ell(w)}(X_{P}, \mathbbm{Z});\mathbbm{Z})$, $\forall w \in \mathscr{W}^{P}$, we have 
\begin{equation}
H^{\bullet}(X_{P},\mathbbm{Z}) = \bigoplus_{w \in \mathscr{W}^{P}}\mathbbm{Z}[Y_{P}(w)];
\end{equation}
\item[(iv)] The divisor group ${\rm{Div}}(X_{P}) = H^{0}(X_{P},\mathcal{M}_{X_{P}}^{\ast}/\mathcal{O}_{X_{P}}^{\ast})$ is freely generated by the Schubert divisors;

\item[(v)]Under the map $\mathcal{O} \colon {\rm{Div}}(X_{P}) \to {\rm{Pic}}(X_{P})$, $D \mapsto \mathcal{O}(D)$, we have $\mathcal{O}(D_{\alpha}) = \mathscr{O}_{\alpha}(1)$, $\forall \alpha \in \Delta \backslash I$;
\item[(vi)] Considering the divisor class group\footnote{The symbol ``$\sim$" stands for linear equivalence. Notice that, since $H^{2}(X_{P},\mathbbm{Z})$ is torsion-free, from Lefschetz theorem on $(1,1)$-classes we have that numerically equivalent divisors are in fact linearly equivalent, see for instance \cite{Lazarsfeld}.} ${\rm{Cl}}(X_{P}) = {\rm{Div}}(X_{P})/\sim$, it follows that 
\begin{equation}
{\rm{Cl}}(X_{P}) = \bigoplus_{\alpha \in \Delta \backslash I}\mathbbm{Z}[D_{\alpha}].
\end{equation}
\end{enumerate}
\begin{remark}
From above, given $[D] \in {\rm{Cl}}(X_{P})$, we have $D \sim \sum_{\alpha \in \Delta \backslash I}(D \cdot \mathbbm{P}_{\alpha}^{1})D_{\alpha}$, where $(D \cdot \mathbbm{P}_{\alpha}^{1}) := [D] \cdot [\mathbbm{P}_{\alpha}^{1}], \forall \alpha \in  \Delta \backslash I$. Thus, we obtain a group isomorphism ${\text{Hom}}(P,\mathbbm{C}^{\times}) \cong {\rm{Cl}}(X_{P})$, such that
\begin{equation}
\label{characterdivisor}
\vartheta \mapsto  [D_{\vartheta}] :=  \sum_{\alpha \in \Delta \backslash I}\langle \vartheta,\alpha^{\vee}\rangle[D_{\alpha}] , \ \ \ \ [D]  \mapsto \vartheta_{D} := \prod_{\alpha \in \Delta \backslash I} \vartheta_{\varpi_{\alpha}}^{(D \cdot \mathbbm{P}_{\alpha}^{1})},
\end{equation}
for all $\vartheta \in {\text{Hom}}(P,\mathbbm{C}^{\times})$, and for all $[D] \in {\rm{Cl}}(X_{P})$, where $\langle \vartheta,\alpha^{\vee}\rangle := \langle ({\rm{d}}\vartheta)_{e},\alpha^{\vee} \rangle$, $\forall \alpha \in \Delta \backslash I$. 

\end{remark}
\begin{remark}
As in the case of holomorphic line bundles, one can attach to each divisor class $[D] \in {\rm{Cl}}(X_{P})$ a weight $\lambda_{D}:= \lambda(\mathcal{O}(D)) \in \Lambda_{P}$. From this, we obtain the following expression for the degree of a divisor $D \in {\rm{Div}}(X_{P})$ with respect to some $G$-invariant K\"{a}hler metric $\omega$ on $X_{P}$:
\begin{equation}
\label{degreeDiv}
\deg_{\omega}(D) = \int_{D}\omega^{n-1} = (n-1)!\Bigg [\sum_{\beta \in \Phi_{I}^{+} } \frac{\langle \lambda_{D}, \beta^{\vee} \rangle}{\langle \lambda([\omega]), \beta^{\vee}\rangle}\Bigg ] \Bigg [\prod_{\beta \in \Phi_{I}^{+}} \frac{\langle \lambda([\omega]),\beta^{\vee} \rangle}{\langle \varrho^{+},\beta^{\vee} \rangle} \Bigg ],
\end{equation}
such that $\lambda([\omega]) \in \Lambda_{P} \otimes \mathbbm{R}$. Notice that, in the particular case that $D \in {\rm{Div}}(X_{P})$ is very ample, one can choose $[\omega] = c_{1}(\mathcal{O}(D))$ so that $\lambda_{D} = \lambda([\omega])$. Thus, in this last case we obtain the well-known formula for the degree of the induced projectve embedding $X_{P} \hookrightarrow \mathbbm{P}(V(\lambda_{D}))$, e.g. \cite[Example 18.13]{Timashev}. It is worth mentioning that, as far as the author is aware about, the explicit formula for the degree of an arbitrary divisor with respect to an arbitrary $G$-invariant K\"{a}hler metric on a flag variety as above is new in the literature, cf. \cite{BorelHizebruch}, \cite[Example 18.13]{Timashev}.
\end{remark}

\section{Proof of main results}
In this section, we shall restate and prove the main results presented in the introduction. Let us start by proving Theorem \ref{theoremA}.

\begin{theorem}
\label{ProofTheo1}
Given a  K\"{a}hler class $[\omega] \in \mathcal{K}(X_{P})$, then for every $[\psi] \in H^{1,1}(X_{P},\mathbbm{R})$ we have 
\begin{equation}
\label{liftphase}
\hat{\Theta}: = {\rm{Arg}} \int_{X_{P}}\frac{(\omega + \sqrt{-1}\psi)^{n}}{n!} = \sum_{\beta \in \Phi_{I}^{+}} \arctan \bigg( \frac{\langle \lambda([\psi]),\beta^{\vee} \rangle}{\langle \lambda([\omega]),\beta^{\vee} \rangle}\bigg) \mod 2\pi,
\end{equation}
such that $\lambda([\psi]), \lambda([\omega_{0}]) \in \Lambda_{P} \otimes \mathbbm{R}$. In particular, fixed the unique $G$-invariant representative $\omega_{0} \in [\omega]$, there exists $\phi \in C^{\infty}(X_{P})$, such that $\chi_{\phi}:= \psi + \sqrt{-1}\partial \overline \partial \phi$ satisfies the deformed Hermitian Yang-Mills equation
\begin{equation}
\label{DHYMeqTeo}
{\rm{Im}}\big ( \omega_{0} + \sqrt{-1}\chi_{\phi}\big)^{n} = \tan(\hat{\Theta}) {\rm{Re}}\big ( \omega_{0} + \sqrt{-1}\chi_{\phi}\big)^{n}.
\end{equation}
\end{theorem}

\begin{proof}
Since $\hat{\Theta}$ depends only on the cohomology classes $[\omega]$ and $[\psi]$, we can consider $G$-invariant representatives $\omega_{0} \in [\omega]$ and $\chi \in [\psi]$. From this, at $x_{0} = eP \in X_{P}$, we have 
\begin{equation}
 \omega_{0} = \sum_{\beta \in \Phi_{I}^{+}} \frac{\sqrt{-1}}{2} \zeta_{\beta} \wedge \overline{\zeta_{\beta}} \ \ \ \ {\text{and}} \ \ \ \ \displaystyle \chi = \sum_{\beta \in \Phi_{I}^{+}} \frac{\sqrt{-1}}{2} {\bf{q}}_{\beta}(\omega_{0}^{-1} \circ \chi)\zeta_{\beta} \wedge \overline{\zeta_{\beta}},
\end{equation}
see Remark \ref{origineigenvalues}. Hence, we obtain
\begin{equation}
\label{lagrangianphase}
(\omega_{0} + \sqrt{-1}\chi)^{n} = \prod_{\beta \in \Phi_{I}^{+}} \Big ( 1 + \sqrt{-1}{\bf{q}}_{\beta}(\omega_{0}^{-1} \circ \chi)\Big )\omega_{0}^{n} = r_{\omega_{0}}(\chi){\rm{e}}^{\sqrt{-1}\Theta_{\omega_{0}}(\chi)}\omega_{0}^{n},
\end{equation}
such that 
\begin{enumerate}
\item[(a)] $\displaystyle r_{\omega_{0}}(\chi) = \prod_{\beta \in \Phi_{I}^{+}}\sqrt{\Big(1 + {\bf{q}}_{\beta}(\omega_{0}^{-1} \circ \chi)^{2}\Big )}$;
\item[(b)] $\displaystyle \Theta_{\omega_{0}}(\chi) = \sum_{\beta \in \Phi_{I}^{+}} \arctan \Big({\bf{q}}_{\beta}(\omega_{0}^{-1} \circ \chi)\Big ) \mod 2\pi.$
\end{enumerate}
Since the eigenvalues ${\bf{q}}_{\beta}(\chi)$ and ${\bf{q}}_{\beta}(\omega_{0})$ are constant, $\forall \beta \in \Phi_{I}^{+}$, we conclude that 
\begin{equation}
\int_{X_{P}}\frac{(\omega_{0} + \sqrt{-1}\chi)^{n}}{n!} =  r_{\omega_{0}}(\chi){\rm{e}}^{\sqrt{-1}\Theta_{\omega_{0}}(\chi)}{\rm{Vol}}(X_{P},\omega_{0})\neq 0.
\end{equation}
Therefore, we have
\begin{equation}
{\rm{Arg}} \int_{X_{P}}\frac{(\omega + \sqrt{-1}\psi)^{n}}{n!} = \sum_{\beta \in \Phi_{I}^{+}} \arctan \bigg( \frac{\langle \lambda([\chi]),\beta^{\vee} \rangle}{\langle \lambda([\omega_{0}]),\beta^{\vee} \rangle}\bigg) \mod 2\pi.
\end{equation}
Since $\lambda([\chi]) = \lambda([\psi])$ and $\lambda([\omega_{0}]) = \lambda([\omega])$, we obtain the desired equality. For the second part, we just need to observe that $\chi = \psi + \sqrt{-1}\partial \overline{\partial} \phi$, for some $\phi \in C^{\infty}(X_{P})$. Since Eq. (\ref{DHYMeqTeo}) is equivalent to $\Theta_{\omega_{0}}(\chi) = \hat{\Theta} \mod 2\pi$, we conclude the proof.
\end{proof}

By keeping the notation of the above result, it follows that 
\begin{equation}
{\rm{Re}} \big ( {\rm{e}}^{\sqrt{-1} \hat{\Theta}}(\omega_{0} + \sqrt{-1} \chi)^{n} \big) = \bigg (\prod_{\beta \in \Phi_{I}^{+}}\sqrt{\Big(1 + {\bf{q}}_{\beta}(\omega_{0}^{-1} \circ \chi)^{2}\Big )} \bigg) \omega_{0}^{n} > 0,
\end{equation}
where $\chi \in [\psi]$ is the unique $G$-invariant representative. Hence, following \cite{Collins2018deformed, Collins2020stability}, we obtain the following corollary.
\begin{corollary}
In the setting of the previous theorem, given $[\psi] \in H^{1,1}(X_{P},\mathbbm{R})$, considering $\eta_{\phi} = \psi + \sqrt{-1}\partial \overline{\partial} \phi$, for each $\phi \in C^{\infty}(X_{P})$, we have that the space of the almost calibrated $(1,1)$-forms 
\begin{equation}
\mathcal{H} := \big \{ \phi \in C^{\infty}(X_{P}) \ \big | \ {\rm{Re}} \big ( {\rm{e}}^{\sqrt{-1} \hat{\Theta}}(\omega_{0} + \sqrt{-1} \eta_{\phi})^{n} \big) > 0 \big \},
\end{equation}
is non-empty ($\mathcal{H} \neq \emptyset$). In particular, we have the unique lift $\hat{\Theta}([\psi]) \in (-n\frac{\pi}{2},n\frac{\pi}{2})$ of $\hat{\Theta}$ given by
\begin{equation}
\hat{\Theta}([\psi]):= \sum_{\beta \in \Phi_{I}^{+}} \arctan \bigg( \frac{\langle \lambda([\psi]),\beta^{\vee} \rangle}{\langle \lambda([\omega]),\beta^{\vee} \rangle}\bigg),
\end{equation}
such that $\lambda([\psi]), \lambda([\omega]) \in \Lambda_{P} \otimes \mathbbm{R}$.
\end{corollary}

\begin{remark}
In what follows, for the sake of simplicity, we shall denote $\hat{\Theta} = \hat{\Theta}([\psi])$, $[\psi] \in H^{1,1}(X_{P},\mathbbm{R})$, when the dependence on the cohomology class was not relevant.
\end{remark}

The next result to be proved is Theorem \ref{theoremB}.
\begin{theorem}
\label{T2}
For every irreducible analytic subvariety $Y \subset X_{P}$, define
\begin{equation}
\Theta_{Y} := {\rm{Arg}}\int_{Y}(\omega + \sqrt{-1}\psi)^{\dim(Y)},
\end{equation}
such that $[\omega] \in \mathcal{K}(X_{P})$ and $[\psi] \in H^{1,1}(X_{P},\mathbbm{R})$. If
\begin{equation}
\label{conditioncot}
 \sum_{\beta \in \Phi_{I}^{+}} \arccot \bigg( \frac{\langle \lambda([\psi]),\beta^{\vee} \rangle}{\langle \lambda([\omega]),\beta^{\vee} \rangle}\bigg) < \pi,
\end{equation}
then, for every proper irreducible analytic subvariety $Y \subset X_{P}$, the following holds
\begin{equation}
\Theta_{Y} > \Theta_{X_{P}} - \big (n -\dim(Y) \big ) \frac{\pi}{2}.
\end{equation}
\end{theorem}

\begin{proof}
At first, given $[\omega] \in \mathcal{K}(X_{P})$ and $[\psi] \in H^{1,1}(X_{P},\mathbbm{R})$, we have
\begin{equation}
\label{equivcot}
\sum_{\beta \in \Phi_{I}^{+}} \arccot \bigg( \frac{\langle \lambda([\psi]),\beta^{\vee} \rangle}{\langle \lambda([\omega]),\beta^{\vee} \rangle}\bigg) < \pi \iff \underbrace{n\frac{\pi}{2} - \sum_{\beta \in \Phi_{I}^{+}} \arctan \bigg( \frac{\langle \lambda([\psi]),\beta^{\vee} \rangle}{\langle \lambda([\omega]),\beta^{\vee} \rangle}\bigg)}_{= n\frac{\pi}{2} - \hat{\Theta} } < \pi.
\end{equation}
Given a proper irreducible analytic subvariety $Y \subset X_{P}$, with $\dim(Y) = p$, considering the positive current associated with the fundamental class $[Y]\in H^{n-p,n-p}(X_{P},\mathbbm{Z})$, it follows that\footnote{For more details, see \cite[Chapter 4]{Griffiths}, \cite[Chapter III]{Demailly_cplx}, and \cite[Chapter 3, \S 14]{Chirka_complex}.}
\begin{center}
$\displaystyle \int_{Y}(\omega + \sqrt{-1}\psi)^{p} := \int_{Y_{\text{reg}}}(\omega + \sqrt{-1}\psi)^{p} = \int_{X_{P}}[Y_{\text{reg}}]\wedge (\omega + \sqrt{-1}\psi)^{p}$.
\end{center}
Since the above integral depends only on the cohomology classes involved, we can consider $G$-invariant representatives $\omega_{0} \in [\omega]$, $\chi \in [\psi]$, and $T_{Y} \in [Y_{\text{reg}}]$. Let $\zeta_{1},\ldots,\zeta_{n}$, be an orthonormal coframe w.r.t. $\omega_{0}$ of $T^{\ast}_{x_{0}}X_{P}$, where $x_{0} = eP$, such that 
\begin{center}
$\displaystyle \omega_{0} = \sum_{j = 1}^{n} \frac{\sqrt{-1}}{2} \zeta_{j} \wedge \overline{\zeta_{j}}$ \ \ and \ \ $\displaystyle \chi = \sum_{j = 1}^{n} \frac{\sqrt{-1}}{2} {\bf{q}}_{j}\zeta_{j} \wedge \overline{\zeta_{j}},$
\end{center}
where ${\bf{q}}_{j} = \frac{\langle \lambda([\chi]),\beta_{j}^{\vee} \rangle }{\langle \lambda([\omega_{0}]),\beta_{j}^{\vee} \rangle}$, $\forall j = 1,\ldots,n$, are the eigenvalues of $\omega_{0}^{-1} \circ \chi$ at $x_{0}$, see for instance Remark \ref{origineigenvalues}. From this, we can represented $T_{Y}$ at $x_{0} = eP$ by
\begin{center}
$\displaystyle T_{Y} = \frac{(\sqrt{-1})^{(n-p)^{2}}}{2^{(n-p)}}\sum_{|A| = |B| = n-p}T_{A,B}(x_{0})\zeta_{A} \wedge \overline{\zeta_{B}}$,
\end{center}
where $\zeta_{A} = \zeta_{a_{1}} \wedge \cdots \wedge \zeta_{a_{n-p}}$, $\forall A = \{a_{1} < \cdots <a_{n-p}\}$. From above, we obtain at $x_{0}$ the following 
\begin{equation}
T_{Y}\wedge (\omega_{0} + \sqrt{-1}\chi)^{p} = \frac{(\sqrt{-1})^{p^{2}}}{2^{p}}\sum_{|J| = p}p! \prod_{j \in J}(1+\sqrt{-1}{\bf{q}}_{j})T_{Y} \wedge \zeta_{J} \wedge \overline{\zeta_{J}},
\end{equation}
see for instance \cite[Chapter III]{Demailly_cplx}. Since 
\begin{equation}
 \frac{(\sqrt{-1})^{p^{2}}}{2^{p}} T_{Y} \wedge  \zeta_{J} \wedge \overline{\zeta_{J}} = T_{J^{c},J^{c}}(x_{0}) \frac{(\sqrt{-1})^{n}}{2^{n}} \zeta_{1} \wedge \overline{\zeta_{1}} \wedge \cdots \wedge \zeta_{n} \wedge \overline{\zeta_{n}} = \frac{1}{n!}T_{J^{c},J^{c}}(x_{0}) \omega_{0}^{n},
\end{equation}
where $J^{c} \subset \{1,\ldots, n\}$ is the ordered complementary multi-indices of $J$, we conclude that 
\begin{equation}
T_{Y}\wedge (\omega_{0} + \sqrt{-1}\chi)^{p} = \sum_{|J| = p}p! \prod_{j \in J}(1+\sqrt{-1}{\bf{q}}_{j})\frac{T_{J^{c},J^{c}}(x_{0})}{n!} \omega_{0}^{n},
\end{equation}
at $x_{0} = eP$. Considering\footnote{Notice that we can identify $\Theta_{J}$ with its lift $\sum_{j \in J} \arctan({\bf{q}}_{j})$, $J = \{j_{1} < \cdots < j_{p}\}$.}
\begin{equation}
\Theta_{J} := {\rm{Arg}} \prod_{j \in J}(1+\sqrt{-1}{\bf{q}}_{j}) \ \ {\text{and}} \ \ r_{J}:= \big |\prod_{j \in J}(1+\sqrt{-1}{\bf{q}}_{j}) \big |,
\end{equation}
$\forall J = \{j_{1} < \cdots < j_{p}\}$, we obtain
\begin{equation}
\label{expressionorigin}
T_{Y}\wedge(\omega_{0} + \sqrt{-1}\chi)^{p} =  \sum_{|J| = p}p!r_{J}{\rm{e}}^{\sqrt{-1}\Theta_{J} }\frac{T_{J^{c},J^{c}}(x_{0})}{n!} \omega_{0}^{n},
\end{equation}
at $x_{0} = eP$. From the $G$-invariance\footnote{$(T_{Y} \wedge (\omega_{0} + \sqrt{-1}\chi)^{p})_{gx_{0}} = (g^{-1})^{\ast}\big ((T_{Y}\wedge (\omega_{0} + \sqrt{-1}\chi)^{p})_{x_{0}}\big), \forall g \in G$.} of $T_{Y}\wedge(\omega_{0} + \sqrt{-1}\chi)^{p}$ and Eq. (\ref{expressionorigin}), we have
\begin{equation}
\label{intrelcomp}
 \frac{\displaystyle {\rm{e}}^{\sqrt{-1}(n-p)\frac{\pi}{2}} \int_{X_{P}}T_{Y}\wedge(\omega_{0} + \sqrt{-1}\chi)^{p}}{\displaystyle  \int_{X_{P}}(\omega_{0} + \sqrt{-1}\chi)^{n}} =  \sum_{|J| = p}T_{J^{c},J^{c}}(x_{0})R_{J}{\rm{e}}^{\sqrt{-1}(\Theta_{J} - \hat{\Theta} + (n-p)\frac{\pi}{2})},
\end{equation}
such that 
\begin{equation}
R_{J}:= \frac{p! r_{J}}{n!r_{\omega_{0}}(\chi)},
\end{equation}
$\forall J = \{j_{1} < \cdots < j_{p}\}$, see Eq. (\ref{lagrangianphase}). Notice that, since $[Y_{\text{reg}}]$ is a positive current, we have $T_{J^{c},J^{c}}(x_{0}) \geq 0$, for all ordered sets $J$. Moreover, we can suppose that at least one $T_{J^{c},J^{c}}(x_{0})$ is strictly positive for some $J$. In fact, from Wirtinger's theorem and the Poincar\'{e} duality, it follows that\footnote{Notice that $T_{Y}\wedge \omega_{0}^{p} = p!\Big (\sum_{|J| = p} T_{J^{c},J^{c}}(x_{0})\Big )\frac{\omega_{0}^{n}}{n!}$, e.g. \cite[Chapter III]{Demailly_cplx}.}
\begin{center}
$\displaystyle {\rm{Vol}}(Y) = \frac{1}{p!}\int_{X_{P}}[Y_{\text{reg}}]\wedge [\omega_{0}]^{p} = \frac{1}{p!}\int_{X_{P}}T_{Y} \wedge \omega_{0}^{p} = \int_{X_{P}}\Big ( \sum_{|J| = p} T_{J^{c},J^{c}}(x_{0})\Big)\frac{\omega_{0}^{n}}{n!},$
\end{center}
see for instance \cite[Chapter 3, \S 13]{Chirka_complex} and \cite[Chapter 4]{Griffiths}, in other words, we have
\begin{equation}
0 \neq {\rm{Vol}}(Y) = \Big ( \sum_{|J| = p} T_{J^{c},J^{c}}(x_{0})\Big) {\rm{Vol}}(X_{P},\omega_{0}).
\end{equation}
Further, since $\Theta_{J} =  \sum_{i \in J} \arctan({\bf{q}}_{i}) < \frac{p\pi}{2}$, it follows that  
\begin{equation}
\label{ineqpbase}
\hat{\Theta} = \Theta_{J} + \sum_{i \notin J} \arctan({\bf{q}}_{i})< \Theta_{J} + (n-p)\frac{\pi}{2} \iff  \hat{\Theta} - (n-p)\frac{\pi}{2} < \Theta_{J} < \frac{p\pi}{2}
\end{equation}
$\forall J = \{j_{1} < \cdots < j_{p}\}$. Therefore, if Eq. (\ref{conditioncot}) holds, from Eq. (\ref{equivcot}), we have 
\begin{equation}
0 < \Theta_{J} - \hat{\Theta} + (n-p)\frac{\pi}{2} < n\frac{\pi}{2} - \hat{\Theta} < \pi.
\end{equation}
Combining the above fact with Eq. (\ref{intrelcomp}) and Eq. (\ref{ineqpbase}), we conclude that
\begin{equation}
{\rm{Im}} \Bigg (\frac{{\rm{e}}^{\sqrt{-1}(n-p)\frac{\pi}{2}} \int_{Y}(\omega_{0} + \sqrt{-1}\chi)^{p}}{\int_{X_{P}}(\omega_{0} + \sqrt{-1}\chi)^{n}}\Bigg ) > 0.
\end{equation}
From this, we have
\begin{equation}
{\rm{Arg}}  \Bigg ( \frac{{\rm{e}}^{\sqrt{-1}(n-p)\frac{\pi}{2}} \int_{Y} (\omega_{0} + \sqrt{-1}\chi)^{p}}{\int_{X_{P}}(\omega_{0} + \sqrt{-1}\chi)^{n}}\Bigg ) > 0,
\end{equation}
hence $\Theta_{Y} > \hat{\Theta} - (n-p)\frac{\pi}{2}$. 
\end{proof}

Considering $\mathbbm{P}_{\beta}^{1} = \overline{\exp(\mathfrak{g}_{-\beta})P} \subset X_{P}$, $\beta \in \Phi_{I}^{+}$, we have 
\begin{equation}
[\mathbbm{P}_{\beta}^{1}] = \sum_{\alpha \in \Delta \backslash I}\langle \varpi_{\alpha},\beta^{\vee} \rangle [\mathbbm{P}_{\alpha}^{1}] \ \ \ (\forall \beta \in \Phi_{I}^{+}),
\end{equation}
see Remark \ref{dykinlineroot}. Thus, we obtain
\begin{equation}
\int_{{\mathbbm{P}}^{1}_{\beta}}\big (\omega + \sqrt{-1}\psi \big ) = \sum_{\alpha \in \Delta \backslash I}\langle \varpi_{\alpha},\beta^{\vee} \rangle\int_{{\mathbbm{P}}^{1}_{\alpha}}\big (\omega + \sqrt{-1}\psi \big ) = \langle \lambda([\omega]),\beta^{\vee} \rangle + \sqrt{-1} \langle \lambda([\psi]),\beta^{\vee}\rangle,
\end{equation}
so
\begin{equation}
\Theta_{{\mathbbm{P}}^{1}_{\beta}} = {\rm{Arg}}\int_{{\mathbbm{P}}^{1}_{\beta}}(\omega + \sqrt{-1}\psi) = \arctan \bigg ( \frac{\langle \lambda([\psi]),\beta^{\vee} \rangle}{\langle \lambda([\omega]),\beta^{\vee} \rangle }\bigg),
\end{equation}
for every $\beta \in \Phi_{I}^{+}$. Therefore, we have the following corollary.
\begin{corollary}
In the setting of the previous theorem, the unique lifted angle $\hat{\Theta} \in (-n\frac{\pi}{2},n\frac{\pi}{2})$ of $[\psi] \in H^{1,1}(X_{P},\mathbbm{R})$ is given by
\begin{equation}
\hat{\Theta} = \sum_{\beta \in \Phi_{I}^{+}} \Theta_{{\mathbbm{P}}^{1}_{\beta}},
\end{equation}
such that $\mathbbm{P}_{\beta}^{1} = \overline{\exp(\mathfrak{g}_{-\beta})P} \subset X_{P}$, $\forall \beta \in \Phi_{I}^{+}$.
\end{corollary}

Let us now restate and prove Theorem \ref{theoremC}.

\begin{theorem}
\label{TheoA1}
Let ${\bf{E}} \to X_{P}$ be a holomorphic vector bundle with $\rank({\bf{E}}) = r$. Then, the slope of ${\bf{E}}$ with respect to a K\"{a}hler class $\xi \in \mathcal{K}(X_{P})$ is given by
\begin{equation}
\mu_{\xi}({\bf{E}}) = \frac{(n-1)!}{r} \Bigg [ \sum_{\beta \in \Phi_{I}^{+} } \frac{\langle \lambda({\bf{E}}), \beta^{\vee} \rangle}{\langle \lambda(\xi), \beta^{\vee}\rangle} \Bigg ] \Bigg [ \prod_{\beta \in \Phi_{I}^{+}} \frac{\langle \lambda(\xi),\beta^{\vee} \rangle}{\langle \varrho^{+},\beta^{\vee} \rangle}\Bigg],
\end{equation}
such that $\lambda({\bf{E}}) \in \Lambda_{P}$, and $\lambda(\xi) \in \Lambda_{P} \otimes \mathbbm{R}$. Moreover, if the restriction of ${\bf{E}}$ to each generator $[{\mathbbm{P}}^{1}_{\alpha}]$, $\alpha \in \Delta \backslash I$, of the cone of curves ${\rm{NE}}(X_{P})$ is semistable, then ${\bf{E}}$ is $\xi$-semistable with respect to any K\"{a}hler class $\xi \in \mathcal{K}(X_{P})$.
\end{theorem}

\begin{proof}
By definition, given a K\"{a}hler class $\xi = [\omega]$, we have 
\begin{equation}
\mu_{[\omega]}({\bf{E}}) = \frac{\int_{X_{P}}c_{1}({\bf{E}}) \wedge [\omega]^{n-1}}{r} = \frac{\deg_{\omega}({\bf{E}})}{r}.
\end{equation}
Since $c_{1}({\bf{E}}) = \sum_{\alpha \in \Delta \backslash I} \langle c_{1}({\bf{E}}),[\mathbbm{P}_{\alpha}^{1}] \rangle [{\bf{\Omega_{\alpha}}}]$, we obtain the following 
\begin{center}
$\displaystyle\mu_{[\omega]}({\bf{E}}) = \sum_{\alpha \in \Delta \backslash I} \frac{\langle c_{1}({\bf{E}}),[\mathbbm{P}_{\alpha}^{1}] \rangle}{r}\int_{X_{P}}{\bf{\Omega}}_{\alpha} \wedge \omega^{n-1} = \sum_{\alpha \in \Delta \backslash I} \frac{\langle c_{1}({\bf{E}}),[\mathbbm{P}_{\alpha}^{1}] \rangle}{nr} \Lambda_{\omega_{0}}({\bf{\Omega}}_{\alpha})\int_{X_{P}}\omega^{n}.$
\end{center}
From Eq. (\ref{contractiongenerators}) and Eq. (\ref{VolKahler}), we have
\begin{equation}
\mu_{[\omega]}({\bf{E}}) = \sum_{\alpha \in \Delta \backslash I} \frac{\langle c_{1}({\bf{E}}),[\mathbbm{P}_{\alpha}^{1}] \rangle}{nr} \sum_{\beta \in \Phi_{I}^{+}} \frac{\langle \varpi_{\alpha}, \beta^{\vee} \rangle}{\langle \lambda([\omega]), \beta^{\vee}\rangle} n !\Bigg [\prod_{\beta \in \Phi_{I}^{+}} \frac{\langle \lambda([\omega]),\beta^{\vee} \rangle}{\langle \varrho^{+},\beta^{\vee} \rangle}\Bigg].
\end{equation}
Observing that $\lambda({\bf{E}}) = \sum_{\alpha \in \Delta \backslash I} \langle c_{1}({\bf{E}}),[\mathbbm{P}_{\alpha}^{1}] \rangle \varpi_{\alpha}$, rearranging the above expression, we conclude that 
\begin{equation}
\mu_{[\omega]}({\bf{E}}) = \frac{(n-1)!}{r} \Bigg [ \sum_{\beta \in \Phi_{I}^{+} } \frac{\langle \lambda({\bf{E}}), \beta^{\vee} \rangle}{\langle \lambda([\omega]), \beta^{\vee}\rangle} \Bigg ] \Bigg [ \prod_{\beta \in \Phi_{I}^{+}} \frac{\langle \lambda([\omega]),\beta^{\vee} \rangle}{\langle \varrho^{+},\beta^{\vee} \rangle}\Bigg].
\end{equation}
Suppose now that the restriction of ${\bf{E}}$ to each generator $[{\mathbbm{P}}^{1}_{\alpha}]$, $\alpha \in \Delta \backslash I$, of the cone of curves ${\rm{NE}}(X_{P})$ is semistable. Considering the inclusion $\iota_{\alpha} \colon {\mathbbm{P}}^{1}_{\alpha} \hookrightarrow X_{P}$, $\alpha \in \Delta \backslash I$, for every subbundle $0 \neq \mathcal{F} \varsubsetneq {\bf{E}}$, $\rank(\mathcal{F}) = s < r$, we have 
\begin{center}
$\displaystyle \frac{c_{1}(\iota_{\alpha}^{\ast}{\bf{E}})}{r} \geq  \frac{c_{1}(\iota_{\alpha}^{\ast}\mathcal{F})}{s} \Rightarrow \frac{1}{r}\int_{\mathbbm{P}_{\alpha}^{1}}c_{1}({\bf{E}}) \geq \frac{1}{s}\int_{\mathbbm{P}_{\alpha}^{1}}c_{1}(\mathcal{F}) \Rightarrow  \frac{\langle c_{1}({\bf{E}}),[\mathbbm{P}_{\alpha}^{1}] \rangle}{r} \geq  \frac{\langle c_{1}(\mathcal{F}),[\mathbbm{P}_{\alpha}^{1}] \rangle}{s},$
\end{center}
for every $\alpha \in \Delta \backslash I$. Since $\langle \varpi_{\alpha},\beta^{\vee} \rangle \geq 0$, $\forall \alpha \in \Delta \backslash I$ and $\forall \beta \in \Phi_{I}^{+}$, it follows that
\begin{center}
$\displaystyle \frac{\langle \lambda({\bf{E}}), \beta^{\vee} \rangle}{r} - \frac{\langle \lambda(\mathcal{F}), \beta^{\vee} \rangle}{s} = \sum_{\alpha \in \Delta \backslash I}\Bigg (\frac{\langle c_{1}({\bf{E}}),[\mathbbm{P}_{\alpha}^{1}] \rangle}{r} - \frac{\langle c_{1}(\mathcal{F}),[\mathbbm{P}_{\alpha}^{1}] \rangle}{s} \Bigg )\langle \varpi_{\alpha},\beta^{\vee} \rangle \geq 0$.
\end{center}
From above, given any $G$-invariant K\"{a}hler metric $\omega$ on $X_{P}$, since $\langle \lambda([\omega]), \beta^{\vee}\rangle > 0$, $\forall \beta \in \Phi_{I}^{+}$, we conclude that 
\begin{equation}
\frac{1}{s} \Bigg [\sum_{\beta \in \Phi_{I}^{+}}\frac{\langle \lambda(\mathcal{F}), \beta^{\vee} \rangle}{\langle \lambda([\omega]), \beta^{\vee}\rangle} \Bigg ] \leq \frac{1}{r} \Bigg [\sum_{\beta \in \Phi_{I}^{+}} \frac{\langle \lambda({\bf{E}}), \beta^{\vee} \rangle}{\langle \lambda([\omega]), \beta^{\vee}\rangle} \Bigg ] \Rightarrow \mu_{[\omega]}(\mathcal{F}) \leq \mu_{[\omega]}({\bf{E}}).
\end{equation}
Therefore, ${\bf{E}}$ is $[\omega]$-semistable. Since every K\"{a}hler class admits a $G$-invariant representative, we obtain the desired result.
\end{proof}

Now we are in position to prove Theorem \ref{theoremD}.

\begin{theorem}
\label{TheoA2}
Fixed a $G$-invariant K\"{a}hler metric $\omega_{0}$ on $X_{P}$, for every holomorphic vector bundle ${\bf{E}} \to X_{P}$ and every irreducible analytic subvariety $Y \subset X_{P}$, define
\begin{equation}
Z_{[\omega_{0}]}({\bf{E}},Y):= - \int_{Y}{\rm{e}}^{-\sqrt{-1}[\omega_{0}]}{\rm{ch}}({\bf{E}}).
\end{equation}
Then, the following hold:
\begin{enumerate}
\item In the particular case that ${\bf{E}} \in {\rm{Pic}}(X_{P})$ and $Y \in {\rm{Div}}(X_{P})$, we have
\begin{equation}
Z_{[\omega_{0}]}({\bf{E}},Y) = -\sum_{\alpha \in \Phi_{I}^{+}}  \Bigg [ \prod_{\beta \in \Phi_{I}^{+}\backslash \{\alpha\}}\Bigg (\frac{ \langle \lambda({\bf{E}}), \beta^{\vee} \rangle}{\langle \lambda([\omega_{0}]), \beta^{\vee} \rangle}\frac{ \langle \lambda_{Y}, \alpha^{\vee} \rangle}{\langle \lambda([\omega_{0}]), \alpha^{\vee} \rangle} - \sqrt{-1}\frac{ \langle \lambda_{Y}, \alpha^{\vee} \rangle}{\langle \lambda([\omega_{0}]), \alpha^{\vee} \rangle} \Bigg ) \Bigg ]V_{0},
\end{equation}
such that $V_{0} = {\rm{Vol}}(X_{P},\omega_{0})$, $\lambda({\bf{E}}),\lambda_{Y} \in \Lambda_{P}$, and $\lambda([\omega_{0}]) \in \Lambda_{P} \otimes \mathbbm{R}$;
\item For every ${\bf{E}} \in {\rm{Pic}}(X_{P})$, such that $c_{1}({\bf{E}}) \neq 0$, we have
\begin{equation}
\frac{|Z_{[\omega_{0}]}({\bf{E}},X_{P})|}{||{\rm{ch}}(\bf{E})||} > 0,
\end{equation}
where $||\cdot||$ is any norm on the finite dimensional vector space $H^{{\text{even}}}(X_{P},\mathbbm{R})$;
\item For every holomorphic vector bundle ${\bf{E}} \to X_{P}$, define
\begin{equation}
\hat{\mu}_{[\omega_{0}]}({\bf{E}}) := \sum_{\beta \in \Phi_{I}^{+} } \tan \big ( \Theta_{\omega_{0}}({\bf{E}},\mathbbm{P}_{\beta}^{1})\big),
\end{equation}
such that 
\begin{equation}
\Theta_{\omega_{0}}({\bf{E}},\mathbbm{P}_{\beta}^{1}) := {\rm{Arg}} \big (Z_{[\omega_{0}]}({\bf{E}},\mathbbm{P}_{\beta}^{1}) \big) - \frac{\pi}{2}, 
\end{equation}
where $\mathbbm{P}_{\beta}^{1} = \overline{\exp(\mathfrak{g}_{-\beta})P} \subset X_{P}$, $\forall \beta \in \Phi_{I}^{+}$. Then, we have that ${\bf{E}}$ is $[\omega_{0}]$-(semi)stable if, and only if, 
\begin{equation}
\hat{\mu}_{[\omega_{0}]}({\bf{E}}) \geoq \hat{\mu}_{[\omega_{0}]}(\mathcal{F}),
\end{equation}
for every subbundle $0 \neq \mathcal{F} \varsubsetneq {\bf{E}}$;
\item Given ${\bf{E}} \in {\rm{Pic}}(X_{P})$, if
\begin{equation}
\label{hypercharge}
\frac{\pi (n-2)}{2} < \Theta_{\omega_{0}}({\bf{E}}):= \sum_{\beta \in \Phi_{I}^{+} }  \Theta_{\omega_{0}}({\bf{E}},\mathbbm{P}_{\beta}^{1}) < \frac{n\pi}{2},
\end{equation}
then 
\begin{equation}
\label{inequalitycharge}
{\rm{Im}} \bigg( \frac{Z_{[\omega_{0}]}({\bf{E}},Y)}{Z_{[\omega_{0}]}({\bf{E}},X_{P})}\bigg) > 0,
\end{equation}
for every irreducible analytic subvariety $Y \varsubsetneq X_{P}$;
\item Given a holomorphic vector bundle ${\bf{E}} \to X_{P}$, if
\begin{equation}
{\rm{Arg}} \big (Z_{[\omega_{0}]}({\bf{E}},\mathbbm{P}_{\beta}^{1}) \big)  \geoq {\rm{Arg}} \big (Z_{[\omega_{0}]}(\mathcal{F},\mathbbm{P}_{\beta}^{1}) \big), 
\end{equation}
for every subbundle $0 \neq \mathcal{F} \varsubsetneq {\bf{E}}$ and every $\beta \in \Phi_{I}^{+}$, then ${\bf{E}}$ is $[\omega_{0}]$-(semi)stable.
\end{enumerate}

\end{theorem}

\begin{proof}
In order to prove item (1), let us recall some basic facts. Given ${\bf{E}} \in {\rm{Pic}}(X_{P})$, we have
\begin{equation}
{\rm{ch}}({\bf{E}}) = {\rm{e}}^{c_{1}(\bf{E})} = \sum_{k = 0}^{\infty}\frac{c_{1}({\bf{E}})^{k}}{k!}.
\end{equation}
Thus, fixed a $G$-invariant K\"{a}hler metric $\omega_{0}$ on $X_{P}$, given $Y \in {\rm{Div}}(X_{P})$ and $\chi \in c_{1}({\bf{E}})$, we have
\begin{equation}
Z_{[\omega_{0}]}({\bf{E}},Y)= - \int_{Y}{\rm{e}}^{-\sqrt{-1}\big ([\omega_{0}]+\sqrt{-1}c_{1}(\bf{E})\big )} = -\frac{(-\sqrt{-1})^{n-1}}{(n-1)!}\int_{Y}(\omega_{0} + \sqrt{-1}\chi)^{n-1}.
\end{equation}
Since the integral above on the right-hand side depends only on the cohomology classes involved, we can suppose that $\chi \in c_{1}({\bf{L}})$ is also $G$-invariant. Also, let us consider the unique $G$-invariant representative $\omega_{Y} \in c_{1}(\mathcal{O}(D))$. Let $\zeta_{1},\ldots,\zeta_{n}$ be an orthonormal coframe w.r.t. $\omega_{0}$ of $T^{\ast}_{x_{0}}X_{P}$, where $x_{0} = eP$, such that 
\begin{equation}
\omega_{0} = \sum_{\beta \in \Phi_{I}^{+}}\frac{\sqrt{-1}}{2}\zeta_{\beta}\wedge \overline{\zeta_{\beta}}, \ \ \chi = \sum_{\beta \in \Phi_{I}^{+}}\frac{\sqrt{-1}}{2}{\bf{a}}_{\beta}\zeta_{\beta}\wedge \overline{\zeta_{\beta}}, \ \ \omega_{Y} = \sum_{\beta \in \Phi_{I}^{+}}\frac{\sqrt{-1}}{2}{\bf{b}}_{\beta}\zeta_{\beta}\wedge \overline{\zeta_{\beta}},
\end{equation}
where ${\bf{a}}_{\beta} = \frac{\langle \lambda([\chi]),\beta^{\vee} \rangle }{\langle \lambda([\omega_{0}]),\beta^{\vee} \rangle}$, ${\bf{b}}_{\beta} = \frac{\langle \lambda([\omega_{Y}]),\beta^{\vee} \rangle }{\langle \lambda([\omega_{0}]),\beta^{\vee} \rangle}$, $\beta \in \Phi_{I}^{+}$, are, respectively, the eigenvalues of $\omega_{0}^{-1} \circ \chi$ and $\omega_{0}^{-1} \circ \omega_{Y}$ at $x_{0}$ (see Remark \ref{origineigenvalues}). From these data, we obtain at $x_{0} = eP$ the following
\begin{equation}
\omega_{Y} \wedge (\omega_{0} + \sqrt{-1}\chi)^{n-1} =\sum_{\alpha \in \Phi_{I}^{+}}(n-1)! \prod_{\beta \in \Phi_{I}^{+}\backslash \{\alpha\}}(1+\sqrt{-1}{\bf{a}}_{\beta}){\bf{b}}_{\alpha}\frac{\omega_{0}^{n}}{n!}.
\end{equation}
From above, it follows that
\begin{equation}
 -(-\sqrt{-1})^{n-1}\int_{X_{P}}\frac{\omega_{Y} \wedge (\omega_{0} + \sqrt{-1}\chi)^{n-1}}{(n-1)!} = -\sum_{\alpha \in \Phi_{I}^{+}}\prod_{\beta \in \Phi_{I}^{+}\backslash \{\alpha\}}({\bf{a}}_{\beta}{\bf{b}}_{\alpha} - \sqrt{-1}{\bf{b}}_{\alpha})V_{0},
\end{equation}
such that $V_{0} := {\rm{Vol}}(X_{P},\omega_{0})$. Since $\lambda([\chi]) = \lambda({\bf{E}})$ and $\lambda([\omega_{Y}]) = \lambda_{Y}$, we conclude that  
\begin{equation}
Z_{[\omega_{0}]}({\bf{E}},Y) = -\sum_{\alpha \in \Phi_{I}^{+}}  \Bigg [ \prod_{\beta \in \Phi_{I}^{+}\backslash \{\alpha\}}\Bigg (\frac{ \langle \lambda({\bf{E}}), \beta^{\vee} \rangle}{\langle \lambda([\omega_{0}]), \beta^{\vee} \rangle}\frac{ \langle \lambda_{Y}, \alpha^{\vee} \rangle}{\langle \lambda([\omega_{0}]), \alpha^{\vee} \rangle} - \sqrt{-1}\frac{ \langle \lambda_{Y}, \alpha^{\vee} \rangle}{\langle \lambda([\omega_{0}]), \alpha^{\vee} \rangle} \Bigg ) \Bigg ]V_{0}.
\end{equation}

The proof of item (2) goes as follows. Given ${\bf{E}} \in {\rm{Pic}}(X_{P})$, and considering the unique $G$-invariant representative $\chi \in c_{1}({\bf{E}})$, it follows that
\begin{equation}
Z_{[\omega_{0}]}({\bf{E}},X_{P})=  -\frac{(-\sqrt{-1})^{n}}{n!}\int_{X_{P}}(\omega_{0} + \sqrt{-1}\chi)^{n} = -(-\sqrt{-1})^{n}r_{\omega_{0}}(\chi){\rm{e}}^{\sqrt{-1}\Theta_{\omega_{0}}(\chi)}{\rm{Vol}}(X_{P},\omega_{0}),
\end{equation}
see for instance Eq. (\ref{lagrangianphase}). Thus, since
\begin{equation}
r_{\omega_{0}}(\chi) = \prod_{\beta \in \Phi_{I}^{+}}\sqrt{\Big(1 + {\bf{q}}_{\beta}(\omega_{0}^{-1} \circ \chi)^{2}\Big )},
\end{equation}
see for instance Eq. (\ref{eigenvalues}), we have
\begin{equation}
|Z_{[\omega_{0}]}({\bf{E}},X_{P})| = r_{\omega_{0}}(\chi) {\rm{Vol}}(X_{P},\omega_{0}) > 0.
\end{equation}
For item (3), since 
\begin{equation}
{\rm{ch}}({\bf{E}}) = r + c_{1}({\bf{E}}) + \frac{1}{2} \big (c_{1}({\bf{E}})^{2} - 2c_{2}({\bf{E}}) \big ) + \cdots
\end{equation}
for every holomorphic vector bundle ${\bf{E}} \to X_{P}$, with $\rank({\bf{E}}) = r$, e.g. \cite{kobayashi_dif_cplx_vec}, it follows that 
\begin{equation}
Z_{[\omega_{0}]}({\bf{E}},\mathbbm{P}_{\beta}^{1})= -\int_{\mathbbm{P}_{\beta}^{1}} \big ( c_{1}({\bf{E}}) - r\sqrt{-1}[\omega_{0}]\big) = \int_{\mathbbm{P}_{\beta}^{1}} \big ( - c_{1}({\bf{E}}) + r\sqrt{-1}[\omega_{0}]\big),
\end{equation}
for every $\mathbbm{P}_{\beta}^{1} = \overline{\exp(\mathfrak{g}_{-\beta})P} \subset X_{P}$, $\beta \in \Phi_{I}^{+}$. Since 
\begin{equation}
[\mathbbm{P}_{\beta}^{1}] = \sum_{\alpha \in \Delta \backslash I}\langle \varpi_{\alpha},\beta^{\vee} \rangle [\mathbbm{P}_{\alpha}^{1}] \ \ \ (\forall \beta \in \Phi_{I}^{+}),
\end{equation}
see Remark \ref{dykinlineroot}, we obtain 
\begin{equation}
\label{chargecurve}
Z_{[\omega_{0}]}({\bf{E}},\mathbbm{P}_{\beta}^{1}) = \sum_{\alpha \in \Delta \backslash I}\langle \varpi_{\alpha},\beta^{\vee} \rangle\int_{\mathbbm{P}_{\alpha}^{1}}\big ( - c_{1}({\bf{E}}) + r\sqrt{-1}[\omega_{0}]\big) = - \langle \lambda({\bf{E}}),\beta^{\vee} \rangle + r\sqrt{-1}\langle \lambda([\omega_{0}]),\beta^{\vee} \rangle, 
\end{equation}
for every $\mathbbm{P}_{\beta}^{1} = \overline{\exp(\mathfrak{g}_{-\beta})P} \subset X_{P}$, $\beta \in \Phi_{I}^{+}$. Therefore, we have 
\begin{equation}
\label{phaseofbundle}
{\rm{Arg}} \big ( Z_{[\omega_{0}]}({\bf{E}},\mathbbm{P}_{\beta}^{1})\big ) = \underbrace{\arctan \bigg ( \frac{\langle \lambda({\bf{E}}),\beta^{\vee} \rangle}{r \langle \lambda([\omega_{0}]),\beta^{\vee} \rangle} \bigg )}_{\Theta_{\omega_{0}}({\bf{E}},\mathbbm{P}_{\beta}^{1})} + \frac{\pi}{2},
\end{equation}
for every $\mathbbm{P}_{\beta}^{1} = \overline{\exp(\mathfrak{g}_{-\beta})P} \subset X_{P}$, $\beta \in \Phi_{I}^{+}$. Since 
\begin{equation}
 - \frac{\langle \lambda({\bf{E}}),\beta^{\vee} \rangle}{r \langle \lambda([\omega_{0}]),\beta^{\vee} \rangle} = \frac{{\rm{Re}}\big (Z_{[\omega_{0}]}({\bf{E}},\mathbbm{P}_{\beta}^{1}))\big )}{{\rm{Im}}\big (Z_{[\omega_{0}]}({\bf{E}},\mathbbm{P}_{\beta}^{1}))\big )} = \cot \bigg ( \Theta_{\omega_{0}}({\bf{E}},\mathbbm{P}_{\beta}^{1}) + \frac{\pi}{2}\bigg ) = -  \tan \big ( \Theta_{\omega_{0}}({\bf{E}},\mathbbm{P}_{\beta}^{1}) \big ),
\end{equation}
for every $\beta \in \Phi_{I}^{+}$, from Theorem \ref{TheoA1}, it follows that 
\begin{equation}
\mu_{[\omega_{0}]}({\bf{E}}) = -(n-1)!\Bigg [ \sum_{\beta \in \Phi_{I}^{+} } \tan \big ( \Theta_{\omega_{0}}({\bf{E}},\mathbbm{P}_{\beta}^{1}) \big )\Bigg ] V_{0} = (n-1)!\hat{\mu}_{[\omega_{0}]}({\bf{E}})V_{0},
\end{equation}
such that $V_{0} = {\rm{Vol}}(X_{P},\omega_{0})$. From the above expression we conclude the proof of item (3). In order to prove item (4), we observe the following. Given ${\bf{E}} \in {\rm{Pic}}(X_{P})$, suppose that Eq. (\ref{hypercharge}) holds. Under this last hypothesis, considering $[\psi] = c_{1}({\bf{E}})$, from Eq. (\ref{phaseofbundle}), we have
\begin{equation}
\frac{\pi(n-2)}{2}<\Theta_{\omega_{0}}({\bf{E}}) = \sum_{\beta \in \Phi_{I}^{+}} \arctan \bigg( \frac{\langle \lambda([\psi]),\beta^{\vee} \rangle}{\langle \lambda([\omega_{0}]),\beta^{\vee} \rangle}\bigg) < \frac{n\pi}{2}.
\end{equation}
Thus, from Theorem \ref{T2}, it follows that
\begin{equation}
 \underbrace{{\rm{Arg}} \big ({\rm{e}}^{-\sqrt{-1}(\pi - \dim(Y) \frac{\pi}{2})}Z_{[\omega_{0}]}({\bf{E}},Y)\big)}_{{\rm{Arg}}\int_{Y}(\omega_{0} + \sqrt{-1}\psi)^{\dim(Y)}} > \underbrace{{\rm{Arg}} \big ({\rm{e}}^{-\sqrt{-1}(\pi - n \frac{\pi}{2})}Z_{[\omega_{0}]}({\bf{E}},X_{P})\big)}_{{\rm{Arg}}\int_{X_{P}}(\omega_{0} + \sqrt{-1}\psi)^{n}} - \big (n -\dim(Y) \big ) \frac{\pi}{2},
\end{equation}
for every irreducible analytic subvariety $Y \varsubsetneq X_{P}$. Rearranging the above expressions, we conclude that Eq. (\ref{hypercharge}) implies the following (equivalent) inequalities
\begin{equation}
{\rm{Arg}}\big (Z_{[\omega_{0}]}({\bf{E}},Y)\big)> {\rm{Arg}}\big (Z_{[\omega_{0}]}({\bf{E}},X_{P})\big) \iff {\rm{Im}} \bigg( \frac{Z_{[\omega_{0}]}({\bf{E}},Y)}{Z_{[\omega_{0}]}({\bf{E}},X_{P})}\bigg) > 0,
\end{equation}
for every irreducible analytic subvariety $Y \varsubsetneq X_{P}$. In order to prove item (5), we proceed as follows. Given a holomorphic vector bundle ${\bf{E}} \to (X_{P},\omega_{0}$), such that $r = \rank({\bf{E}})$, suppose that
\begin{equation}
{\rm{Arg}} \big (Z_{[\omega_{0}]}({\bf{E}},\mathbbm{P}_{\beta}^{1}) \big) \geoq {\rm{Arg}} \big (Z_{[\omega_{0}]}(\mathcal{F},\mathbbm{P}_{\beta}^{1}) \big), 
\end{equation}
for every subbundle $0 \neq \mathcal{F} \varsubsetneq {\bf{E}}$ and every $\beta \in \Phi_{I}^{+}$. From Eq. (\ref{phaseofbundle}), we have
\begin{equation}
\Theta_{\omega_{0}}({\bf{E}},\mathbbm{P}_{\beta}^{1}) \geoq \Theta_{\omega_{0}}({\bf{\mathcal{F}}},\mathbbm{P}_{\beta}^{1}),
\end{equation}
for every subbundle $0 \neq \mathcal{F} \varsubsetneq {\bf{E}}$ and every $\beta \in \Phi_{I}^{+}$. Since $\tan(-)$ is strictly increasing on $(-\frac{\pi}{2},\frac{\pi}{2})$, it follows that 
\begin{equation}
\hat{\mu}_{[\omega_{0}]}({\bf{E}}) = \sum_{\beta \in \Phi_{I}^{+} } \tan \big ( \Theta_{\omega_{0}}({\bf{E}},\mathbbm{P}_{\beta}^{1})\big) \geoq \sum_{\beta \in \Phi_{I}^{+} } \tan \big ( \Theta_{\omega_{0}}({\bf{\mathcal{F}}},\mathbbm{P}_{\beta}^{1})\big) = \hat{\mu}_{[\omega_{0}]}(\mathcal{F}), 
\end{equation}
for every subbundle $0 \neq \mathcal{F} \varsubsetneq {\bf{E}}$. Thus, from item (2), it follows that ${\bf{E}}$ is $[\omega_{0}]$-(semi)stable.
\end{proof}

\begin{remark}
\label{additivecentralcharge}
In the setting of Eq. (\ref{phaseofbundle}), since $ -\frac{\pi}{2} < \Theta_{\omega_{0}}({\bf{E}},\mathbbm{P}_{\beta}^{1}) < \frac{\pi}{2}$, $\forall \beta \in \Phi_{I}^{+}$, it follows that 
\begin{equation}
 0 < {\rm{Arg}} \big ( Z_{[\omega_{0}]}({\bf{E}},\mathbbm{P}_{\beta}^{1})\big ) < \pi, \ \ (\forall \beta \in \Phi_{I}^{+}).
\end{equation}
see also Eq. (\ref{chargecurve}). Thus, we have that $Z_{[\omega_{0}]}({\bf{E}},\mathbbm{P}_{\beta}^{1})$ lies in the strict upper half of $\mathbbm{C}$, for every $\beta \in \Phi_{I}^{+}$, and every holomorphic vector bundle ${\bf{E}} \to X_{P}$, such that $\rank({\bf{E}}) > 0$. Considering the associated Grothendieck group of coherent sheaves $K_{0}(X_{P})$, we have that
\begin{center}
$Z_{[\omega_{0}]}(-,\mathbbm{P}_{\beta}^{1}) \colon K_{0}(X_{P}) \to \mathbbm{C}$ 
\end{center}
defines an additive homomorphism explicitly given by
\begin{equation}
Z_{[\omega_{0}]}([\mathcal{E}],\mathbbm{P}_{\beta}^{1}) = -\int_{\mathbbm{P}_{\beta}^{1}} \big ( c_{1}(\det(\mathcal{E})) - \rank(\mathcal{E})\sqrt{-1}[\omega_{0}]\big),
\end{equation}
for every $[\mathcal{E}] \in  K_{0}(X_{P})$ and every $\beta \in \Phi_{I}^{+}$.
\end{remark}

From Kobayashi-Hitchin correspondence \cite{donaldson1985anti,donaldson1987infinite},\cite{uhlenbeck1986existence}, and item (5) of the last theorem, we obtain the following.
\begin{corollary}
Given a $G$-invariant K\"{a}hler metric $\omega_{0}$ on $X_{P}$ and a simple holomorphic vector bundle ${\bf{E}} \to X_{P}$, if
\begin{equation}
{\rm{Arg}} \big (Z_{[\omega_{0}]}({\bf{E}},\mathbbm{P}_{\beta}^{1}) \big) > {\rm{Arg}} \big (Z_{[\omega_{0}]}(\mathcal{F},\mathbbm{P}_{\beta}^{1}) \big), 
\end{equation}
for every subbundle $0 \neq \mathcal{F} \varsubsetneq {\bf{E}}$ and every $\beta \in \Phi_{I}^{+}$, then ${\bf{E}}$ admits a Hermitian metric $H$ solving the Hermitian-Yang-Mills equation
\begin{equation}
\sqrt{-1}\Lambda_{\omega_{0}}F(H) = c \mathbbm{1}_{{\bf{E}}}.
\end{equation}
\end{corollary}

Combining \cite[Proposition 7.11]{kobayashi_dif_cplx_vec} with item (4) of the last theorem, we have the following.

\begin{corollary}
Let ${\bf{E}}$ and ${\bf{F}}$ be $[\omega_{0}]$-semistable holomorphic vector bundles over $X_{P}$, for some $G$-invariant K\"{a}hler metric $\omega_{0}$ on $X_{P}$. If
\begin{equation}
{\rm{Arg}} \big (Z_{[\omega_{0}]}({\bf{E}},\mathbbm{P}_{\beta}^{1}) \big) > {\rm{Arg}} \big (Z_{[\omega_{0}]}({\bf{F}},\mathbbm{P}_{\beta}^{1}) \big), 
\end{equation}
for every $\beta \in \Phi_{I}^{+}$, then ${\rm{Hom}}({\bf{E}},{\bf{F}}) = 0$.
\end{corollary}

\section{Applications and final comments}

In what follows,  in order to illustrate our main results, we provide some explicit computations through a detailed and constructive example. 
\begin{example} 
\label{computationWallach}
Consider the complex Lie group $G^{\mathbbm{C}} = {\rm{SL}}_{3}(\mathbbm{C})$. In this case, the structure of the associated Lie algebra $\mathfrak{sl}_{3}(\mathbbm{C})$ can be completely determined by means of its Dynkin diagram
\begin{center}
${\dynkin[labels={\alpha_{1},\alpha_{2}},scale=3]A{oo}} $
\end{center}
Fixed the Cartan subalgebra $\mathfrak{h} \subset \mathfrak{sl}_{3}(\mathbbm{C})$ of diagonal matrices, we have the associated simple root system given by $\Delta  = \{\alpha_{1},\alpha_{2}\}$, such that 
\begin{center}
$\alpha_{j}({\rm{diag}}(d_{1},d_{2},d_{3})) = d_{j} - d_{j+1}$, $j = 1,2$.
\end{center}
$\forall {\rm{diag}}(d_{1},d_{2},d_{3}) \in \mathfrak{h}$. The set of positive roots in this case is given by 
\begin{center}
$\Phi^{+} = \{\alpha_{1}, \alpha_{2}, \alpha_{3} = \alpha_{1} + \alpha_{2}\}$. 
\end{center}
Considering the Cartan-Killing form\footnote{In this case, we have $\kappa(X,Y) = 6{\rm{tr}}(XY), \forall X,Y \in \mathfrak{sl}_{3}(\mathbbm{C})$, see for instance \cite[Chapter 10, \S 4]{procesi2007lie}.} $\kappa(X,Y) = {\rm{tr}}({\rm{ad}}(X){\rm{ad}}(Y)), \forall X,Y \in \mathfrak{sl}_{3}(\mathbbm{C})$, it follows that $\alpha_{j} = \kappa(\cdot,h_{\alpha_{j}})$, $j =1,2,3$, such that\footnote{Notice that $\langle \alpha_{j},\alpha_{j} \rangle = \alpha_{j}(h_{\alpha_{j}}) = \frac{1}{3}, \forall j = 1,2,3.$} 
\begin{equation}
h_{\alpha_{1}} =\frac{1}{6}(E_{11} - E_{22}), \ \ h_{\alpha_{2}} =\frac{1}{6}(E_{22} - E_{33}), \ \ h_{\alpha_{3}} =\frac{1}{6}(E_{11} - E_{33}),
\end{equation}
here we consider the matrices $E_{ij}$ as being the elements of the standard basis of ${\mathfrak{gl}}_{3}(\mathbbm{C})$. Moreover, we have the following relation between simple roots and fundametal weights (Fig. \ref{diagram}):
\begin{center}
$\displaystyle{\begin{pmatrix}
\alpha_{1} \\ 
\alpha_{2}
\end{pmatrix} = \begin{pmatrix} \ \ 2 & -1 \\
-1 & \ \ 2\end{pmatrix} \begin{pmatrix}
\varpi_{\alpha_{1}} \\ 
\varpi_{\alpha_{2}}
\end{pmatrix}, \ \ \ \begin{pmatrix}
\varpi_{\alpha_{1}} \\ 
\varpi_{\alpha_{2}}
\end{pmatrix} = \frac{1}{3}\begin{pmatrix} 2 & 1 \\
1 & 2\end{pmatrix} \begin{pmatrix}
\alpha_{1} \\ 
\alpha_{2}
\end{pmatrix}},$
\end{center}
here we consider the Cartan matrix $C = (C_{ij})$ of $\mathfrak{sl}_{3}(\mathbbm{C})$ given by 
\begin{equation}
\label{Cartanmatrix}
C = \begin{pmatrix}
 \ \ 2 & -1 \\
-1 & \ \ 2 
\end{pmatrix}, \ \ C_{ij} = \frac{2\langle \alpha_{i}, \alpha_{j} \rangle}{\langle \alpha_{j}, \alpha_{j} \rangle},
\end{equation}
for more details on the above subject, see for instance \cite{Humphreys}. 
\begin{figure}[H]
\includegraphics[scale = .28]{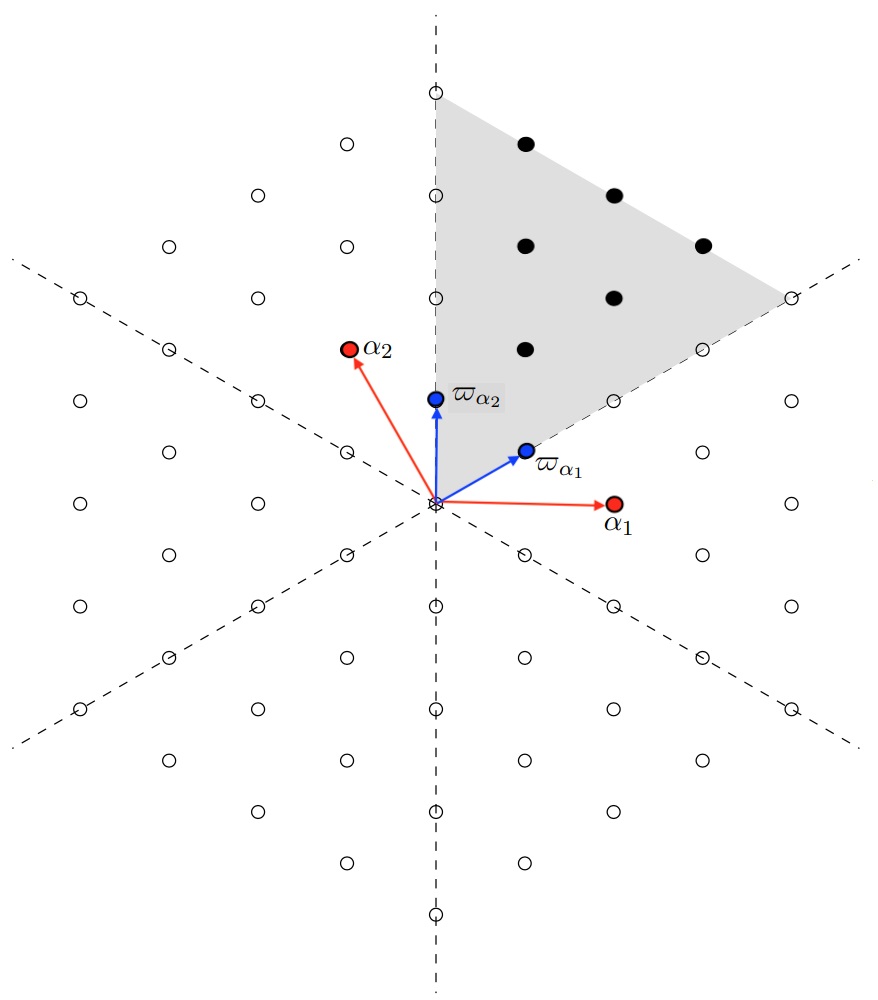}
\caption{Simple roots, dominant integral elements, and fundamental weights for $\mathfrak{sl}_{3}(\mathbbm{C})$.}
\label{diagram}
\end{figure}
Fixed the standard Borel subgroup $B \subset {\rm{SL}}_{3}(\mathbbm{C})$, i.e.,
\begin{center}
$B = \Bigg \{ \begin{pmatrix} \ast & \ast & \ast \\
0 & \ast & \ast \\
0 & 0 & \ast \end{pmatrix} \in {\rm{SL}}_{3}(\mathbbm{C})\Bigg\},$
\end{center}
we consider the flag variety obtained from $I = \emptyset$, i.e., the homogeneous Fano threefold given by the Wallach space ${\mathbbm{P}}(T_{{\mathbbm{P}^{2}}}) = {\rm{SL}}_{3}(\mathbbm{C})/B$. In this particular case, we have the following:
\begin{enumerate}
\item[(i)] $H^{2}({\mathbbm{P}}(T_{{\mathbbm{P}^{2}}}),\mathbbm{R}) = H^{1,1}({\mathbbm{P}}(T_{{\mathbbm{P}^{2}}}),\mathbbm{R}) = \mathbbm{R}[{\bf{\Omega}}_{\alpha_{1}}] \oplus \mathbbm{R}[{\bf{\Omega}}_{\alpha_{2}}]$;
\item[(ii)] $H_{2}({\mathbbm{P}}(T_{{\mathbbm{P}^{2}}}),\mathbbm{Z}) = \pi_{2}({\mathbbm{P}}(T_{{\mathbbm{P}^{2}}})) = \mathbbm{Z}[\mathbbm{P}_{\alpha_{1}}^{1}] \oplus \mathbbm{Z}[\mathbbm{P}_{\alpha_{2}}^{1}]$.
\end{enumerate}
Let $\omega_{0}$ be the unique ${\rm{SU}}(3)$-invariant K\"{a}hler metric on ${\mathbbm{P}}(T_{{\mathbbm{P}^{2}}})$, such that $[\omega_{0}] = c_{1}({\mathbbm{P}}(T_{{\mathbbm{P}^{2}}}))$\footnote{It is worth pointing out that there is nothing special with this choice. In fact, all the computations presented in this example work for an arbitrary choice of ${\rm{SU}}(3)$-invariant (integral) K\"{a}hler class on ${\mathbbm{P}}(T_{{\mathbbm{P}^{2}}})$.}. Since $\lambda({\bf{K}}_{{\mathbbm{P}}(T_{{\mathbbm{P}^{2}}})}^{-1}) = \delta_{B} = 2(\varpi_{\alpha_{1}} + \varpi_{\alpha_{2}})$, from Eq. (\ref{ChernFlag}), it follows that 
\begin{equation}
\omega_{0} = \langle \delta_{B}, \alpha_{1}^{\vee} \rangle {\bf{\Omega}}_{\alpha_{1}} + \langle \delta_{B}, \alpha_{2}^{\vee} \rangle {\bf{\Omega}}_{\alpha_{2}} = 2 \big ({\bf{\Omega}}_{\alpha_{1}} + {\bf{\Omega}}_{\alpha_{2}}\big),
\end{equation}
in particular, notice that $\lambda([\omega_{0}]) = \delta_{B}$. Given any $[\psi] \in H^{1,1}({\mathbbm{P}}(T_{{\mathbbm{P}^{2}}}),\mathbbm{R})$, from Theorem \ref{ProofTheo1}, we have
\begin{equation}
\hat{\Theta} = {\rm{Arg}}\int_{{\mathbbm{P}}(T_{{\mathbbm{P}^{2}}})}(\omega_{0} + \sqrt{-1}\psi)^{3} = \sum_{j = 1}^{3} \arctan \bigg( \frac{\langle \lambda([\psi]),\alpha_{j}^{\vee} \rangle}{\langle \delta_{B},\alpha_{j}^{\vee} \rangle}\bigg) \mod 2\pi,
\end{equation}
notice that, in this particular case, $I = \emptyset$, so $\Phi_{I}^{+} = \Phi^{+} = \{\alpha_{1}, \alpha_{2}, \alpha_{3} = \alpha_{1} + \alpha_{2}\}$. Therefore, if we suppose that $[\psi] = s_{1}[{\bf{\Omega}}_{\alpha_{1}}] + s_{2}[{\bf{\Omega}}_{\alpha_{2}}]$, for some $s_{1},s_{2} \in \mathbbm{R}$, by considering the Cartan matrix $C = (C_{ij})$ of $\mathfrak{sl}_{3}(\mathbbm{C})$ (see Eq. (\ref{Cartanmatrix})), we obtain the following:
\begin{enumerate}
\item $\langle \delta_{B},\alpha_{1}^{\vee} \rangle = \langle \delta_{B},\alpha_{2}^{\vee} \rangle = 2$ and $\langle \delta_{B},\alpha_{3}^{\vee} \rangle = 4$;
\item $\langle \lambda([\psi]),\alpha_{1}^{\vee} \rangle = s_{1}$, $\langle \lambda([\psi]),\alpha_{2}^{\vee} \rangle = s_{2}$, $\langle \lambda([\psi]),\alpha_{3}^{\vee} \rangle = s_{1} + s_{2}.$
\end{enumerate}
From above, we conclude that $\hat{\Theta}([\psi])\in (-\frac{3\pi}{2},\frac{3\pi}{2})$ is given by
\begin{equation}
\label{lagrangianphasesu3}
\hat{\Theta}([\psi])= \arctan \bigg ( \frac{s_{1}}{2}\bigg) + \arctan \bigg ( \frac{s_{2}}{2}\bigg) + \arctan \bigg(\frac{s_{1} + s_{2}}{4} \bigg).
\end{equation}
From Eq. (\ref{lagrangianphase}), given an arbitrary ${\rm{SU}}(3)$-invariant $(1,1)$-form $\chi = s_{1}{\bf{\Omega}}_{\alpha_{1}} + s_{2}{\bf{\Omega}}_{\alpha_{2}}$, we have the following concrete expression for its Lagrangian phase w.r.t. $\omega_{0}$:
\begin{equation}
\Theta_{\omega_{0}}(\chi) = \arctan \bigg ( \frac{s_{1}}{2}\bigg) + \arctan \bigg ( \frac{s_{2}}{2}\bigg) + \arctan \bigg(\frac{s_{1} + s_{2}}{4} \bigg).
\end{equation}
The $(1,1)$-classes $[\psi] = s_{1}[{\bf{\Omega}}_{\alpha_{1}}] + s_{2}[{\bf{\Omega}}_{\alpha_{2}}]$ which satisfy the inequality 
\begin{equation}
\pi <\arctan \bigg ( \frac{s_{1}}{2}\bigg) + \arctan \bigg ( \frac{s_{2}}{2}\bigg) + \arctan \bigg(\frac{s_{1} + s_{2}}{4} \bigg) < \frac{3\pi}{2},
\end{equation}
define the hypercritical solutions of the dHYM equation on $({\mathbbm{P}}(T_{{\mathbbm{P}^{2}}}),\omega_{0})$. If we consider, for instance, $[\psi] = 4\big ([{\bf{\Omega}}_{\alpha_{1}}] + [{\bf{\Omega}}_{\alpha_{2}}])$, since
\begin{equation}
\pi <3\arctan (2) < \frac{3\pi}{2},
\end{equation}
we have that $\chi = 4({\bf{\Omega}}_{\alpha_{1}} + {\bf{\Omega}}_{\alpha_{2}}) \in [\psi]$ defines a hypercritical solution to the dHYM on $({\mathbbm{P}}(T_{{\mathbbm{P}^{2}}}),\omega_{0})$. In this particular case, following \cite{Chu2021space}, we have that the Riemannian manifold with non-positive sectional curvature defined by the space of calibrated $(1,1)$-forms 
\begin{equation}
\mathcal{H} = \Big \{ \Upsilon \in [\psi] \ \Big | \ {\rm{Re}} \big ( {\rm{e}}^{\sqrt{-1} \hat{\Theta}([\psi])}(\omega_{0} + \sqrt{-1} \Upsilon)^{n} \big) > 0 \Big \},
\end{equation}
has a well-defined metric structure, and its completion is a ${\rm{CAT}}(0)$ geodesic metric space. The $(1,1)$-classes $[\psi] = s_{1}[{\bf{\Omega}}_{\alpha_{1}}] + s_{2}[{\bf{\Omega}}_{\alpha_{2}}]$ which satisfy the inequality 
\begin{equation}
\frac{\pi}{2} <\arctan \bigg ( \frac{s_{1}}{2}\bigg) + \arctan \bigg ( \frac{s_{2}}{2}\bigg) + \arctan \bigg(\frac{s_{1} + s_{2}}{4} \bigg) < \frac{3\pi}{2},
\end{equation}
define the supercritical solutions of the dHYM equation on $({\mathbbm{P}}(T_{{\mathbbm{P}^{2}}}),\omega_{0})$. From the above condition, it is straightforward to produce solutions of the dHYM equation which do not satisfy the supercritical condition. For instance, we can take the unique solution $\chi$ of the dHYM equation in the class $- ([{\bf{\Omega}}_{\alpha_{1}}] + [{\bf{\Omega}}_{\alpha_{2}}])$. In this last case, we obtain
\begin{equation}
 -\frac{3\pi}{2}<\Theta_{\omega_{0}}(\chi) = 3\arctan \bigg ( - \frac{1}{2}\bigg) < \frac{\pi}{2}.
\end{equation}
Consider now the following central charges associated with the aforementioned solution $\chi$:
\begin{enumerate}
\item[(a)] $Z_{{\mathbbm{P}}(T_{{\mathbbm{P}^{2}}})}([\chi]) = - \int_{{\mathbbm{P}}(T_{{\mathbbm{P}^{2}}})} {\rm{e}}^{-\sqrt{-1}(\omega_{0} +\sqrt{-1}\chi)} = -\frac{(-\sqrt{-1})^{3}}{3!} \int_{{\mathbbm{P}}(T_{{\mathbbm{P}^{2}}})}(\omega_{0} + \sqrt{-1}\chi)^{3}$;
\item[(b)] $Z_{\mathbbm{P}_{\alpha_{1}}^{1}}([\chi]) = - \int_{\mathbbm{P}_{\alpha_{1}}^{1}} {\rm{e}}^{-\sqrt{-1}(\omega_{0} +\sqrt{-1}\chi)} = \sqrt{-1} \int_{\mathbbm{P}_{\alpha_{1}}^{1}}(\omega_{0} + \sqrt{-1}\chi)$.
\end{enumerate}
Computing the above integrals, we obtain
\begin{enumerate}
\item[(c)] ${\rm{Arg}}(Z_{{\mathbbm{P}}(T_{{\mathbbm{P}^{2}}})}([\chi])) = \frac{3\pi}{2} + 3\arctan \big( - \frac{1}{2}\big)$;
\item[(d)] ${\rm{Arg}}(Z_{\mathbbm{P}_{\alpha_{1}}^{1}}([\chi])) = \frac{\pi}{2} + \arctan \big( - \frac{1}{2}\big)$.
\end{enumerate}
Hence, it follows that ${\rm{Arg}}(Z_{{\mathbbm{P}}(T_{{\mathbbm{P}^{2}}})}([\chi])) > {\rm{Arg}}(Z_{\mathbbm{P}_{\alpha_{1}}^{1}}([\chi])) $, in other words, we have 
\begin{equation}
{\rm{Im}} \bigg ( \frac{Z_{\mathbbm{P}_{\alpha_{1}}^{1}}([\chi])}{Z_{{\mathbbm{P}}(T_{{\mathbbm{P}^{2}}})}([\chi])}\bigg) = \Bigg | \frac{Z_{\mathbbm{P}_{\alpha_{1}}^{1}}([\chi])}{Z_{{\mathbbm{P}}(T_{{\mathbbm{P}^{2}}})}([\chi])}\Bigg| \sin \Big ( - \pi - 2\arctan \Big (-\frac{1}{2}\Big )\Big ) < 0.
\end{equation}
As in \cite[Remark 1.10]{Chen2021j}, the example above shows that the “easier” direction of Collins–Jacob–Yau’s conjecture (Conjectures \ref{conjecture1}) holds only in the supercritical case. 

By keeping the previous notation, let ${\bf{E}} \in {\rm{Pic}}({\mathbbm{P}}(T_{{\mathbbm{P}^{2}}}))$ be some holomorphic line bundle. Since ${\rm{Pic}}({\mathbbm{P}}(T_{{\mathbbm{P}^{2}}}))$ is generate by $\mathscr{O}_{\alpha_{1}}(1)$ and $\mathscr{O}_{\alpha_{2}}(1)$, without loss of generality, we can suppose that 
\begin{equation}
{\bf{E}} = \mathscr{O}_{\alpha_{1}}(a) \otimes \mathscr{O}_{\alpha_{2}}(b),
\end{equation}
for some $a,b \in \mathbbm{Z}$. From above, it follows that $\lambda({\bf{E}}) = a \varpi_{\alpha_{1}} + b\varpi_{\alpha_{2}}$. As before, considering the unique ${\rm{SU}}(3)$-invariant K\"{a}hler metric $\omega_{0}$ on ${\mathbbm{P}}(T_{{\mathbbm{P}^{2}}})$, such that $[\omega_{0}] = c_{1}({\mathbbm{P}}(T_{{\mathbbm{P}^{2}}}))$, we have 
\begin{equation}
Z_{\mathbbm{P}_{\alpha_{j}}^{1}}({\bf{E}}) = - \int_{\mathbbm{P}_{\alpha_{j}}^{1}}{\rm{e}}^{-\sqrt{-1}[\omega_{0}]}{\rm{ch}}({\bf{E}}) = - \langle \lambda({\bf{E}}),\alpha_{j} \rangle + \sqrt{-1}\langle \delta_{B},\alpha_{j} \rangle,
\end{equation}
$\forall j = 1,2,3$, recall that $\lambda([\omega_{0}]) = \delta_{B}$. From Theorem \ref{TheoA1} and item (2) of Theorem \ref{TheoA2}, respectively, we have 
\begin{enumerate}
\item $\displaystyle \mu_{[\omega_{0}]}({\bf{E}}) = 2!\Bigg[\sum_{j = 1}^{3} \frac{\langle \lambda({\bf{E}}),\alpha_{j}^{\vee} \rangle}{\langle \delta_{B},\alpha_{j}^{\vee} \rangle} \Bigg] \Bigg [ \prod_{j = 1}^{3}\frac{\langle \delta_{B},\alpha_{j}^{\vee} \rangle}{\langle \varrho^{+},\alpha_{j}^{\vee} \rangle}\Bigg] = 12(a+b)$;
\item $\displaystyle \hat{\mu}({\bf{E}}) = -\Bigg [ \sum_{j = 1}^{3} \frac{{\rm{Re}}\big (Z_{\mathbbm{P}_{\alpha_{j}}^{1}}({\bf{E}})\big )}{{\rm{Im}}\big (Z_{\mathbbm{P}_{\alpha_{j}}^{1}}({\bf{E}})\big )}\Bigg ] = \sum_{j = 1}^{3} \frac{\langle \lambda({\bf{E}}),\alpha_{j}^{\vee} \rangle}{\langle \delta_{B},\alpha_{j}^{\vee} \rangle} = \frac{3}{4}(a+b)$.
\end{enumerate}
In the above computation, for item (1), we have used that $\delta_{B} = 2\varrho^{+}$, notice that
\begin{equation}
{\rm{Vol}}({\mathbbm{P}}(T_{{\mathbbm{P}^{2}}}),\omega_{0}) = \prod_{j = 1}^{3}\frac{\langle \delta_{B},\alpha_{j}^{\vee} \rangle}{\langle \varrho^{+},\alpha_{j}^{\vee} \rangle} = 8,
\end{equation}
see Eq. (\ref{VolKahler}). The computations above show that, in general, it is more simple to compute $\hat{\mu}({\bf{E}})$ than $\mu_{[\omega_{0}]}({\bf{E}})$. In what follows, we present a constructive method to obtain non-trivial examples of Hermitian-Yang-Mills structures on certain holomorphic vector bundles over ${\mathbbm{P}}(T_{{\mathbbm{P}^{2}}})$. Fixed an integer number $m \in \mathbbm{Z}$, one can seek for solutions of the linear diophantine equation 
\begin{equation}
\label{diophantineslope}
\mu_{[\omega_{0}]}({\bf{E}}) = 12(a+b) = m, \ {\bf{E}} \in {\rm{Pic}}({\mathbbm{P}}(T_{{\mathbbm{P}^{2}}})),
\end{equation}
in order to construct example of polystable holomorphic vector bundles through Whitney sums. The diophantine equation above can be solved if, and only if, $12|m$ (e.g. \cite[Chapter 5]{mordelldiophantine}), so let us suppose that $m = 12k$, for some $k \in \mathbbm{Z}$. From this, the previous equation becomes
\begin{equation}
a + b = k.
\end{equation}
Therefore, given a particular solution $a_{0},b_{0} \in \mathbbm{Z}$, for the last equation above, we have an infinite number of solutions explicitly given by $a(s):= a_{0} + s$ and $b(s):= b_{0} - s$, $s \in \mathbbm{Z}$. For every $s \in \mathbbm{Z}$, let us define ${\bf{E}}(s) := \mathscr{O}_{\alpha_{1}}(a(s)) \otimes \mathscr{O}_{\alpha_{2}}(b(s))$, such that $a(s):= a_{0} + s$ and $b(s):= b_{0} - s$. By construction, we have
\begin{equation}
\mu_{[\omega_{0}]}({\bf{E}}(s)) = 12(a(s) + b(s)) = 12k = m.
\end{equation}
Given $s_{1},\ldots,s_{r} \in \mathbbm{Z}$, $r \in \mathbbm{Z}_{>0}$, we can define
\begin{equation}
{\bf{E}}(s_{1},\ldots,s_{r}):= \bigoplus_{j = 1}^{r}{\bf{E}}(s_{j}).
\end{equation}
If $r = 2$, since ${\bf{E}}(s)$ is $[\omega_{0}]$-stable for every $s \in \mathbbm{Z}$, and $\mu_{[\omega_{0}]}({\bf{E}}(s_{1})) = \mu_{[\omega_{0}]}({\bf{E}}(s_{2}))$, it follows that ${\bf{E}}(s_{1},s_{2})$ is $[\omega_{0}]$-semistable, see for instance \cite{kobayashi_dif_cplx_vec}. Observing that 
\begin{equation}
\mu_{[\omega_{0}]}({\bf{E}}(s_{1},\ldots,s_{r})) = \frac{1}{r} \sum_{j = 1}^{r}\mu_{[\omega_{0}]}({\bf{E}}(s_{j})) = m,
\end{equation}
for every $r \in \mathbbm{Z}_{>0}$, by an inductive argument, one can conclude that ${\bf{E}}(s_{1},\ldots,s_{r})$ is a direct sum of stable vector bundles of the same $[\omega_{0}]$-slope, i.e., ${\bf{E}}(s_{1},\ldots,s_{r})$ is $[\omega_{0}]$-polystable, for all  $s_{1},\ldots,s_{r} \in \mathbbm{Z}$, and all $r \in \mathbbm{Z}_{>0}$. From Kobayashi-Hitchin correspondence \cite{donaldson1985anti,donaldson1987infinite},\cite{uhlenbeck1986existence}, we conclude that ${\bf{E}}(s_{1},\ldots,s_{r})$ is Hermite-Einstein, for all  $s_{1},\ldots,s_{r} \in \mathbbm{Z}$, and all $r \in \mathbbm{Z}_{>0}$. Moreover, in this case, we can describe the associated Hermite-Einstein structure as follows. At first, we observe the following, given $s \in \mathbbm{Z}$, we have $c_{1}({\bf{E}}(s)) = [{\bf{\Omega}}(s)]$, such that
\begin{equation}
{\bf{\Omega}}(s) = a(s){\bf{\Omega}}_{\alpha_{1}} + b(s){\bf{\Omega}}_{\alpha_{2}}.
\end{equation}
Thus, by construction, we obtain
\begin{equation}
m = \mu_{[\omega_{0}]}({\bf{E}}(s)) = \int_{{\mathbbm{P}}(T_{{\mathbbm{P}^{2}}})} {\bf{\Omega}}(s) \wedge \omega_{0}^{2} = \frac{1}{3}\Lambda_{\omega_{0}}({\bf{\Omega}}(s)){\rm{Vol}}({\mathbbm{P}}(T_{{\mathbbm{P}^{2}}}),\omega_{0})3!.
\end{equation}
Since ${\rm{Vol}}({\mathbbm{P}}(T_{{\mathbbm{P}^{2}}}),\omega_{0}) = 8$, we have 
\begin{equation}
\Lambda_{\omega_{0}}({\bf{\Omega}}(s)) = \frac{m}{16}.
\end{equation}
Therefore, given ${\bf{E}} = {\bf{E}}(s_{1},\ldots,s_{r})$, as before, we can take a Hermitian structure $H$ on ${\bf{E}}$, such that the curvature $F(H)$ of the associated Chern connection $\nabla^{H} \myeq {\rm{d}} + H^{-1}\partial H$ satisfies
\begin{equation}
\frac{\sqrt{-1}}{2\pi}F(H) = {\rm{diag}} \bigg \{ {\bf{\Omega}}(s_{1}), \ldots,{\bf{\Omega}}(s_{r})\bigg \}.
\end{equation}
From above, we obtain
\begin{equation}
\sqrt{-1}\Lambda_{\omega_{0}}(F(H)) =  2\pi {\rm{diag}} \bigg \{ \Lambda_{\omega_{0}}({\bf{\Omega}}(s_{1})), \ldots,\Lambda_{\omega_{0}}({\bf{\Omega}}(s_{r}))\bigg \} = \frac{m\pi}{8} \mathbbm{1}_{{\bf{E}}}.
\end{equation}
Hence, we conclude that $\nabla^{H} \myeq {\rm{d}} + H^{-1}\partial H$ is a Hermitian Yang-Mills connection. We can still go one step further to describe the Hermitian Yang-Mills instanton $\nabla^{H}$ in a quite explicitly way. In fact, given an open set $U \subset {\mathbbm{P}}(T_{{\mathbbm{P}^{2}}})$ which trivializes both ${\bf{E}} \to {\mathbbm{P}}(T_{{\mathbbm{P}^{2}}})$ and $B \hookrightarrow {\rm{SL}}_{3}(\mathbbm{C}) \to {\mathbbm{P}}(T_{{\mathbbm{P}^{2}}})$, and considering fiber coordinates $(w_{1},\ldots,w_{r})$ in ${\bf{E}}|_{U}$, we can construct $H$ by gluing the local Hermitian structures 
\begin{equation}
H_{U} = \sum_{j = 1}^{r}\frac{w_{j}\overline{w_{j}}} {||s_{U}v_{\varpi_{\alpha_{1}}}^{+}||^{2a(s_{j})}||s_{U}v_{\varpi_{\alpha_{2}}}^{+}||^{2b(s_{j})}},
\end{equation}
for some local section $s_{U} \colon U \subset {\mathbbm{P}}(T_{{\mathbbm{P}^{2}}})\to {\rm{SL}}_{3}(\mathbbm{C})$, here we consider $||\cdot||$ defined by some fixed ${\rm{SU}}(3)$-invariant inner product on $V(\varpi_{\alpha_{k}})$, $k = 1,2$. From this, we have
\begin{equation}
\nabla^{H}|_{U} = {\rm{d}} + {\rm{diag}}\bigg \{ A_{U}^{(1)}, \ldots , A_{U}^{(r)}\bigg\},
\end{equation}
such that 
\begin{equation}
A_{U}^{(j)} = - \partial \log \Big ( ||s_{U}v_{\varpi_{\alpha_{1}}}^{+}||^{2a(s_{j})}||s_{U}v_{\varpi_{\alpha_{2}}}^{+}||^{2b(s_{j})}\Big), \ \ \forall j = 1,\ldots, r.
\end{equation}
In particular, consider $U = U^{-}(B)$, such that
\begin{equation}
U^{-}(B) = \Bigg \{ \begin{pmatrix}
1 & 0 & 0 \\
z_{1} & 1 & 0 \\                  
z_{2}  & z_{3} & 1
 \end{pmatrix}B \ \Bigg | \ z_{1},z_{2},z_{3} \in \mathbbm{C} \Bigg \} \ \ \ ({\text{opposite big cell}}),
\end{equation}
see Remark \ref{bigcellcosntruction}. By taking the local section $s_{U} \colon U^{-}(B) \to {\rm{SL}}_{3}(\mathbbm{C})$, such that $s_{U}(nB) = n$, $\forall nB \in U^{-}(B)$, and observing that 
\begin{center}
$V(\varpi_{\alpha_{1}}) = \mathbbm{C}^{3}$ \ \  and \ \ $V(\varpi_{\alpha_{2}}) = \bigwedge^{2}(\mathbbm{C}^{3}),$
\end{center}
where $v_{\varpi_{\alpha_{1}}}^{+} = e_{1}$, and $v_{\varpi_{\alpha_{2}}}^{+} = e_{1} \wedge e_{2}$, fixed $||\cdot||$ defined by the standard ${\rm{SU}}(3)$-invariant inner product on $\mathbbm{C}^{3}$ and $\bigwedge^{2}(\mathbbm{C}^{3})$, we obtain
\begin{equation}
A_{U}^{(j)} = - \partial \log \Bigg [ \bigg ( 1 + \sum_{i = 1}^{2}|z_{i}|^{2} \bigg )^{a(s_{j})} \bigg (1 + |z_{3}|^{2} + \bigg | \det \begin{pmatrix}
z_{1} & 1  \\                  
z_{2}  & z_{3} 
 \end{pmatrix} \bigg |^{2} \bigg )^{b(s_{j})} \Bigg ],
\end{equation}
for each $j = 1,\ldots,r$. From above we obtain infinitely many explicit examples of Hermitian Yang-Mills instantons. 
\end{example}

\appendix

\section{Line bundles with prescribed slope}
\label{appendix}

Let $X_{P}$ be a flag variety. Fixed some integral K\"{a}hler class $[\omega_{0}] \in \mathcal{K}(X_{P})$, and fixed some integer number $m_{0} \in \mathbbm{Z}$, in this appendix we investigate the problem related to the solvability of the equation
\begin{equation}
\label{slopeeqproblem}
\mu_{[\omega_{0}]}({\bf{E}}) = m_{0}, \ \ {\bf{E}} \in {\rm{Pic}}(X_{P}).
\end{equation}
The main purpose is to describe the constraints which should be imposed on $\omega_{0}$ and $m_{0}$ to ensure solvability of Eq. (\ref{slopeeqproblem}). As we shall see from the result bellow, the investigation of this problem leads us to a fruitful interaction between intersection theory and number theory. We start the investigation by proving the following fundamental result.
\begin{proposition}
\label{diophantineprescribedslope}
Given $m_{0} \in \mathbbm{Z}$ and an integral K\"{a}hler class $[\omega_{0}] \in \mathcal{K}(X_{P})$, then the equation
\begin{equation}
\mu_{[\omega_{0}]}({\bf{E}}) = m_{0}, \ \ {\bf{E}} \in {\rm{Pic}}(X_{P}),
\end{equation}
has a solution if, and only if, $\tau([\omega_{0}])\big |m_{0}$, such that  
\begin{equation}
\tau([\omega_{0}]) := {\rm{gcd}}\bigg \{ \int_{X_{P}}c_{1}(\mathscr{O}_{\alpha}(1)) \wedge [\omega_{0}]^{n-1} \  \bigg | \ \ \alpha \in \Delta \backslash I\bigg\},
\end{equation}
where $\mathscr{O}_{\alpha}(1)$, $\alpha \in \Delta \backslash I$, are the generators of ${\rm{Pic}}(X_{P})$.
\end{proposition}
\begin{proof}
Given a flag variety $X_{P}$, fixed an integral K\"{a}hler class $[\omega_{0}] \in \mathcal{K}(X_{P})$, and fixed some $m_{0} \in \mathbbm{Z}$, we seek for solutions of the equation
\begin{equation}
\mu_{[\omega_{0}]}({\bf{E}}) = m_{0}, \ \ {\bf{E}} \in {\rm{Pic}}(X_{P}). 
\end{equation}
Considering ${\bf{E}} = \bigotimes_{\alpha \in \Delta \backslash I}\mathscr{O}_{\alpha}(x_{\alpha})$, such that $x_{\alpha} \in \mathbbm{Z}$, $\forall \alpha \in \Delta \backslash I$, the equation above turns out to be the linear equation
\begin{equation}
\label{linearequationdegree}
\sum_{\alpha \in \Delta \backslash I}\deg_{\omega_{0}}(\mathscr{O}_{\alpha}(1))x_{\alpha} = m_{0},
\end{equation}
such that $\deg_{\omega_{0}}(\mathscr{O}_{\alpha}(1)) = \int_{X_{P}}c_{1}(\mathscr{O}_{\alpha}(1)) \wedge [\omega_{0}]^{n-1}$, $\forall \alpha \in \Delta \backslash I$. Since $c_{1}(\mathscr{O}_{\alpha}(1))$, $\alpha \in \Delta \backslash I$, and $[\omega_{0}]$ are integral classes, it follows that 
\begin{equation}
\deg_{\omega_{0}}(\mathscr{O}_{\alpha}(1)) = \int_{X_{P}}c_{1}(\mathscr{O}_{\alpha}(1)) \wedge [\omega_{0}] \wedge [\omega_{0}]^{n-2} = \mathcal{Q}_{\omega_{0}}\big (c_{1}(\mathscr{O}_{\alpha}(1)),[\omega_{0}] \big ) \in \mathbbm{Z}, 
\end{equation}
$\forall \alpha \in \Delta \backslash I$, where $\mathcal{Q}_{\omega_{0}} \colon H^{2}(X_{P},\mathbbm{Z}) \times H^{2}(X_{P},\mathbbm{Z}) \to \mathbbm{Z}$ is the Hodge-Riemann bilinear form associated with the underlying polarized Hodge structure (e.g. \cite{VoisinBook1}, \cite{peters2008mixed}). Thus, it follows that Eq. (\ref{linearequationdegree}) is equivalent to the linear diophantine equation
\begin{equation}
\sum_{\alpha \in \Delta \backslash I}\mathcal{Q}_{\omega_{0}}\big (c_{1}(\mathscr{O}_{\alpha}(1)),[\omega_{0}] \big ) x_{\alpha} = m_{0}.
\end{equation}
Since the linear diophantine equation above admits a solution if, and only if, 
\begin{equation}
{\rm{gcd}} \big \{ \mathcal{Q}_{\omega_{0}}\big (c_{1}(\mathscr{O}_{\alpha}(1)),[\omega_{0}] \big )  \ \big | \alpha \in \Delta \backslash I \big \} \Big | m_{0},
\end{equation}
see for instance \cite[Chapter 5]{mordelldiophantine}, \cite[\S 1.6]{nathanson2008elementary}, we obtain the desired result.
\end{proof}

\begin{remark} From the above result, we have that $\tau([\omega_{0}]) \big | \mu_{[\omega_{0}]}({\bf{E}})$, $\forall {\bf{E}} \in {\rm{Pic}}(X_{P})$.
\end{remark}

\begin{remark}
Notice that, by the generalized Bezout's identity \cite[\S 1.2]{nathanson2008elementary}, we have 
\begin{equation}
\tau([\omega_{0}])\mathbbm{Z} = \bigoplus_{\alpha \in \Delta \backslash I}\mathcal{Q}_{\omega_{0}}({\bf{\Omega}}_{\alpha},\omega_{0})\mathbbm{Z},
\end{equation}
here we have used the fact that 
\begin{equation}
\mathcal{Q}_{\omega_{0}}({\bf{\Omega}}_{\alpha},\omega_{0}) := \int_{X_{P}}{\bf{\Omega}}_{\alpha} \wedge \omega_{0}^{n-1} = \mathcal{Q}_{\omega_{0}}\big (c_{1}(\mathscr{O}_{\alpha}(1)),[\omega_{0}] \big ), 
\end{equation}
for every $\alpha \in \Delta \backslash I$.
\end{remark}

\begin{remark}
In the setting of Proposition \ref{diophantineprescribedslope}, given a prime number $p \in \mathbbm{Z}$, the equation
\begin{equation}
\mu_{[\omega_{0}]}({\bf{E}}) = p, \ \ {\bf{E}} \in {\rm{Pic}}(X_{P}),
\end{equation}
is solvable if, and only if, $\tau([\omega_{0}]) = p$ or $\tau([\omega_{0}]) = 1$.
\end{remark}

\begin{remark}
Notice that $\mathscr{O}_{\alpha}(1) = \mathcal{O}(D_{\alpha})$, for every Schubert divisor $D_{\alpha} \in {\rm{Div}}(X_{P})$, $\alpha \in \Delta \backslash I$, see Subsection \ref{divisorsandcycles}. Thus, fixed an integral K\"{a}hler class $[\omega_{0}] \in \mathcal{K}(X_{P})$, we have
\begin{equation}
\deg_{\omega_{0}}(\mathscr{O}_{\alpha}(1)) = \int_{D_{\alpha}}\omega_{0}^{n-1}  = \big \langle [\omega_{0}^{n-1}], [D_{\alpha}] \big \rangle.
\end{equation}
For every integer $n$ and every prime number $p \in \mathbbm{Z}$, let $v_{p}(n)$ be the greatest integer, such that $p^{v_{p}(n)} \big | n$. From the fundamental theorem of arithmetic, it follows that
\begin{equation}
\tau([\omega_{0}]) = \prod_{p  | \tau([\omega_{0}])}p^{\min \big \{v_{p}(\langle [\omega_{0}^{n-1}], [D_{\alpha}] \rangle) \ \big | \ \alpha \in \Delta \backslash I \big \}},
\end{equation}
see for instance \cite[Theorem 1.11]{nathanson2008elementary}. Thus, the prime factorization of $\tau([\omega_{0}])$ is completely determined by the intersection numbers $\langle [\omega_{0}^{n-1}], [D_{\alpha}] \rangle$, $\alpha \in \Delta \backslash I$.
\end{remark}

From Proposition \ref{diophantineprescribedslope}, we have the following corollary.

\begin{corollary}
If $[\omega_{0}] \in \mathcal{K}(X_{P})$ is an integral K\"{a}hler class, such that $\tau([\omega_{0}]) = 1$, and $m_{0} \in \mathbbm{Z}$, satisfies 
\begin{equation}
\label{nefcondition}
m_{0} \geq \big ( \mathcal{Q}_{\omega_{0}}({\bf{\Omega}}_{\gamma},\omega_{0}) - 1\big ) \sum_{\substack{\alpha \in \Delta \backslash I, \alpha \neq \gamma}}\mathcal{Q}_{\omega_{0}}({\bf{\Omega}}_{\alpha},\omega_{0}),
\end{equation}
for some $\gamma \in \Delta \backslash I$, then there exists ${\bf{E}} \in {\rm{Pic}}(X_{P})$, such that $H^{0}(X_{P},{\bf{E}}) \neq \{0\}$, and $\mu_{[\omega_{0}]}({\bf{E}}) = m_{0}$.
\end{corollary}
\begin{proof}
If the condition given in Eq. (\ref{nefcondition}) holds, it follows from \cite[Theorem 1.16]{nathanson2008elementary} that there exist integers $s_{\alpha} \in \mathbbm{Z}_{\geq 0}$, $\alpha \in \Delta \backslash I$, such that 
\begin{equation}
\sum_{\alpha \in \Delta \backslash I}\mathcal{Q}_{\omega_{0}}\big ({\bf{\Omega}}_{\alpha},\omega_{0} \big ) s_{\alpha} = m_{0}.
\end{equation}
Defining ${\bf{E}} :=  \bigotimes_{\alpha \in \Delta \backslash I} \mathscr{O}_{\alpha}(s_{\alpha})$, it follows that $\mu_{[\omega_{0}]}({\bf{E}}) = m_{0}$. In particular, we have
\begin{equation}
\lambda({\bf{E}}) = \sum_{\alpha \in \Delta \backslash I}s_{\alpha} \varpi_{\alpha} \in \Lambda^{+},
\end{equation}
i.e., $\lambda({\bf{E}})$ is an integral dominant weight of $\mathfrak{g}^{\mathbbm{C}}$. Therefore, from the Borel-Weil theorem, it follows that 
\begin{equation}
H^{0}(X_{P},{\bf{E}}) \cong V(\lambda({\bf{E}}))^{\ast} \neq \{0\},
\end{equation}
which concludes the proof.
\end{proof}

\begin{remark}
In the setting of the above corollary, we have that ${\bf{E}} \in {\rm{Pic}}(X_{P})$ constructed in the proof is a nef holomorphic line bundle (e.g. \cite{Snowhomovec}). Therefore, under the hypotheses of the previous corollary, for every sufficiently large integer $m_{0} \in \mathbbm{Z}$, the equation $\mu_{[\omega_{0}]}({\bf{E}}) = m_{0}$ can be solved by some nef holomorphic line bundle.
\end{remark}

Now we are in position to prove Theorem \ref{densityslope}. Fixed some integral K\"{a}hler class $[\omega_{0}] \in \mathcal{K}(X_{P})$, consider the following subset of $\mathbbm{Z}_{>0}$:
\begin{equation}
\mathcal{A}_{[\omega_{0}]}(n) := \Bigg \{ m \in  \mathbbm{Z}_{>0}  \ \ \Bigg | \ \ \begin{array}{l}
 \ \ \ \ 1 \leq m \leq n, \\
\exists \ {\bf{E}} \in {\rm{Pic}}(X_{P}), \ {\text{s.t.}} \ \mu_{[\omega_{0}]}({\bf{E}}) = m.
\end{array}\Bigg \}.
\end{equation}
From Proposition \ref{diophantineprescribedslope}, we have the following result.
\begin{theorem}
\label{naturaldensity}
For every integral K\"{a}hler class $[\omega_{0}] \in \mathcal{K}(X_{P})$, we have
\begin{equation}
\lim_{n \to +\infty}\frac{\big | \mathcal{A}_{[\omega_{0}]}(n)\big |}{n} = \frac{1}{\tau([\omega_{0}])},
\end{equation}
such that $\tau([\omega_{0}]):= {\rm{gcd}}\big \{ \deg_{\omega_{0}}(\mathscr{O}_{\alpha}(1)) \ \big | \ \alpha \in \Delta \backslash I\big\}$, where $\mathscr{O}_{\alpha}(1)$, $\alpha \in \Delta \backslash I$, are the generators of the Picard group ${\rm{Pic}}(X_{P})$.
\end{theorem}
\begin{proof}
From Proposition \ref{diophantineprescribedslope}, we have 
\begin{equation}
\mathcal{A}_{[\omega_{0}]}(n) = \Big \{ m \in \mathbbm{Z}_{>0} \ \Big | \ 1 \leq m \leq n, \ \tau([\omega_{0}])|m \Big \},  
\end{equation}
$\forall n \in \mathbbm{Z}_{>0}$. Given $n \in \mathbbm{Z}_{>0}$, from the division algorithm, there exist unique integers $t(n)$ and $r(n)$, such that 
\begin{equation}
n = \tau([\omega_{0}])t(n) + r(n), \ \ \ 0 \leq r(n) < \tau([\omega_{0}]). 
\end{equation}
Therefore, for every $n \in \mathbbm{Z}_{>0}$, we have 
\begin{equation}
\big | \mathcal{A}_{[\omega_{0}]}(n)\big | = t(n) = \frac{n - r(n)}{\tau([\omega_{0}])}.
\end{equation}
Since $0 \leq r(n) < \tau([\omega_{0}])$, it follows that 
\begin{equation}
 0 \leq \frac{1}{\tau([\omega_{0}])} - \frac{\big | \mathcal{A}_{[\omega_{0}]}(n)\big |}{n} < \frac{\tau([\omega_{0}])}{n}.
\end{equation}
Hence, $\frac{|\mathcal{A}_{[\omega_{0}]}(n)|}{n} \to \frac{1}{\tau([\omega_{0}])}$ as $n \to +\infty$.
\end{proof}
\begin{remark}
The above result shows that the subset 
\begin{equation}
\mathcal{A}_{[\omega_{0}]} = \big \{m \in \mathbbm{Z}_{>0} \ \big | \ \mu_{[\omega_{0}]}({\bf{E}}) = m, \ {\text{for some}} \ {\bf{E}} \in {\rm{Pic}}(X_{P}) \big\},
\end{equation}
has asymptotic density (e.g. \cite[Chapter 16]{nathanson2008elementary}) 
\begin{equation}
d(\mathcal{A}_{[\omega_{0}]}) := {\displaystyle{\lim_{n\to +\infty}}}\frac{\big |\mathcal{A}_{[\omega_{0}]} \cap [1,n] \big |}{n} = \frac{1}{\tau([\omega_{0}])}.
\end{equation}
\end{remark}
\begin{remark}[Hodge-Riemann bilinear form]
\label{Hodge_Riemann}
 In the setting of Proposition \ref{linearequationdegree}, if we consider $\Delta \backslash I = \{\alpha_{1},\ldots,\alpha_{\rho}\}$, where $\rho$ is the Picard number of $X_{P}$, we have
 \begin{equation}
\deg_{\omega_{0}}({\bf{E}}) = \mathcal{Q}_{\omega_{0}}\big (c_{1}({\bf{E}}),[\omega_{0}] \big ) = \sum_{i,j = 1}^{\rho}\langle c_{1}({\bf{E}}),\mathbbm{P}_{\alpha_{i}}^{1} \rangle \langle [\omega_{0}],\mathbbm{P}_{\alpha_{j}}^{1} \rangle \underbrace{\mathcal{Q}_{\omega_{0}} \big ( {\bf{\Omega}}_{\alpha_{i}},{\bf{\Omega}}_{\alpha_{j}} \big )}_{(\mathcal{Q}_{\omega_{0}})_{ij}}, 
 \end{equation}
 such that 
 \begin{equation}
(\mathcal{Q}_{\omega_{0}})_{ij} =  \int_{X_{P}} {\bf{\Omega}}_{\alpha_{i}} \wedge {\bf{\Omega}}_{\alpha_{j}}  \wedge \omega_{0}^{n-2},
 \end{equation}
 $\forall i,j = 1,\ldots,\rho$. From \cite[Lemma 4.7]{szekelyhidi2014introduction}, we have 
 \begin{equation}
(\mathcal{Q}_{\omega_{0}})_{ij} = (n-2)!\Big ( \Lambda_{\omega_{0}}( {\bf{\Omega}}_{\alpha_{i}})\Lambda_{\omega_{0}}( {\bf{\Omega}}_{\alpha_{j}}) - \langle {\bf{\Omega}}_{\alpha_{i}}, {\bf{\Omega}}_{\alpha_{j}}\rangle_{\omega_{0}} \Big ){\rm{Vol}}(X_{P},\omega_{0}),
 \end{equation}
$\forall i,j = 1,\ldots,\rho$. By following Remark \ref{origineigenvalues}, it follows that 
\begin{equation}
\Lambda_{\omega_{0}}({\bf{\Omega}}_{\alpha})=\sum_{\beta \in \Phi_{I}^{+}} \frac{\langle \varpi_{\alpha}, \beta^{\vee} \rangle}{\langle \lambda([\omega_{0}]), \beta^{\vee}\rangle}, \ \langle {\bf{\Omega}}_{\alpha}, {\bf{\Omega}}_{\gamma}\rangle_{\omega_{0}} = \sum_{\beta \in \Phi_{I}^{+}} \frac{\langle \varpi_{\alpha}, \beta^{\vee} \rangle}{\langle \lambda([\omega_{0}]), \beta^{\vee}\rangle}  \frac{\langle \varpi_{\gamma}, \beta^{\vee} \rangle}{\langle \lambda([\omega_{0}]), \beta^{\vee}\rangle},
\end{equation}
$\forall \alpha,\gamma \in \Delta \backslash I$. Moreover, from Eq. (\ref{VolKahler}), we have 
\begin{equation}
{\rm{Vol}}(X_{P},\omega_{0}) = \prod_{\beta \in \Phi_{I}^{+}} \frac{\langle \lambda([\omega_{0}]),\beta^{\vee} \rangle}{\langle \varrho^{+},\beta^{\vee} \rangle}.
\end{equation}
Therefore, the Hodge-Riemann bilinear form $\mathcal{Q}_{\omega_{0}} \colon H^{2}(X_{P},\mathbbm{Z}) \times H^{2}(X_{P},\mathbbm{Z}) \to \mathbbm{Z}$ is completely determined by the relation between fundamental weights and simple roots provided by the Cartan matrix of $\mathfrak{g}^{\mathbbm{C}}$. Moreover, fixed an integral K\"{a}hler class $\omega_{0} = \sum_{\alpha \in \Delta \backslash I}s_{\alpha} {\bf{\Omega}}_{\alpha}$, and fixed $m_{0} \in \mathbbm{Z}$, the equation $\mu_{[\omega_{0}]}({\bf{E}}) = m_{0}$, ${\bf{E}} \in {\rm{Pic}}(X_{P})$, is equivalent to 
\begin{equation}
\label{linearequationslope}
\begin{pmatrix} \int_{\mathbbm{P}_{\alpha_{1}}^{1}}c_{1}({\bf{E}}) & \cdots & \int_{\mathbbm{P}_{\alpha_{\rho}}^{1}}c_{1}({\bf{E}})\end{pmatrix}\begin{pmatrix} (\mathcal{Q}_{\omega_{0}})_{11} & \cdots & (\mathcal{Q}_{\omega_{0}})_{1\rho} \\
 \vdots & \ddots & \vdots  \\ 
 (\mathcal{Q}_{\omega_{0}})_{\rho 1} & \cdots & (\mathcal{Q}_{\omega_{0}})_{\rho \rho}\end{pmatrix} \begin{pmatrix} s_{1} \\ \vdots \\ s_{\rho}\end{pmatrix} = m_{0}, \ {\bf{E}} \in {\rm{Pic}}(X_{P}).
\end{equation}
It is worth to point out that, if ${\bf{F}} \in {\rm{Pic}}(X_{P})$ satisfies $\mu_{[\omega_{0}]}({\bf{F}}) = m_{0}$, then all solutions of Eq. (\ref{linearequationslope}) can be written as ${\bf{F}} \otimes {\bf{G}}$, where ${\bf{G}} \in {\rm{Pic}}(X_{P})$ is a solution of the associated homogeneous problem $\mu_{[\omega_{0}]}({\bf{E}}) = 0$, ${\bf{E}} \in {\rm{Pic}}(X_{P})$. 
\end{remark}

Let us illustrate the ideas developed above by means of an example.

\begin{example}
Let $X_{P} = {\mathbbm{P}}(T_{{\mathbbm{P}^{2}}})$ as in Example \ref{computationWallach}. Given some integral K\"{a}hler metric $\omega_{0} = s_{1} {\bf{\Omega}}_{\alpha_{1}} + s_{2}{\bf{\Omega}}_{\alpha_{2}}$, a straightforward computation shows us that 
\begin{enumerate}
\item $(\mathcal{Q}_{\omega_{0}})_{11} = 1! \Big [ \Big ( \frac{1}{s_{1}} + \frac{1}{s_{1} + s_{2}}\Big)^{2} - \Big ( \frac{1}{s_{1}^{2}} + \frac{1}{(s_{1}+s_{2})^{2}}\Big )\Big] \frac{s_{1}s_{2}(s_{1}+s_{2})}{2} = s_{2}$;
\item $(\mathcal{Q}_{\omega_{0}})_{22} = 1! \Big [ \Big ( \frac{1}{s_{2}} + \frac{1}{s_{1} + s_{2}}\Big)^{2} - \Big ( \frac{1}{s_{2}^{2}} + \frac{1}{(s_{1}+s_{2})^{2}}\Big )\Big] \frac{s_{1}s_{2}(s_{1}+s_{2})}{2} = s_{1}$;
\item $(\mathcal{Q}_{\omega_{0}})_{12} = (\mathcal{Q}_{\omega_{0}})_{21} =  1! \Big [ \Big ( \frac{1}{s_{1}} + \frac{1}{s_{1} + s_{2}}\Big)\Big ( \frac{1}{s_{2}} + \frac{1}{s_{1} + s_{2}}\Big)- \frac{1}{(s_{1}+s_{2})^{2}}\Big] \frac{s_{1}s_{2}(s_{1}+s_{2})}{2} = s_{1} + s_{2}$.
\end{enumerate}
Therefore, for an arbitrary integral class K\"{a}hler metric $\omega_{0} = s_{1} {\bf{\Omega}}_{\alpha_{1}} + s_{2}{\bf{\Omega}}_{\alpha_{2}}$, we have
\begin{equation}
\label{HRbilinearmatrix}
\mathcal{Q}_{\omega_{0}} = \begin{pmatrix} s_{2} & s_{1} + s_{2} \\
 s_{1}+s_{2} & s_{1}\end{pmatrix} \in {\rm{GL}}_{2}(\mathbbm{Z}).
\end{equation}
From above, the general equation associated with the problem $\mu_{[\omega_{0}]}({\bf{E}}) = m_{0}$, such that $m_{0} \in \mathbbm{Z}$, and ${\bf{E}} \in {\rm{Pic}}({\mathbbm{P}}(T_{{\mathbbm{P}^{2}}}))$, can be explicitly written as follows: 
\begin{equation}
\label{explicitdiophantine}
(s_{1}s_{2} + s_{2}(s_{1}+s_{2}))x_{1} + (s_{1}s_{2} + s_{1}(s_{1}+s_{2}))x_{1} = m_{0},
\end{equation}
where ${\bf{E}} = \mathscr{O}_{\alpha_{1}}(x_{1}) \otimes \mathscr{O}_{\alpha_{2}}(x_{2}) \in {\rm{Pic}}({\mathbbm{P}}(T_{{\mathbbm{P}^{2}}}))$, cf. Eq. (\ref{linearequationslope}). Notice that, in the particular case that $s_{1} = s_{2} = 2$, we recover Eq. (\ref{diophantineslope}). From Proposition \ref{diophantineprescribedslope}, we conclude that the linear diophantine equation Eq. (\ref{explicitdiophantine}) can be solved if, and only if,
\begin{equation}
\tau([\omega_{0}]) = {\rm{gcd}} \big \{s_{1}s_{2} + s_{2}(s_{1}+s_{2}), s_{1}s_{2} + s_{1}(s_{1}+s_{2})\big \} \Big | m_{0}.
\end{equation}
As in Example \ref{computationWallach}, the above ideas allow us to construct several explicit non-trivial examples of Hermitian Yang-Mills instantons. Also, from Theorem \ref{naturaldensity}, for every integral K\"{a}hler class $[\omega_{0}] \in \mathcal{K}({\mathbbm{P}}(T_{{\mathbbm{P}^{2}}}))$, such that $\omega_{0} =  \omega_{0} = s_{1} {\bf{\Omega}}_{\alpha_{1}} + s_{2}{\bf{\Omega}}_{\alpha_{2}}$, we have the natural density of $\mathcal{A}_{[\omega_{0}]}$ given by
\begin{equation}
\label{naturaldensityexample}
d(\mathcal{A}_{[\omega_{0}]}) = {\displaystyle{\lim_{n\to +\infty}}}\frac{\big |\mathcal{A}_{[\omega_{0}]} \cap [1,n] \big |}{n} = \frac{1}{\tau(s_{1},s_{2})},
\end{equation}
such that $\tau(s_{1},s_{2}):= {\rm{gcd}} \big \{s_{1}s_{2} + s_{2}(s_{1}+s_{2}), s_{1}s_{2} + s_{1}(s_{1}+s_{2})\big \}$. Notice that, since $\tau(s_{1},s_{2})$ is symmetric in the variables $s_{1}$ and $s_{2}$, we can easily produce examples of integral K\"{a}hler classes $[\omega_{1}], [\omega_{2}] \in \mathcal{K}({\mathbbm{P}}(T_{{\mathbbm{P}^{2}}}))$, satisfying $\tau([\omega_{1}]) = \tau([\omega_{2}])$, with $[\omega_{1}] \neq [\omega_{2}]$.
\end{example}

In order to prove Theorem \ref{Kgroup}, we introduce some basic results concerned with primitive $(1,1)$-forms on flag varieties.

\begin{remark}[Primitive 2-forms] 
\label{primitivesplitting}
Given a flag variety $X_{P}$ with some fixed integral K\"{a}hler class $[\omega_{0}] \in \mathcal{K}(X_{P})$, and considering  $\Lambda_{\omega_{0}}$, i.e., the dual of the associated Lefschetz operator, we can describe explicitly the primitive submodule 
\begin{equation}
H^{2}_{\omega_{0}}(X_{P},\mathbbm{Z})_{{\text{prim}}} = \ker \Big (\Lambda_{\omega_{0}} \colon H^{2}(X_{P},\mathbbm{Z}) \to \mathbbm{Z} \Big ),
\end{equation}
in terms of the Cartan matrix of the complex simple Lie algebra underlying $X_{P}$. In fact, keeping the notation of Remark \ref{Hodge_Riemann}, given $[\psi] \in H^{2}(X_{P},\mathbbm{Z})$, we have 
\begin{equation}
\Lambda_{\omega_{0}}([\psi]) = \sum_{\alpha \in \Delta \backslash I}\Lambda_{\omega_{0}}([{\bf{\Omega}}_{\alpha}]) x_{\alpha},
\end{equation}
such that $x_{\alpha} = \big \langle \lambda([\psi]), [\mathbbm{P}^{1}_{\alpha}] \big \rangle$, $\forall \alpha \in \Delta \backslash I$. Since $\Lambda_{\omega_{0}}([{\bf{\Omega}}_{\alpha}]) = \mathcal{Q}_{\omega_{0}}({\bf{\Omega}}_{\alpha},\omega_{0})$, $\forall \alpha \in \Delta \backslash I$, we have
\begin{equation}
\Lambda_{\omega_{0}}([\psi]) = 0 \iff \sum_{\alpha \in \Delta \backslash I}\frac{\mathcal{Q}_{\omega_{0}}({\bf{\Omega}}_{\alpha},\omega_{0})}{\tau([\omega_{0}])}x_{\alpha}=0.
\end{equation}
Denoting $q_{\alpha}(\omega_{0})= \frac{\mathcal{Q}_{\omega_{0}}({\bf{\Omega}}_{\alpha},\omega_{0})}{\tau([\omega_{0}])}$, $\forall \alpha \in \Delta \backslash I$, and taking some $\gamma \in \Delta \backslash I$, we obtain a $\mathbbm{Z}$-basis for $H^{2}(X_{P},\mathbbm{Z})_{{\text{prim}}}$ by setting 
\begin{equation}
\label{basiszeroslope}
{\bf{\xi}}_{\alpha} := -q_{\alpha}(\omega_{0})[{\bf{\Omega}}_{\gamma}] + q_{\gamma}(\omega_{0})[{\bf{\Omega}}_{\alpha}], 
\end{equation}
for all $\forall \alpha \in \Delta \backslash I$, such that $\alpha \neq \gamma$. From this, we have
\begin{equation}
H^{2}(X_{P},\mathbbm{Z})_{{\text{prim}}} = \bigoplus_{\substack{\alpha \in \Delta \backslash I, \alpha \neq \gamma}}\mathbbm{Z}\xi_{\alpha}.
\end{equation}
By construction, we obtain a $\mathcal{Q}_{\omega_{0}}$-orthogonal decomposition 
\begin{equation}
H^{2}(X_{P},\mathbbm{Z}) = \mathbbm{Z}[\omega_{0}] \oplus H^{2}(X_{P},\mathbbm{Z})_{{\text{prim}}}.
\end{equation}
If we consider the extension $\mathcal{Q}_{\omega_{0}} \colon H^{2}(X_{P},\mathbbm{Q}) \times H^{2}(X_{P},\mathbbm{Q}) \to \mathbbm{Q}$, then we can apply the Gram–Schmidt process w.r.t. $\mathcal{Q}_{\omega_{0}}$ on the basis $\xi_{\alpha}$, $\alpha \in \Delta \backslash I, \alpha \neq \gamma$, in order to obtain a complete decomposition of $H^{2}(X_{P},\mathbbm{Q})$ as a direct sum of $\mathcal{Q}_{\omega_{0}}$-orthogonal $\mathbbm{Q}$-subspaces. From Eq. (\ref{basiszeroslope}), we can describe explicitly a set of generators for the kernel of the homomorphism $\mu_{[\omega_{0}]}\colon {\rm{Pic}}(X_{P}) \to \mathbbm{Z}$. In fact, denoting 
\begin{equation}
{\rm{Pic}}^{0}_{\omega_{0}}(X_{P}) = \Big \{ {\bf{E}} \in {\rm{Pic}}(X_{P}) \ \Big | \ \mu_{[\omega_{0}]}({\bf{E}}) = 0 \Big \},
\end{equation}
since ${\rm{Pic}}^{0}_{\omega_{0}}(X_{P}) \cong H^{2}_{\omega_{0}}(X_{P},\mathbbm{Z})_{{\text{prim}}}$, we have that
\begin{equation} \mathscr{O}_{\gamma}(-q_{\alpha}(\omega_{0})) \otimes \mathscr{O}_{\alpha}(q_{\gamma}(\omega_{0})), \ \ \forall \alpha \in \Delta \backslash I, \alpha \neq \gamma,
\end{equation}
define a set of generators for ${\rm{Pic}}^{0}_{\omega_{0}}(X_{P})$. Moreover, consider the finitely generated subgroup ${\rm{H}}_{\omega_{0}}\subset {\text{Hom}}(P,\mathbbm{C}^{\times})$, such that
\begin{equation}
{\rm{H}}_{\omega_{0}} := \Big \langle \vartheta_{\varpi_{\gamma}}^{-q_{\alpha}(\omega_{0})} \vartheta_{\varpi_{\alpha}}^{q_{\gamma}(\omega_{0})} \ \Big | \ \alpha \in \Delta \backslash I, \alpha \neq \gamma \Big \rangle,
\end{equation}
recall that ${\text{Hom}}(P,\mathbbm{C}^{\times}) = {\text{Hom}}(T(\Delta \backslash I)^{\mathbbm{C}},\mathbbm{C}^{\times})$, and $({\rm{d}}\vartheta_{\varpi_{\alpha}})_{e} = \varpi_{\alpha}$, $\forall \alpha \in \Delta$. The subgroup ${\rm{H}}_{\omega_{0}}$ constructed above completes the following commutative diagram: 
\begin{center}
\begin{tikzcd} &  1 \arrow[d]& 1 \arrow[d] & &  \\ 1 \arrow[r] &  {\rm{H}}_{\omega_{0}} \arrow[r,"\iota"] \arrow[d]& {\text{Hom}}(P,\mathbbm{C}^{\times}) \arrow[r,"\widehat{\mu}_{[\omega_{0}]}"] \arrow[d] & \tau([\omega_{0}])\mathbbm{Z} \arrow[r] & 1 \\
1 \arrow[r]  &  {\rm{Pic}}^{0}_{\omega_{0}}(X_{P}) \arrow[r,"\iota"]  \arrow[d] & {\rm{Pic}}(X_{P})  \arrow[r,"\mu_{[\omega_{0}]}"]  \arrow[d] & \tau([\omega_{0}])\mathbbm{Z} \arrow[r] \arrow[u,equals] & 1 \\
&  1 & 1  & &  
\end{tikzcd}
\end{center}
In the top line of the above diagram we consider $\widehat{\mu}_{[\omega_{0}]} \colon {\text{Hom}}(P,\mathbbm{C}^{\times}) \to \tau([\omega_{0}])\mathbbm{Z}$, such that 
\begin{equation}
\widehat{\mu}_{[\omega_{0}]}(\vartheta) := (n-1)!\Bigg [ \sum_{\beta \in \Phi_{I}^{+} } \frac{\langle ({\rm{d}}\vartheta)_{e}, \beta^{\vee} \rangle}{\langle \lambda([\omega_{0}]), \beta^{\vee}\rangle} \Bigg ] \Bigg [ \prod_{\beta \in \Phi_{I}^{+}} \frac{\langle \lambda([\omega_{0}]),\beta^{\vee} \rangle}{\langle \varrho^{+},\beta^{\vee} \rangle}\Bigg].
\end{equation}
$\forall \vartheta \in {\text{Hom}}(P,\mathbbm{C}^{\times})$. Thus, from Theorem \ref{TheoA1}, we have $\widehat{\mu}_{[\omega_{0}]}(\vartheta) = \mu_{[\omega_{0}]}({\bf{E}}_{\vartheta})$, where ${\bf{E}}_{\vartheta} \in {\rm{Pic}}(X_{P})$ is the holomorphic line bundle defined by $\vartheta \in {\text{Hom}}(P,\mathbbm{C}^{\times})$. Notice that Proposition \ref{diophantineprescribedslope} ensures that both $\widehat{\mu}_{[\omega_{0}]}$ and $\mu_{[\omega_{0}]}$ are surjective homomorphisms.
\end{remark}

\begin{example}
\label{generatorspic_0}
As before, consider the case that $X_{P} = {\mathbbm{P}}(T_{{\mathbbm{P}^{2}}})$. Given some integral K\"{a}hler class $[\omega_{0}] = s_{1}[{\bf{\Omega}}_{\alpha_{1}}] +  s_{2}[{\bf{\Omega}}_{\alpha_{2}}]$, it follows from Eq. (\ref{HRbilinearmatrix}) that
\begin{enumerate}
\item $\mathcal{Q}_{\omega_{0}}({\bf{\Omega}}_{\alpha_{1}},\omega_{0}) =  \begin{pmatrix} 1 & 0 \end{pmatrix}\begin{pmatrix} s_{2} & s_{1} + s_{2} \\
 s_{1}+s_{2} & s_{1}\end{pmatrix} \begin{pmatrix} s_{1} \\ s_{2}\end{pmatrix} = s_{1}s_{2} + s_{2}(s_{1} + s_{2})$;
 \item $\mathcal{Q}_{\omega_{0}}({\bf{\Omega}}_{\alpha_{2}},\omega_{0}) =  \begin{pmatrix} 0 & 1 \end{pmatrix}\begin{pmatrix} s_{2} & s_{1} + s_{2} \\
 s_{1}+s_{2} & s_{1}\end{pmatrix} \begin{pmatrix} s_{1} \\ s_{2}\end{pmatrix} = s_{1}s_{2} + s_{1}(s_{1} + s_{2})$.
\end{enumerate}
Let us denote $q_{j}(\omega_{0}) = \frac{\mathcal{Q}_{\omega_{0}}({\bf{\Omega}}_{\alpha_{j}},\omega_{0})}{\tau([\omega_{0}])}$, $j = 1,2$. From above, we have that 
\begin{equation}
\xi = -q_{2}(\omega_{0})[{\bf{\Omega}}_{\alpha_{1}}] + q_{1}(\omega_{0})[{\bf{\Omega}}_{\alpha_{2}}],
\end{equation}
generates $H_{\omega_{0}}^{2}({\mathbbm{P}}(T_{{\mathbbm{P}^{2}}}),\mathbbm{Z})_{{\text{prim}}}$. Hence, we have the ${\mathcal{Q}}_{\omega_{0}}$-orthogonal decomposition
\begin{equation}
H^{2}({\mathbbm{P}}(T_{{\mathbbm{P}^{2}}}),\mathbbm{Z}) = \mathbbm{Z}[\omega_{0}] \oplus \mathbbm{Z}\xi.
\end{equation}
In this particular case, we have
\begin{equation}
{\rm{H}}_{\omega_{0}} = \Big \{ \vartheta_{\varpi_{\alpha_{1}}}^{-nq_{2}(\omega_{0})}\vartheta_{\varpi_{\alpha_{2}}}^{nq_{1}(\omega_{0})} \ \Big | \ n \in \mathbbm{Z}\Big \} = \Big \langle \vartheta_{\varpi_{\alpha_{1}}}^{-q_{2}(\omega_{0})}\vartheta_{\varpi_{\alpha_{2}}}^{q_{1}(\omega_{0})} \Big \rangle .
\end{equation}
If we consider $\omega_{0} \in c_{1}({\mathbbm{P}}(T_{{\mathbbm{P}^{2}}}))$ as in Example \ref{computationWallach}, we obtain $q_{1}(\omega_{0})=q_{2}(\omega_{0})=1$. Thus, it follows that
\begin{equation}
{\rm{H}}_{\omega_{0}} = \Big \{ \vartheta_{\varpi_{\alpha_{1}}}^{-n}\vartheta_{\varpi_{\alpha_{2}}}^{n} \ \Big | \ n \in \mathbbm{Z}\Big \} = \Big \langle \vartheta_{\varpi_{\alpha_{1}}}^{-1}\vartheta_{\varpi_{\alpha_{2}}} \Big \rangle.
\end{equation}
In particular, we obtain
\begin{equation}
{\rm{Pic}}^{0}_{\omega_{0}}({\mathbbm{P}}(T_{{\mathbbm{P}^{2}}})) = \Big \{ \mathscr{O}_{\alpha_{1}}(-n)\otimes \mathscr{O}_{\alpha_{2}}(n) \ \Big | \ n \in \mathbbm{Z} \Big \}.
\end{equation}
The subgroups ${\rm{H}}_{\omega_{0}}$ and ${\rm{Pic}}^{0}_{\omega_{0}}({\mathbbm{P}}(T_{{\mathbbm{P}^{2}}}))$ described above complete the following commutative diagram: 
\begin{center}
\begin{tikzcd}&  1 \arrow[d]& 1 \arrow[d] & &  \\ 1 \arrow[r] &  {\rm{H}}_{\omega_{0}} \arrow[r,"\iota"] \arrow[d]& {\text{Hom}}(B,\mathbbm{C}^{\times}) \arrow[r,"\widehat{\mu}_{[\omega_{0}]}"] \arrow[d] & 12\mathbbm{Z} \arrow[r] & 1 \\
1 \arrow[r]  &  {\rm{Pic}}^{0}_{\omega_{0}}({\mathbbm{P}}(T_{{\mathbbm{P}^{2}}})) \arrow[r,"\iota"]  \arrow[d] & {\rm{Pic}}({\mathbbm{P}}(T_{{\mathbbm{P}^{2}}}))  \arrow[r,"\mu_{[\omega_{0}]}"]  \arrow[d] & 12\mathbbm{Z} \arrow[r] \arrow[u,equals] & 1 \\
&  1 & 1  & &  
\end{tikzcd}
\end{center}
Building on the above ideas, we can construct examples of $[\omega_{0}]$-polystable holomorphic vector bundles over ${\mathbbm{P}}(T_{{\mathbbm{P}^{2}}})$ in the following way: Given some ${\bf{F}}_{0} \in {\rm{Pic}}({\mathbbm{P}}(T_{{\mathbbm{P}^{2}}}))$, take a collection of distinct elements ${\bf{G}}_{1},\ldots,{\bf{G}}_{r} \in  {\rm{Pic}}^{0}_{\omega_{0}}({\mathbbm{P}}(T_{{\mathbbm{P}^{2}}}))$, and define
\begin{equation}
{\bf{E}}:= \bigoplus_{j = 1}^{r}\big ({\bf{F}}_{0}\otimes {\bf{G}}_{j}\big ).
\end{equation}
By construction, we have that ${\bf{E}}$ is $[\omega_{0}]$-polystable. Notice that $\mu_{[\omega_{0}]}({\bf{E}}) = \mu_{[\omega_{0}]}({\bf{F}}_{0})$. In summary, for every $m_{0} \in 12\mathbbm{Z}$ and every integer $r >0$, there exists a $[\omega_{0}]$-polystable holomorphic vector bundle ${\bf{E}}$ over ${\mathbbm{P}}(T_{{\mathbbm{P}^{2}}})$, such that $\rank({{\bf{E}}}) = r$ and $\mu_{[\omega_{0}]}({\bf{E}}) = m_{0}$.
\end{example}
From the ideas introduced in the above remark we can prove Theorem \ref{Kgroup}.
\begin{theorem}
Given an integral K\"{a}hler class $[\omega_{0}] \in \mathcal{K}(X_{P})$, then we have 
\begin{equation}
\label{splitting}
K_{0}(X_{P}) \cong SK_{0}(X_{P}) \oplus {\rm{Pic}}^{0}_{\omega_{0}}(X_{P}) \oplus \tau([\omega_{0}])\mathbbm{Z},
\end{equation}
such that 
\begin{enumerate}
\item $SK_{0}(X_{P}) := \ker \big ( \det \colon K_{0}(X_{P}) \to {\rm{Pic}}(X_{P})\big)$,
\item ${\rm{Pic}}^{0}_{\omega_{0}}(X_{P}) := \big \{ {\bf{E}} \in {\rm{Pic}}(X_{P}) \ \big | \ \deg_{\omega_{0}}({\bf{E}}) = 0 \big \}$,
\item $\tau([\omega_{0}]):= {\rm{gcd}}\big \{ \deg_{\omega_{0}}(\mathscr{O}_{\alpha}(1)) \ \big | \ \alpha \in \Delta \backslash I\big\}$,
\end{enumerate}
where $\mathscr{O}_{\alpha}(1)$, $\alpha \in \Delta \backslash I$, are the generators of ${\rm{Pic}}(X_{P})$. Moreover, the generators of ${\rm{Pic}}^{0}_{\omega_{0}}(X_{P})$ are completely determined by the Hodge-Riemann bilinear form ${\mathcal{Q}}_{\omega_{0}}$.
\end{theorem}

\begin{proof}
Given ${\bf{E}} \in {\rm{Pic}}(X_{P})$, we have ${\bf{E}} = \bigotimes_{\alpha \in \Delta \backslash I} \mathscr{O}_{\alpha}(n_{\alpha})$, such that $n_{\alpha} \in \mathbbm{Z}$, $\forall \alpha \in \Delta \backslash I$. Thus, if we define 
\begin{equation}
\mathcal{E}:= \bigoplus_{\alpha \in \Delta \backslash I} \mathscr{O}_{\alpha}(n_{\alpha}),
\end{equation}
it follows that $[\mathcal{E}] \in  K_{0}(X_{P})$ and $\det([\mathcal{E}]) = {\bf{E}}$. Hence, $\det \colon K_{0}(X_{P}) \to {\rm{Pic}}(X_{P})$ is a surjective homomorphism. Denoting $SK_{0}(X_{P}) := \ker \big ( \det \colon K_{0}(X_{P}) \to {\rm{Pic}}(X_{P})\big)$, we have the short exact sequence of abelian groups
\begin{center}
\begin{tikzcd}1 \arrow[r] &  SK_{0}(X_{P}) \arrow[r,"\iota"] & K_{0}(X_{P}) \arrow[r,"\det"]  & {\rm{Pic}}(X_{P}) \arrow[r] & 1 \end{tikzcd}
\end{center}
Once ${\rm{Pic}}(X_{P})$ is a free abelian group, the above short exact sequence of abelian groups splits, i.e.,
\begin{equation}
K_{0}(X_{P}) \cong SK_{0}(X_{P}) \oplus {\rm{Pic}}(X_{P}).
\end{equation}
From Remark \ref{primitivesplitting}, we also have the following short exact sequence of abelian groups
\begin{center}
\begin{tikzcd}1 \arrow[r] &  {\rm{Pic}}^{0}_{\omega_{0}}(X_{P}) \arrow[r,"\iota"] & {\rm{Pic}}(X_{P}) \arrow[r,"\deg_{\omega_{0}}"]  & \tau([\omega_{0}])\mathbbm{Z} \arrow[r] & 1 \end{tikzcd}
\end{center}
Since  $\tau([\omega_{0}])\mathbbm{Z} \subset \mathbbm{Z}$ is also a free abelian group, the above short exact sequence of abelian groups also splits, thus
\begin{equation}
{\rm{Pic}}(X_{P}) \cong {\rm{Pic}}^{0}_{\omega_{0}}(X_{P}) \oplus \tau([\omega_{0}])\mathbbm{Z}.
\end{equation}
Combining the above facts, we obtain 
\begin{equation}
K_{0}(X_{P}) \cong SK_{0}(X_{P}) \oplus {\rm{Pic}}(X_{P}) \cong SK_{0}(X_{P}) \oplus {\rm{Pic}}^{0}_{\omega_{0}}(X_{P}) \oplus \tau([\omega_{0}])\mathbbm{Z}.
\end{equation}
As we have seen, fixed some $\gamma \in \Delta \backslash I$, it follows that
\begin{equation}
{\rm{Pic}}^{0}_{\omega_{0}}(X_{P}) = \Big \langle \mathscr{O}_{\gamma}(-q_{\alpha}(\omega_{0})) \otimes \mathscr{O}_{\alpha}(q_{\gamma}(\omega_{0})) \ \Big | \  \alpha \in \Delta \backslash I, \alpha \neq \gamma \Big \rangle,
\end{equation}
such that $q_{\alpha}(\omega_{0})= \frac{\mathcal{Q}_{\omega_{0}}({\bf{\Omega}}_{\alpha},\omega_{0})}{\tau([\omega_{0}])}$, $\forall \alpha \in \Delta \backslash I$, i.e., the generators of ${\rm{Pic}}^{0}_{\omega_{0}}(X_{P})$ are completely determined by the Hodge-Riemann bilinear form ${\mathcal{Q}}_{\omega_{0}}$.
\end{proof}
\begin{remark}
In the setting of the above theorem, the decomposition provided in Eq. (\ref{splitting}) also can be obtained from the following split exact sequences of abelian groups:
\begin{center}
\begin{tikzcd}1 \arrow[r] &  \det^{-1}\big ( {\rm{Pic}}^{0}_{\omega_{0}}(X_{P})\big )  \arrow[r,"\iota"] & K_{0}(X_{P}) \arrow[r,"\deg_{\omega_{0}}\circ \det"]  & \tau([\omega_{0}])\mathbbm{Z} \arrow[r] & 1 \end{tikzcd}
\end{center}
and
\begin{center}
\begin{tikzcd}1 \arrow[r] &  SK_{0}(X_{P})  \arrow[r,"\iota"] & \det^{-1}\big ( {\rm{Pic}}^{0}_{\omega_{0}}(X_{P})\big )  \arrow[r," \det"]  & {\rm{Pic}}^{0}_{\omega_{0}}(X_{P}) \arrow[r] & 1 \end{tikzcd}
\end{center}
\end{remark}
Let us illustrate the results provided by the last theorems.
\begin{example}
As before, consider $X_{P} = {\mathbbm{P}}(T_{{\mathbbm{P}^{2}}})$. Given $\omega_{0} \in c_{1}({\mathbbm{P}}(T_{{\mathbbm{P}^{2}}}))$, such that $\omega_{0} = 2({\bf{\Omega}}_{\alpha_{1}} + {\bf{\Omega}}_{\alpha_{2}})$, it follows from the previous computations (see Example \ref{generatorspic_0}) that 
\begin{equation}
K_{0}({\mathbbm{P}}(T_{{\mathbbm{P}^{2}}})) \cong SK_{0}({\mathbbm{P}}(T_{{\mathbbm{P}^{2}}})) \oplus \Big \langle \mathscr{O}_{\alpha_{1}}(-1)\otimes \mathscr{O}_{\alpha_{2}}(1) \Big \rangle \oplus 12 \mathbbm{Z}.
\end{equation}
Let us consider now $[\omega_{0}] \in \mathcal{K}({\mathbbm{P}}(T_{{\mathbbm{P}^{2}}}))$, such that $\omega_{0} = 2{\bf{\Omega}}_{\alpha_{1}} + {\bf{\Omega}}_{\alpha_{2}}$. From Example \ref{generatorspic_0}, we obtain
\begin{equation}
\mathcal{Q}_{\omega_{0}}({\bf{\Omega}}_{\alpha_{1}},\omega_{0}) = 5 \ \ {\text{and}} \ \ \mathcal{Q}_{\omega_{0}}({\bf{\Omega}}_{\alpha_{1}},\omega_{0}) = 8.
\end{equation}
Therefore, we have $\tau([\omega_{0}]) = 1$, so
\begin{equation}
K_{0}({\mathbbm{P}}(T_{{\mathbbm{P}^{2}}})) \cong SK_{0}({\mathbbm{P}}(T_{{\mathbbm{P}^{2}}})) \oplus \Big \langle \mathscr{O}_{\alpha_{1}}(-8)\otimes \mathscr{O}_{\alpha_{2}}(5) \Big \rangle \oplus \mathbbm{Z}.
\end{equation}
\begin{figure}[H]
    \centering
    \subfloat[\centering Integral weight corresponding to $\omega_{0} = 2({\bf{\Omega}}_{\alpha_{1}} + {\bf{\Omega}}_{\alpha_{2}})$.]{{\includegraphics[width=6.5cm]{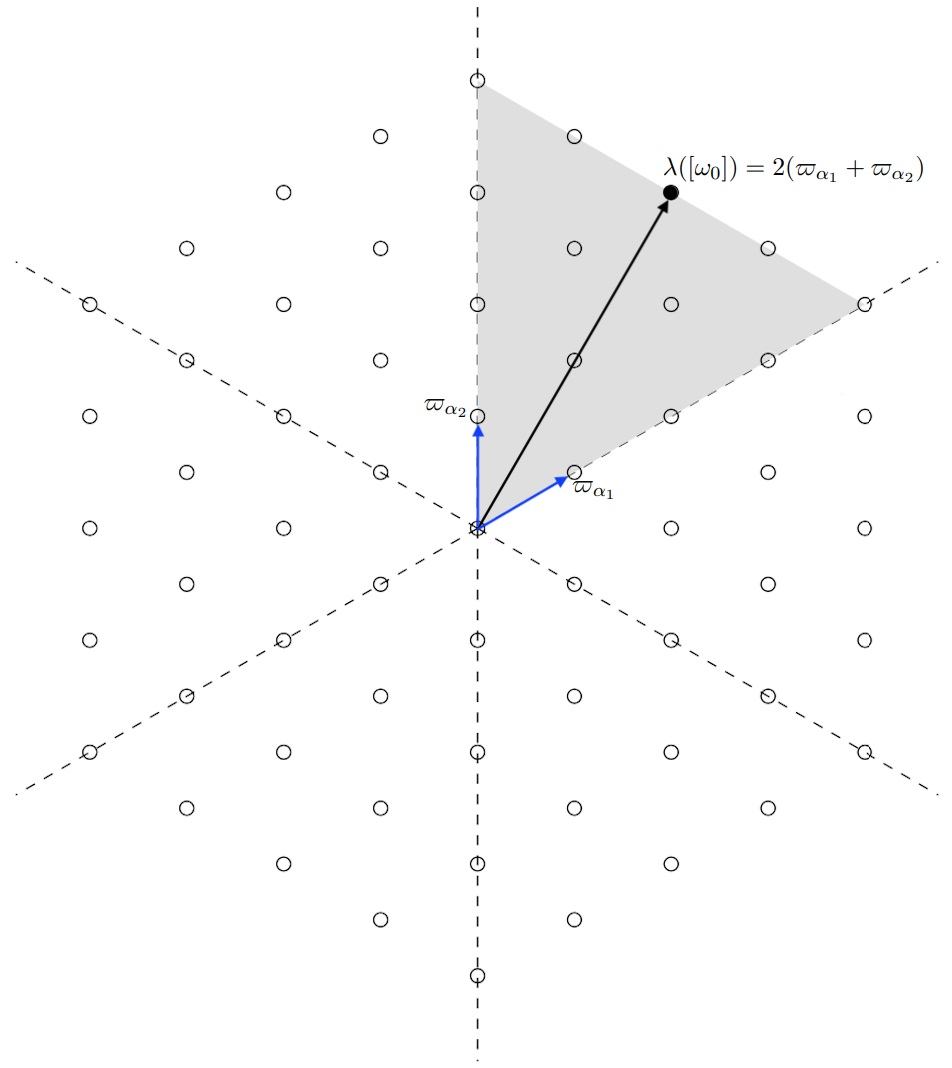} }}%
    \qquad
    \subfloat[\centering  Integral weight corresponding to $\omega_{0} = 2{\bf{\Omega}}_{\alpha_{1}} + {\bf{\Omega}}_{\alpha_{2}}$.]{{\includegraphics[width=6.5cm]{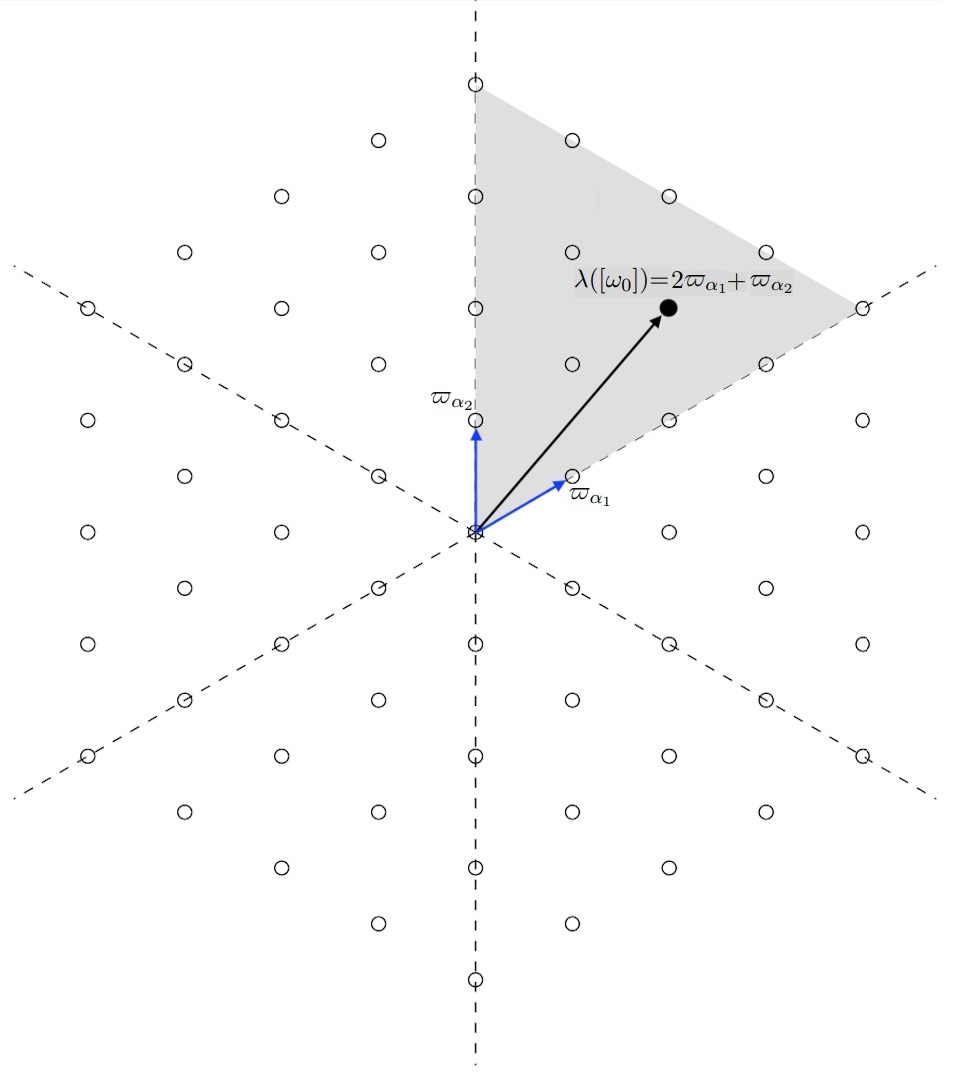} }}%
    \caption{Integral weights defined by integral K\"{a}hler classes.}%
    \label{fig:example}%
\end{figure}
Furthermore, from Eq. (\ref{naturaldensityexample}), it follows that:
\begin{enumerate}
\item $d(\mathcal{A}_{[\omega_{0}]}) = \frac{1}{12}$, if $\omega_{0} = 2({\bf{\Omega}}_{\alpha_{1}} + {\bf{\Omega}}_{\alpha_{2}})$;
\item $d(\mathcal{A}_{[\omega_{0}]}) = 1$, if $\omega_{0} = 2{\bf{\Omega}}_{\alpha_{1}} + {\bf{\Omega}}_{\alpha_{2}}$.
\end{enumerate}
As it can be seen, if we consider $\omega_{0} = 2{\bf{\Omega}}_{\alpha_{1}} + {\bf{\Omega}}_{\alpha_{2}}$, then the equation $\mu_{[\omega_{0}]}({\bf{E}}) = m_{0}$, ${\bf{E}} \in {\rm{Pic}}({\mathbbm{P}}(T_{{\mathbbm{P}^{2}}}))$, can be solved for every $m_{0} \in \mathbbm{Z}$. 
\end{example}





\bibliographystyle{alpha}
\bibliography{bibli}

\begin{thebibliography}{MMMS00}

\bibitem[AB03]{AZAD}
Hassan Azad and Indranil Biswas.
\newblock Quasi-potentials and {K}\"{a}hler-{E}instein metrics on flag
  manifolds. {II}.
\newblock {\em J. Algebra}, 269(2):480--491, 2003.

\bibitem[Akh95]{Akhiezer}
Dmitri~N. Akhiezer.
\newblock {\em Lie group actions in complex analysis}.
\newblock Aspects of Mathematics, E27. Friedr. Vieweg \& Sohn, Braunschweig,
  1995.

\bibitem[BGG73]{BernsteinGelfand}
I~N Bernstein, I~M Gel'fand, and S~I Gel'fand.
\newblock Schubert cells and cohomology of the spaces g/p.
\newblock {\em Russian Mathematical Surveys}, 28(3):1, jun 1973.

\bibitem[BH59]{BorelHizebruch}
A.~Borel and F.~Hirzebruch.
\newblock Characteristic classes and homogeneous spaces. {II}.
\newblock {\em Amer. J. Math.}, 81:315--382, 1959.

\bibitem[BMT11]{bayer2011bridgeland}
Arend Bayer, Emanuele Macr{\`\i}, and Yukinobu Toda.
\newblock Bridgeland stability conditions on threefolds i: Bogomolov-gieseker
  type inequalities.
\newblock {\em arXiv preprint arXiv:1103.5010}, 2011.

\bibitem[BR62]{BorelRemmert}
A.~Borel and R.~Remmert.
\newblock \"{U}ber kompakte homogene {K}\"{a}hlersche {M}annigfaltigkeiten.
\newblock {\em Math. Ann.}, 145:429--439, 1961/62.

\bibitem[Bri05]{Brion}
Michel Brion.
\newblock {\em Lectures on the Geometry of Flag Varieties}, pages 33--85.
\newblock Birkh{\"a}user Basel, Basel, 2005.

\bibitem[Bri07]{Bridgeland2007stability}
Tom Bridgeland.
\newblock Stability conditions on triangulated categories.
\newblock {\em Annals of Mathematics}, pages 317--345, 2007.

\bibitem[CCL21]{Chu2021space}
Jianchun Chu, Tristan~C Collins, and Man-Chun Lee.
\newblock The space of almost calibrated (1, 1)--forms on a compact k{\"a}hler
  manifold.
\newblock {\em Geometry \& Topology}, 25(5):2573--2619, 2021.

\bibitem[CG19]{CorreaGrama}
Eder~M. Correa and Lino Grama.
\newblock Calabi-{Y}au metrics on canonical bundles of complex flag manifolds.
\newblock {\em J. Algebra}, 527:109--135, 2019.

\bibitem[Che21]{Chen2021j}
Gao Chen.
\newblock The j-equation and the supercritical deformed hermitian--yang--mills
  equation.
\newblock {\em Inventiones mathematicae}, 225:529--602, 2021.

\bibitem[Chi12]{Chirka_complex}
Evgeni{\u\i}~Mikha{\u\i}lovich Chirka.
\newblock {\em Complex analytic sets}, volume~46.
\newblock Springer Science \& Business Media, 2012.

\bibitem[CJY20]{Collins2020}
Tristan~C Collins, Adam Jacob, and Shing-Tung Yau.
\newblock $(1, 1) $ forms with specified lagrangian phase: a priori estimates
  and algebraic obstructions.
\newblock {\em Cambridge Journal of Mathematics}, 8(2):407--452, 2020.

\bibitem[Cor19]{Correa}
Eder~M. Correa.
\newblock Homogeneous contact manifolds and resolutions of {C}alabi-{Y}au
  cones.
\newblock {\em Comm. Math. Phys.}, 367(3):1095--1151, 2019.

\bibitem[CS20]{Collins2020stability}
Tristan~C Collins and Yun Shi.
\newblock Stability and the deformed hermitian-yang-mills equation.
\newblock {\em arXiv preprint arXiv:2004.04831}, 2020.

\bibitem[CXY18]{Collins2018deformed}
T~Collins, Dan Xie, and Shing-Tung Yau.
\newblock The deformed hermitian--yang--mills equation in geometry and physics.
\newblock {\em Geometry and physics}, 1:69--90, 2018.

\bibitem[CY21]{Collins2021moment}
Tristan~C Collins and Shing-Tung Yau.
\newblock Moment maps, nonlinear pde and stability in mirror symmetry, i:
  geodesics.
\newblock {\em Annals of PDE}, 7(1):11, 2021.

\bibitem[Dem97]{Demailly_cplx}
Jean-Pierre Demailly.
\newblock {\em Complex analytic and differential geometry}.
\newblock Citeseer, 1997.

\bibitem[Don85]{donaldson1985anti}
Simon~K Donaldson.
\newblock Anti self-dual yang-mills connections over complex algebraic surfaces
  and stable vector bundles.
\newblock {\em Proceedings of the London Mathematical Society}, 3(1):1--26,
  1985.

\bibitem[Don87]{donaldson1987infinite}
S.~K. Donaldson.
\newblock {Infinite determinants, stable bundles and curvature}.
\newblock {\em Duke Mathematical Journal}, 54(1):231 -- 247, 1987.

\bibitem[Dou01a]{douglas2001b}
Michael~R Douglas.
\newblock D-branes, categories and n= 1 supersymmetry.
\newblock {\em Journal of Mathematical Physics}, 42(7):2818--2843, 2001.

\bibitem[Dou01b]{douglas2001a}
Michael~R Douglas.
\newblock D-branes on calabi-yau manifolds.
\newblock In {\em European Congress of Mathematics: Barcelona, July 10--14,
  2000 Volume II}, pages 449--466. Springer, 2001.

\bibitem[Dou02]{douglas2002c}
Michael~R Douglas.
\newblock Dirichlet branes, homological mirror symmetry, and stability.
\newblock {\em arXiv preprint math/0207021}, 2002.

\bibitem[Dua03]{duan2003degree}
Haibao Duan.
\newblock The degree of a schubert variety.
\newblock {\em Advances in Mathematics}, 180(1):112--133, 2003.

\bibitem[Fle84]{flenner1984restrictions}
Hubert Flenner.
\newblock Restrictions of semistable bundles on projective varieties.
\newblock {\em Commentarii Mathematici Helvetici}, 59:635--650, 1984.

\bibitem[FW01]{FultonWoodward}
William Fulton and Chris Woodward.
\newblock On the quantum product of schubert classes.
\newblock {\em Journal of Algebraic Geometry}, 13:641--661, 2001.

\bibitem[GH14]{Griffiths}
Phillip Griffiths and Joseph Harris.
\newblock {\em Principles of algebraic geometry}.
\newblock John Wiley \& Sons, 2014.

\bibitem[HL10]{huybrechts2010geometry}
Daniel Huybrechts and Manfred Lehn.
\newblock {\em The geometry of moduli spaces of sheaves}.
\newblock Cambridge University Press, 2010.

\bibitem[Hum72]{Humphreys}
James~E. Humphreys.
\newblock {\em Introduction to {L}ie algebras and representation theory}.
\newblock Graduate Texts in Mathematics, Vol. 9. Springer-Verlag, New
  York-Berlin, 1972.

\bibitem[Hum75]{HumphreysLAG}
James~E. Humphreys.
\newblock {\em Linear algebraic groups}.
\newblock Graduate Texts in Mathematics, No. 21. Springer-Verlag, New
  York-Heidelberg, 1975.

\bibitem[Huy05]{MR2093043}
Daniel Huybrechts.
\newblock {\em Complex geometry}.
\newblock Universitext. Springer-Verlag, Berlin, 2005.
\newblock An introduction.

\bibitem[JY17]{JacobYau2017}
Adam Jacob and Shing-Tung Yau.
\newblock A special lagrangian type equation for holomorphic line bundles.
\newblock {\em Mathematische Annalen}, 369:869--898, 2017.

\bibitem[Kob14]{kobayashi_dif_cplx_vec}
Shoshichi Kobayashi.
\newblock {\em Differential geometry of complex vector bundles}, volume 793.
\newblock Princeton University Press, 2014.

\bibitem[Lan04]{langer2004semistable}
Adrian Langer.
\newblock Semistable sheaves in positive characteristic.
\newblock {\em Annals of mathematics}, pages 251--276, 2004.

\bibitem[Laz17]{Lazarsfeld}
Robert~K Lazarsfeld.
\newblock {\em Positivity in algebraic geometry I: Classical setting: line
  bundles and linear series}, volume~48.
\newblock Springer, 2017.

\bibitem[LB18]{Flagvarieties}
V.~Lakshmibai and Justin Brown.
\newblock {\em Flag varieties}, volume~53 of {\em Texts and Readings in
  Mathematics}.
\newblock Hindustan Book Agency, Delhi, 2018.
\newblock An interplay of geometry, combinatorics, and representation theory,
  Second edition of [ MR2474907].

\bibitem[Lin22]{lin2022deformed}
Chao-Ming Lin.
\newblock The deformed hermitian--yang--mills equation, the positivstellensatz,
  and the solvability.
\newblock {\em arXiv preprint arXiv:2201.01438}, 2022.

\bibitem[LR08]{Lakshmibai2}
Venkatramani Lakshmibai and Komaranapuram~N. Raghavan.
\newblock {\em Standard monomial theory}, volume 137 of {\em Encyclopaedia of
  Mathematical Sciences}.
\newblock Springer-Verlag, Berlin, 2008.
\newblock Invariant theoretic approach, Invariant Theory and Algebraic
  Transformation Groups, 8.

\bibitem[LYZ00]{leung2000special}
Naichung~Conan Leung, Shing-Tung Yau, and Eric Zaslow.
\newblock From special lagrangian to hermitian-yang-mills via fourier-mukai
  transform.
\newblock {\em arXiv preprint math/0005118}, 2000.

\bibitem[Mat72]{MATSUSHIMA}
Yozo Matsushima.
\newblock Remarks on {K}\"{a}hler-{E}instein manifolds.
\newblock {\em Nagoya Math. J.}, 46:161--173, 1972.

\bibitem[MMMS00]{marino2000nonlinear}
Marcos Marino, Ruben Minasian, Gregory Moore, and Andrew Strominger.
\newblock Nonlinear instantons from supersymmetric p-branes.
\newblock {\em Journal of High Energy Physics}, 2000(01):005, 2000.

\bibitem[Mor69]{mordelldiophantine}
Louis~Joel Mordell.
\newblock {\em Diophantine equations}.
\newblock Academic press, 1969.

\bibitem[MR82]{mehta1982semistable}
Vikram~B Mehta and Annamalai Ramanathan.
\newblock Semistable sheaves on projective varieties and their restriction to
  curves.
\newblock {\em Mathematische Annalen}, 258(3):213--224, 1982.

\bibitem[MR84]{mehta1984restriction}
Vikram~Bhagvandas Mehta and Annamalai Ramanathan.
\newblock Restriction of stable sheaves and representations of the fundamental
  group.
\newblock {\em Inventiones mathematicae}, 77(1):163--172, 1984.

\bibitem[MS17]{Macri2017lectures}
Emanuele Macr{\`\i} and Benjamin Schmidt.
\newblock Lectures on bridgeland stability.
\newblock {\em Moduli of Curves: CIMAT Guanajuato, Mexico 2016}, pages
  139--211, 2017.

\bibitem[Nat08]{nathanson2008elementary}
Melvyn~B Nathanson.
\newblock {\em Elementary methods in number theory}, volume 195.
\newblock Springer Science \& Business Media, 2008.

\bibitem[Pin19]{Pingali2019}
Vamsi~P Pingali.
\newblock A note on the deformed hermitian yang-mills pde.
\newblock {\em Complex Variables and Elliptic Equations}, 64(3):503--518, 2019.

\bibitem[Pop74]{Popov}
V~L Popov.
\newblock Picard groups of homogeneous spaces of linear algebraic groups and
  one-dimensional homogeneous vector bundles.
\newblock {\em Mathematics of the USSR-Izvestiya}, 8(2):301, apr 1974.

\bibitem[Pro07]{procesi2007lie}
Claudio Procesi.
\newblock {\em Lie groups: an approach through invariants and representations},
  volume 115.
\newblock Springer, 2007.

\bibitem[PS08]{peters2008mixed}
Chris~AM Peters and Joseph~HM Steenbrink.
\newblock {\em Mixed hodge structures}, volume~52.
\newblock Springer Science \& Business Media, 2008.

\bibitem[Sno89]{Snowhomovec}
Dennis Snow.
\newblock Homogeneous vector bundles.
\newblock In {\em Group Actions and Invariant Theory (Montreal, PQ, 1988), CMS
  Conf. Proc}, volume~10, pages 193--205, 1989.

\bibitem[Sol13]{Solomon2013calabi}
Jake~P Solomon.
\newblock The calabi homomorphism, lagrangian paths and special lagrangians.
\newblock {\em Mathematische Annalen}, 357:1389--1424, 2013.

\bibitem[Sol14]{Solomon2014curvature}
Jake~P Solomon.
\newblock Curvature of the space of positive lagrangians.
\newblock {\em Geometric and Functional Analysis}, 24:670--689, 2014.

\bibitem[Sz{\'e}14]{szekelyhidi2014introduction}
G{\'a}bor Sz{\'e}kelyhidi.
\newblock {\em An Introduction to Extremal Kahler Metrics}, volume 152.
\newblock American Mathematical Soc., 2014.

\bibitem[Tak78]{MR528871}
Masaru Takeuchi.
\newblock Homogeneous {K}\"{a}hler submanifolds in complex projective spaces.
\newblock {\em Japan. J. Math. (N.S.)}, 4(1):171--219, 1978.

\bibitem[Tho06]{thomas2002stability}
R.~P. Thomas.
\newblock Stability conditions and the braid group.
\newblock {\em Communications in Analysis and Geometry}, 14(1):135--161, 2006.

\bibitem[Tim11]{Timashev}
Dmitry~A. Timashev.
\newblock {\em Homogeneous spaces and equivariant embeddings}, volume 138 of
  {\em Encyclopaedia of Mathematical Sciences}.
\newblock Springer, Heidelberg, 2011.
\newblock Invariant Theory and Algebraic Transformation Groups, 8.

\bibitem[UY86]{uhlenbeck1986existence}
Karen Uhlenbeck and Shing-Tung Yau.
\newblock On the existence of hermitian-yang-mills connections in stable vector
  bundles.
\newblock {\em Communications on Pure and Applied Mathematics},
  39(S1):S257--S293, 1986.

\bibitem[Voi07]{VoisinBook1}
Claire Voisin.
\newblock {\em Hodge theory and complex algebraic geometry. {I}}, volume~76 of
  {\em Cambridge Studies in Advanced Mathematics}.
\newblock Cambridge University Press, Cambridge, english edition, 2007.
\newblock Translated from the French by Leila Schneps.

\end{thebibliography}

\end{document}